\documentclass[12pt]{amsart}
\usepackage[hmargin=1in,vmargin=1in]{geometry}
\usepackage{microtype}

\usepackage[colorlinks=true]{hyperref}

\usepackage{enumitem}
\usepackage{graphicx,wrapfig,overpic,tikz-cd,color,cancel}
\usetikzlibrary{positioning}
\usetikzlibrary{decorations.pathmorphing}

\usepackage{amsmath, amssymb}
\usepackage[mathscr]{euscript}

\usepackage{tabularx, booktabs}
\newcolumntype{C}{>{\centering\arraybackslash}X}

\usepackage{amsthm,chngcntr,verbatim}
\usepackage[noabbrev]{cleveref}

\numberwithin{equation}{section}

\makeatletter
\let\c@figure\c@equation
\let\c@table\c@equation
\makeatother

\counterwithout{footnote}{section}

\newtheorem{theorem}[equation]{Theorem}
\newtheorem{lemma}[equation]{Lemma}
\newtheorem{proposition}[equation]{Proposition}
\newtheorem{conjecture}[equation]{Conjecture}
\newtheorem{corollary}[equation]{Corollary}

\theoremstyle{definition}
\newtheorem{definition}[equation]{Definition}
\newtheorem{example}[equation]{Example}
\newtheorem{remark}[equation]{Remark}
\newtheorem{notation}[equation]{Notation}

\newcommand{\taglink}[1]{\href{http://stacks.math.columbia.edu/tag/#1}{#1}}
\renewcommand{\tag}[1]{\cite[{\taglink{#1}}]{stacks}}

\newcommand{\C}{\mathbb{C}}

\newcommand{\Z}{\mathbb{Z}}
\newcommand{\Q}{\mathbb{Q}}
\renewcommand{\P}{\mathbb{P}}

\newcommand{\D}{\mathbb{D}}
\newcommand{\N}{\mathbb{N}}

\renewcommand{\O}{\mathscr{O}}
\newcommand{\scrD}{\mathscr{D}}
\newcommand{\m}{\mathfrak{m}}
\newcommand{\Robba}{\mathscr{R}}

\newcommand{\an}{\mathrm{an}}
\newcommand{\alg}{\mathrm{alg}}

\newcommand{\dR}{\mathrm{dR}}
\newcommand{\rig}{\mathrm{rig}}

\newcommand{\frX}{\mathfrak{X}}

\newcommand{\GL}{\mathrm{GL}}

\newcommand{\rightderived}{\mathrm{R}}

\DeclareMathOperator{\codim}{codim}

\DeclareMathOperator*\colim{colim}

\DeclareMathOperator{\im}{\textnormal{im}}
\DeclareMathOperator{\coker}{\textnormal{coker}}
\DeclareMathOperator{\tr}{\textnormal{tr}}
\DeclareMathOperator{\rank}{\textnormal{rank}}

\DeclareMathOperator{\Hom}{\textnormal{Hom}}
\DeclareMathOperator{\Ext}{\textnormal{Ext}}
\DeclareMathOperator{\Der}{\textnormal{Der}}
\DeclareMathOperator{\PDer}{\textnormal{PDer}}
\DeclareMathOperator{\End}{\textnormal{End}}
\DeclareMathOperator{\Aut}{\textnormal{Aut}}
\DeclareMathOperator{\Inf}{\textnormal{Inf}}
\DeclareMathOperator{\Lift}{\textnormal{Lift}}
\DeclareMathOperator{\MC}{\textnormal{MC}}
\DeclareMathOperator{\Sym}{\textnormal{Sym}}
\DeclareMathOperator{\Def}{\textnormal{Def}}

\DeclareMathOperator{\Frac}{\textnormal{Frac}}
\DeclareMathOperator{\Stab}{\textnormal{Stab}}
\DeclareMathOperator{\Cone}{\textnormal{Cone}}
\DeclareMathOperator{\Hoch}{\textnormal{Hoch}}

\newcommand{\tube}[1]{]#1[}

\newcommand{\Art}{\textnormal{\textsc{Art}}}
\newcommand{\Set}{\textnormal{\textsc{Set}}}
\newcommand{\scF}{\textnormal{\textsc{F}}}
\newcommand{\Isoc}{\textnormal{\textsc{Isoc}}}
\newcommand{\Loc}{\textnormal{\textsc{Loc}}}
\newcommand{\Grp}{\textnormal{\textsc{Grp}}}
\newcommand{\Mod}{\textnormal{\textsc{Mod}}}
\newcommand{\DMod}{\textnormal{\textsc{DMod}}}

\newlength{\perspective}

\title{Deformations of overconvergent isocrystals on the projective line}
\author{Shishir Agrawal}

\begin{document}
	
	\begin{abstract}
		Let $k$ be a perfect field of positive characteristic and $Z$ an effective Cartier divisor in the projective line over $k$ with complement $U$. In this note, we establish some results about the formal deformation theory of overconvergent isocrystals on $U$ with fixed ``local monodromy'' along $Z$. En route, we show that a Hochschild cochain complex governs deformations of a module over an arbitrary associative algebra. We also relate this Hochschild cochain complex to a de Rham complex in order to understand the deformation theory of a differential module over a differential ring. 
	\end{abstract}
	
	\maketitle
	
	\setcounter{tocdepth}{1}
	\tableofcontents
	
	\section{Introduction}

\subsection{Rigid local systems} \label{section:rigid local systems}

Let $Z$ be an effective Cartier divisor on the complex projective line $\P^1$ and $U = \P^1 \setminus Z$. Choosing an appropriate numbering $Z = \{z_1,  \dotsc, z_m\}$ for the puncture points, a \emph{local system} of rank $n$ on $U$ is any of the following types of objects:
\begin{enumerate}[label=(LS\arabic*),leftmargin=12mm]
	\item\label{enumitem:local system as sheaf} A locally constant sheaf of complex vector spaces on $U^\an$ of rank $n$. 
	\item\label{enumitem:local system as representation} An $n$-dimensional representation of the fundamental group $\pi_1(U^\an, *)$, where $*$ is a basepoint in $U^\an$. 
	\item\label{enumitem:local system as matrices} A list $A_1, \dotsc, A_m$ of invertible $n \times n$ matrices such that \[  A_1 \dotsb A_m = 1. \]
\end{enumerate}
If a rank $n$ local system $L$ on $U$ corresponds to the list of matrices $A_1, \dotsc, A_m$, the \emph{local monodromy} of $L$ at $z_i$ is the conjugacy class of $A_i$ (that is, ``the'' Jordan form of $A_i$).

\begin{definition}[Physical rigidity] \label{definition:physically rigid}
	A local system $L$ on $U$ is \emph{physically rigid} if it is determined up to isomorphism by its local monodromies: that is, whenever there exists a local system $L'$ on $U$ such that the local monodromies of $L$ and $L'$ along $Z$ are equal, then $L$ and $L'$ are isomorphic.
\end{definition}

More concretely, the local system corresponding to a list of matrices $A_1, \dotsc, A_m$ is physically rigid if and only if, whenever there is a list of matrices $A_1', \dotsc, A_m'$ such that $A_1' \dotsb A_m' = 1$ and $A_i$ is conjugate to $A_i'$ for each $i$ individually, then in fact the two lists $A_1, \dotsc, A_m$ and $A_1', \dotsc, A_m'$ are simultaneously conjugate. 

Aside from the trivial cases $n = 1$ or $m \leq 2$, it is difficult to check physical rigidity directly from the definition. It turns out, however, that there is a simple characterization of \emph{irreducible} physically rigid local systems. 

\begin{theorem}[Katz's cohomological criterion for rigidity] \label{result:katz cohomological criterion}
	Suppose $L$ is an irreducible local system on $U$. Then $L$ is physically rigid if and only if $H^1(\P^{1,\an}, j_{!+}\End(L)) = 0$, where $j$ denotes the open embedding $U^\an \hookrightarrow \P^{1,\an}$.\footnotemark\ 
\end{theorem}

\footnotetext{Note that, for a local system $L$ on $U$, we have $H^1(\P^{1,\an}, j_{!+}L) = H^1(\P^{1,\an}, j_*L)$, where $j_*$ denotes the usual underived direct image functor.}

It follows from the Euler-Poincar\'e formula that, if a local system $L$ of rank $n$ on $U$ corresponds to the list of matrices $A_1, \dotsc, A_m$ in $\GL_n(\C)$, then
\begin{equation} \label{eqn:euler poincare dimension}
\dim H^1(\P^{1,\an}, j_{!+}\End(L)) = 2(1-n^2) + \sum_{i = 1}^m \codim \mathfrak{z}(A_i)
\end{equation}
where $\mathfrak{z}(A_i)$ denotes the centralizer of $A_i$, regarded as a subspace of the $n^2$-dimensional vector space $\mathfrak{gl}_n$. In this way, \cref{result:katz cohomological criterion} makes it very easy to check rigidity of irreducible local systems. 

\begin{example} \label{result:irreducible rank 2}
	We have $\codim(\mathfrak{z}(A)) = 2$ for any non-scalar $A \in \GL_2(\C)$. Thus an irreducible local system of rank 2 on $U$, all of whose local monodromies are non-scalar, is physically rigid if and only if we have precisely $m = 3$ punctures. 
\end{example}

To explain where Katz's cohomological criterion for rigidity comes from, let us first give a moduli theoretic reinterpretation of physical rigidity. 

\subsection{Deligne-Simpson map}

Consider the functor $H_n$ which sends a commutative $\C$-algebra $R$ to the set
\[ \Hom_{\textsc{Grp}}(\pi_1(U^\an, *), \GL_n(R)). \]
For each $i = 1, \dotsc, m$, we have a ``image of $\gamma_i$'' morphism $\pi_i : H_n \to \GL_n$. The product $(\pi_1, \dotsc, \pi_m)$ defines a closed embedding of $H_n$ into the $m$-fold product $\GL_n \times \dotsb \times \GL_n$, with image the closed subscheme
\[ R \mapsto \{ A_1, \dotsc, A_m \in \GL_n(R) \; | \; A_1  \dotsb A_m = 1 \}. \]
Thus $H_n$ is an affine scheme.\footnote{Note that the projection $(\pi_2, \dotsc, \pi_m)$ defines an isomorphism of $H_n$ onto the $(m-1)$-fold product $\GL_n \times \dotsb \times \GL_n$. In particular, $H_n$ is smooth.}
Observe that $\GL_n$ acts on $H_n$ on the left by conjugation. The quotient stack
\[ \Loc_n := [H_n/\GL_n] \]
is the moduli of local systems of rank $n$ on $U$.\footnotemark\ Taking an infinite coproduct, we obtain the algebraic stack
\[ \Loc := \coprod_{n \geq 1} \Loc_n \]
which parametrizes all local systems on $U$.  

\footnotetext{The algebraic stack $\Loc_n$ is quasi-separated of finite type over $\C$, because $H_n$ and $\GL_n$ are both affine of finite type over $\C$.  
In fact, it is even smooth, since $H_n$ is smooth. 
}

Now $\GL_n$ also acts on itself by conjugation, and the ``image of $\gamma_i$'' map $\pi_i : H_n \to \GL_n$ is $\GL_n$-equivariant. It therefore induces a morphism 
\[ \begin{tikzcd} \Loc_n \ar{r} & J_n := [\GL_n/\GL_n] \end{tikzcd} \]
which we abusively denote $\pi_i$ again. The tuple $(\pi_1, \dotsc, \pi_m)$ defines a morphism 
\begin{equation} \label{eqn:deligne simpson map fixed rank} \begin{tikzcd} \Loc_n \ar{r} & (J_n)^m. \end{tikzcd} \end{equation}
Now let $J := \coprod_{n \geq 1} J_n$ and take the infinite disjoint union of these maps as $n$ varies. 

\begin{definition}
	The \emph{Deligne-Simpson map} is the morphism
	\[ \begin{tikzcd} \Loc \ar{r}{\pi} & J^m \end{tikzcd} \]
	defined by taking the infinite disjoint union of the map \eqref{eqn:deligne simpson map fixed rank} over all $n \geq 1$. If $L$ is a local system of rank $n$ on $U$, then $\pi(L)$ is the local monodromy data of $L$ along $Z$. 
\end{definition}

When $M$ is a finite type point of $J^m$, we let $\Gamma_M$ be the residual gerbe of $J$ at $M$ and we define $\Loc^M := \pi^{-1}(\Gamma_M)$. 
\begin{equation} \label{eqn:locm cartesian diagram} \begin{tikzcd} \Loc^M \ar[hookrightarrow]{r} \ar{d} & \Loc \ar{d}{\pi} \\ \Gamma_M \ar[hookrightarrow]{r} & J^m \end{tikzcd}
\end{equation}
This is a quasi-separated algebraic stack of finite type over $\C$.\footnotemark\ Roughly speaking, $M$ is an $m$-tuple of conjugacy classes invertible matrices over $\C$, and $\Loc^M$ is the moduli of local systems on $U$ whose local monodromy along $Z$ is $M$. 

\footnotetext{Observe that $\Loc^M$ is empty unless $M$ is a point of $(J_n)^m$ for some $n$, in which case the finite type monomorphism $\Loc^M \hookrightarrow \Loc$ factors through $\Loc_n$, which is quasi-separated of finite type over $\C$. 
}

\begin{theorem} \label{result:tangent of trivializable}
	Let $L$ be a local system on $U$ and $M := \pi(L)$ its local monodromy data. Then we have
	\[ \Inf_L(\Loc^M) = \End_{\C}(L) \text{ and } T_L(\Loc^M) = H^1(\P^{1,\an}, j_{!+}\End(L)), \]
	where $j$ denotes the open embedding $U^\an \hookrightarrow \P^{1,\an}$, and where $\Inf_L(\Loc^M)$ and $T_L(\Loc^M)$ respectively denote the space of infinitesimal automorphisms and the tangent space of the stack $\Loc^M$ at the point $L$.\footnotemark\ 
	
	Moreover, the tangent space $H^1(\P^{1,\an}, j_{!+}\End(L))$ carries a natural symplectic form. Finally, if $L$ is irreducible, then $\Loc^{M}$ is smooth at $L$. 
\end{theorem}

\footnotetext{In the notation of \cite[\taglink{07WY}]{stacks}, we have $\Inf_L(\Loc^M) = \Inf(F_{\Loc^M, \C, L})$ and $T_L(\Loc^M) = TF_{\Loc^M, \C, L}$. }

One can prove \cref{result:tangent of trivializable} roughly as follows. We produce an exact sequence \tag{07X2} using the cartesian diagram \eqref{eqn:locm cartesian diagram}. This exact sequence identifies $\Inf_L(\Loc^M)$ with $\End_{\C}(L)$ and $T_L(\Loc^M)$ with $H^1(\P^{1,\an}, j_{!+}\End(L))$.\footnote{Cf. \cite[theorem 4.10]{bloch}, \cite[proposition 2.3 and remark 2.4]{esnault-groechenig}, and \cref{result:main part 1} below.} 
The trace pairing on $\End(L)$ induces a symplectic form on this vector space.\footnote{Cf. \cite[page 3]{katz} and \cref{result:main part 1} below. }
Obstructions to smoothness live in $H^2_c(U^\an, \End(L))$, and the obstruction classes vanish when $L$ is irreducible.\footnote{Cf. \cite[theorem 4.10]{bloch} and \cref{result:main part 2} below.}

\subsection{Sketch of a proof of Katz's cohomological criterion} \label{katz-proof-sketch}

Let $L$ be an irreducible local system on $U$ and $M = \pi(L)$ its monodromy data. The proof that $H^1(\P^{1,\an}, j_{!+}\End(L)) = 0$ implies physical rigidity is an application of the Euler-Poincar\'e formula; see the first part of the proof of \cite[theorem 1.1.2]{katz} for details. 

Here, let us focus on the converse; let us give a slightly different argument than that given in \cite[theorem 1.1.2]{katz}. More specifically, we will see that the fact that physical rigidity implies $H^1(\P^{1,\an}, j_{!+}\End(L)) = 0$ is a relatively formal consequence of the infinitesimal structure of $\Loc^M$ near $L$ as described by \cref{result:tangent of trivializable}, plus the fact that $\Loc^M$ is a quasi-separated algebraic stack of finite type over $\C$.   

Notice that
\[ \Inf_L(\Loc^M) = \End_{\C}(L) = \C \]
since $L$ is irreducible. We know from \cref{result:tangent of trivializable} that $\Loc^M$ is smooth at $L$, so $\dim_L(\Loc^M)$ coincides with the Euler characteristic of the tangent complex at $L$. In other words, we have
\[ \begin{aligned} 
\dim_L(\Loc^M) &= -\dim \Inf_L(\Loc^M) + \dim T_L(\Loc^M) \\
&= -1 + \dim H^1(\P^{1,\an}, j_{!+}\End(L)). \end{aligned} \]
Since $H^1(\P^{1,\an}, j_{!+}\End(L))$ carries a symplectic form, its dimension must be even. We conclude that $H^1(\P^{1,\an}, j_{!+}\End(L)) = 0$ if and only if $\dim_L(\Loc^M) \leq 0$. 

By definition, $L$ is physically rigid if and only if $L$ is the unique finite type point of $\Loc^M$, but this implies that $\dim_L(\Loc^M) \leq0$ (cf. \cref{result:unique ftpt,result:isolated implies nonpositive dimension}).  

\subsection{Simpson's conjecture}

\Cref{result:irreducible rank 2} suggests that the physically rigid local systems form a somewhat sparse class of local systems. Despite this, interest in this class of local systems stems in part from the following result. 

\begin{theorem}[{\cite[theorem 8.4.1]{katz}}] \label{result:katz motivicity}
	Suppose $L$ is an irreducible local system on $U$ such that $\pi(L)$ is quasi-unipotent. If $L$ is physically rigid, then it is motivic; in other words, there exists dense open subset $V \subseteq U$, a morphism of algebraic varieties $f : T \to V$ over $\C$, and a non-negative integer $i$, such that $L|_V$ is a subquotient of the higher direct image sheaf $R^i f_*\underline{\C}_{T^\an}$.
\end{theorem}

\begin{example}
	Let $U = \P^1 \setminus \{0, 1, \infty\}$ and consider the local system on $U$ whose sections are local solutions on $U^\an$ to the hypergeometric differential equation
	\[ z(z-1)f'' + (2z-1)f' + \frac{f}{4} = 0. \]
	It turns out that, up to simultaneous conjugation,
	\[
	A_0 = \begin{pmatrix} 2 & 1 \\ -1 & 0 \end{pmatrix}, \quad
	A_1 = \begin{pmatrix} 1 & 0 \\ -3 &  1 \end{pmatrix}, \text{ and }
	A_\infty = \begin{pmatrix} 0 & -1 \\ 1 & -2 \end{pmatrix}\,.  \]
	All of these matrices are conjugate to Jordan blocks of size 2; the first two have eigenvalue 1, and the third has eigenvalue $-1$. The only nontrivial subspace that is invariant under $A_0$ is the one spanned by the eigenvector $(-1, 1)$. This is not stable under either $A_1$ or $A_\infty$, so $L$ is irreducible. It follows from \cref{result:irreducible rank 2} that $L$ is physically rigid, so Katz's motivicity \cref{result:katz motivicity} guarantees that $L$ is motivic. 
	
	In fact, this is recovering a classical result. Let $f : E \to U$ be the Legendre family of elliptic curves: the fiber above any closed point $u \in U$ is the projective closure of the affine curve on $U$ defined by the Weierstrass equation
	\[ y^2 = x(x-1)(x-u).  \]
	Then $L \simeq R^1f_*\underline{\C}_{E^\an}$. 
\end{example}

Suppose now that $X$ is any smooth, connected, and projective scheme over $\C$. Let $Z$ be a strict normal crossings divisor in $X$ with irreducible components $Z_1, \dotsc, Z_m$ and let $U = X \setminus Z$. We can then form the algebraic stack $\Loc$ of local systems on $U$. Fix a rank 1 local system $D$ and let $\Loc_D$ be the moduli of local systems $L$ on $U$ equipped with an isomorphism $D \simeq \det(L)$. As before, there is a monodromy map 
\[ \begin{tikzcd} \Loc_D \ar{r}{\pi} & J^m. \end{tikzcd} \]
Choosing a finite type point $M$ of $J^m$, the inverse image of the residual gerbe of $J^m$ at $M$ defines a substack $\Loc_D^M \subseteq \Loc_D$. 

\begin{conjecture}[Simpson] \label{simpsons conjecture}
	Suppose $L$ is an irreducible local system on $U$ such that $D := \det(L)$ is torsion and $M := \pi(L)$ is quasi-unipotent. If $L$ is isolated in $\Loc_D^M$, then it is motivic. 
\end{conjecture}

Katz's motivicity theorem \ref{result:katz motivicity} is precisely the $X = \P^1$ case of Simpson's conjecture; it was not necessary to fix determinants in this case roughly because $\P^{1,\an}$ is simply connected. Some evidence towards the general conjecture is provided in \cite{esnault-groechenig}, where it is proved that, under the same hypotheses as \cref{simpsons conjecture}, $L$ must be integral.\footnote{A local system $L$ being integral means that $L$ comes from base changing a local system of finite projective modules over the ring of integers in a number field.} 

\subsection{Positive characteristic}

Suppose now that $k$ is an algebraically closed field of characteristic $p > 0$. Let $\P^1$ be the projective line over $k$ and let $Z = \{z_1, \dotsc, z_m\}$ be an effective Cartier divisor in $\P^1$ with complement $U = \P^1 \setminus Z$. There are several kinds of objects living on $U$ that can rightfully be called analogs of local systems, indexed by an auxiliary prime $\ell$.  

For any $\ell \neq p$, we can consider \emph{$\ell$-adic local systems} on $U$, by which we mean continuous finite dimensional representations of the \'etale fundamental group $\pi_1^{\text{\'et}}(U,*)$ over $\bar{\Q}_\ell$, where $*$ is a fixed basepoint in $U$. For every $z \in Z$, there is a natural group homomorphism
\[ \begin{tikzcd} G_z := \textnormal{Gal}(\Frac \hat{\O}_{\P^1, z}) \ar{r} & \pi_1^{\text{\'et}}(U, *). \end{tikzcd} \]
If $L$ is a $\ell$-adic local system on $U$, the isomorphism class of the continuous representation of $G_z$ obtained by restricting along the above homomorphism plays the role of the \emph{local monodromy} of $L$ around $z$. This lets us formulate a notion of \emph{physical rigidity} for $\ell$-adic local systems on $U$. Katz proved that physical rigidity of an irreducible $\ell$-adic local system $L$ is implied by the vanishing of $H^1(\P^1, j_{!+}\End(E))$ \cite[theorem 5.0.2]{katz}. The role of the Euler-Poincar\'e formula in this situation is played by the Grothendieck-Ogg-Shafarevich formula. More recent work of Fu proves that the converse implication is true as well \cite{fu}.

When $\ell = p$, there are ``too many'' continuous finite dimensional representations of $\pi_1^{\text{\'et}}(U, *)$ over $\bar{\Q}_p$, because the image in $\GL_n(\bar{\Q}_p)$ of the ``pro-$p$ part'' of $\pi_1^{\text{\'et}}(U, *)$ can be quite large. By requiring that the image of the ``pro-$p$ part'' is not too large, we obtain a slightly better behaved category, and we can ``unsolve'' these representations to obtain certain kinds of modules with integrable connections. More precisely, a theorem of Tsuzuki's tells us that continuous finite dimensional representations of $\pi_1^{\text{\'et}}(U, *)$ with finite local monodromy are equivalent to overconvergent isocrystals on $U$ equipped with unit-root Frobenius structures \cite{tsuzuki}. But now, this category turns out to be too small to contain all objects that can ``come from geometry over $U$.'' 
	
A better category seems to be that of \emph{all} overconvergent isocrystals on $U$. It seems to contain objects that ``come from geometry over $U$'' \cite[th\'eor\`eme 5]{berthelot_rigide}, though there remain some open questions about its suitability as an analog of local systems \cite{lazda_incarnations}. Hereafter, we will refer to overconvergent isocrystals simply as \emph{isocrystals}. We will not insist that our isocrystals come equipped with Frobenius structures, though at various points it will be necessary to assume that some isocrystals in question \emph{can} be equipped with a Frobenius structure (or to assume some more technical hypothesis that is implied by the existence of a Frobenius structure) in order to use some finite dimensionality results. In any case, the isocrystals that ``come from geometry over $U$'' come naturally equipped with Frobenius structures, so this will not be too serious an assumption when it does come up. 

\subsection{Rigid isocrystals}

Let $K_0$ be the fraction field of the Witt vectors $W(k)$ and $\bar{K}$ an algebraic closure of $K_0$. For a finite extension $K$ of $K_0$ inside $\bar{K}$, let $\Isoc^\dagger(U/K)$ denote the category of overconvergent isocrystals on $U$ over the spectrum of the ring of integers in $K$ \cite[d\'efinition 2.3.6]{berthelot_rigide2}. Taking a colimit over all such $K$, we obtain
\[ \Isoc^\dagger(U/\bar{K}) := \colim_K \Isoc^\dagger(U/K). \]
The analog of the local monodromy at a point $z \in Z$ of an $E \in \Isoc^\dagger(U/K)$ is played by the isomorphism class of the \emph{Robba fiber} $E_z$ of $E$ at $z$ (where the Robba fiber $E_z$ is a differential module over the Robba ring $\mathscr{R}_z$, as in \cite[definition 3.3]{lestum}). We can thus formulate a notion of physical rigidity for isocrystals: we say that $E \in \Isoc^\dagger(U/K)$ is \emph{physically rigid} if it is determined up to isomorphism as an object of $\Isoc^\dagger(U/\bar{K})$ by the isomorphism classes of its Robba fibers along $Z$.\footnotemark\ 

\footnotetext{It is worth drawing attention to the fact that we are requiring that $E$ be determined up to isomorphism as an object of $\Isoc^\dagger(U/\bar{K})$ by its Robba fibers, not as an object of $\Isoc^\dagger(U/K)$. This condition is therefore slightly stronger than the notion formulated by Crew \cite{crew_rigidity} in that we have forcibly stabilized Crew's definition under finite extensions of $K$. Since the goal we have in mind is a $p$-adic analog of Katz's cohomological criterion and since the condition that the relevant cohomology space $H^1_{p,\rig}(U, \End(E))$ vanish is invariant under finite extensions of $K$, this notion seemed better suited for our purposes. }

Crew proves that an absolutely irreducible $E \in  \Isoc^\dagger(U/K)$ is physically rigid if a certain ``parabolic cohomology'' space $H^1_{p,\rig}(U, \End(E))$ vanishes \cite[theorem 1]{crew_rigidity}. The proof is similar to the ones given by Katz in the complex (resp. $\ell$-adic) settings; the role of the Euler-Poincar\'e formula (resp. Grothendieck-Ogg-Shafarevich formula) is played by the Christol-Mebkhout index formula \cite[th\'eor\`eme 8.4--1]{christol_mebkhout3}. To strengthen the analogy of Crew's result with its analogs in the complex and $\ell$-adic settings, we prove in \cref{result:parabolic and middle} below that Crew's parabolic cohomology $H^1_{p,\rig}(U, \End(E))$ coincides with $H^1(\P^1, j_{!+}\End(E))$, where $j_{!+}$ is the intermediate extension operation in Berthehlot's theory of arithmetic D-modules. 

The contents of this note were motivated by a search for the converse implication: that physical rigidity of an absolutely irreducible $E \in \Isoc^\dagger(U/K)$ implies the vanishing of $H^1_{p,\rig}(U, \End(E))$. An approach like the one we described above in sections \ref{section:rigid local systems} through \ref{katz-proof-sketch} above cannot work verbatim: isocrystals on $U$ with prescribed Robba fibers along $Z$ seem not to be the finite type points of an algebraic stack. 

However, we can still consider the stack $S$ over $K_0$ whose $K$-points are modules with integrable connections on the \emph{tube}\footnotemark\ $\tube{U}$. 
\footnotetext{See \cite[equation (2.3)]{lestum} for a definition of tubes.}
Moreover, if we fix a tuple $M$ of isomorphism classes of differential modules over the Robba rings along $Z$, we can also form the substack $S^M \subseteq S$ whose $K$-points are modules with integrable connection on $\tube{U}$ whose Robba fibers are prescribed by $M$. These stacks are likely not algebraic. Nevertheless, we will see in our main results (\cref{result:main part 1,result:main part 2}) that the stack $S^M$ has satisfying properties infinitesimally, completely analogous to the infinitesimal properties of the stack $\Loc^M$ we saw in \cref{result:tangent of trivializable} above. For instance, if $M$ is the tuple of Robba fibers of $E \in \Isoc^\dagger(U/K)$ along $Z$, and we regard $E$ as a finite type point of $S^M$, we will show that the tangent space to the stack $S^M$ at the point $E$ is precisely Crew's parabolic cohomology $H^1_{p,\rig}(U, \End(E))$, and that this space carries a natural symplectic form, forcing its dimension to be even (when this dimension is finite). 

We should note that the stack $S$ has finite type points corresponding to modules with integrable connection over $K_0$ which do not satisfy the overconvergence condition. In other words, the stack $S$ is classifying objects that are not necessarily canonically attached to $U$ itself. Rather, it is classifying objects that are canonically related to the geometry of the \emph{tube} of $U$ inside the adic projective line. In fact, the objects which are overconvergent are somewhat sparse in $S$. For example, when $U = \P^1 \setminus \{0, \infty\}$ with coordinate $t$, the differential operators $t\partial - \lambda$ for $\lambda \in \bar{K}$ all determine finite type points of $S$, but these are only overconvergent when $\lambda \in \Z_p$. However, if $M$ is a tuple of \emph{solvable} differential modules over the Robba rings along $Z$, then in fact all of the finite type points of $S^M$ do satisfy the overconvergence property and define isocrystals on $U$. 

The picture this suggests is something like \cref{fig:moduli}. There is a ``Robba fibers'' map from the moduli $S$ of all modules with integrable connection on $\tube{U}$, down to the moduli $R^m$ of $m$-tuples of differential modules over the Robba rings along $Z$. The objects of $\Isoc^\dagger(U/\bar{K})$ are precisely the ones living over solvable differential modules over the Robba rings along $Z$.  

\begin{figure}[ht]
	\centering
	
	\begin{overpic}[width=0.3\textwidth]{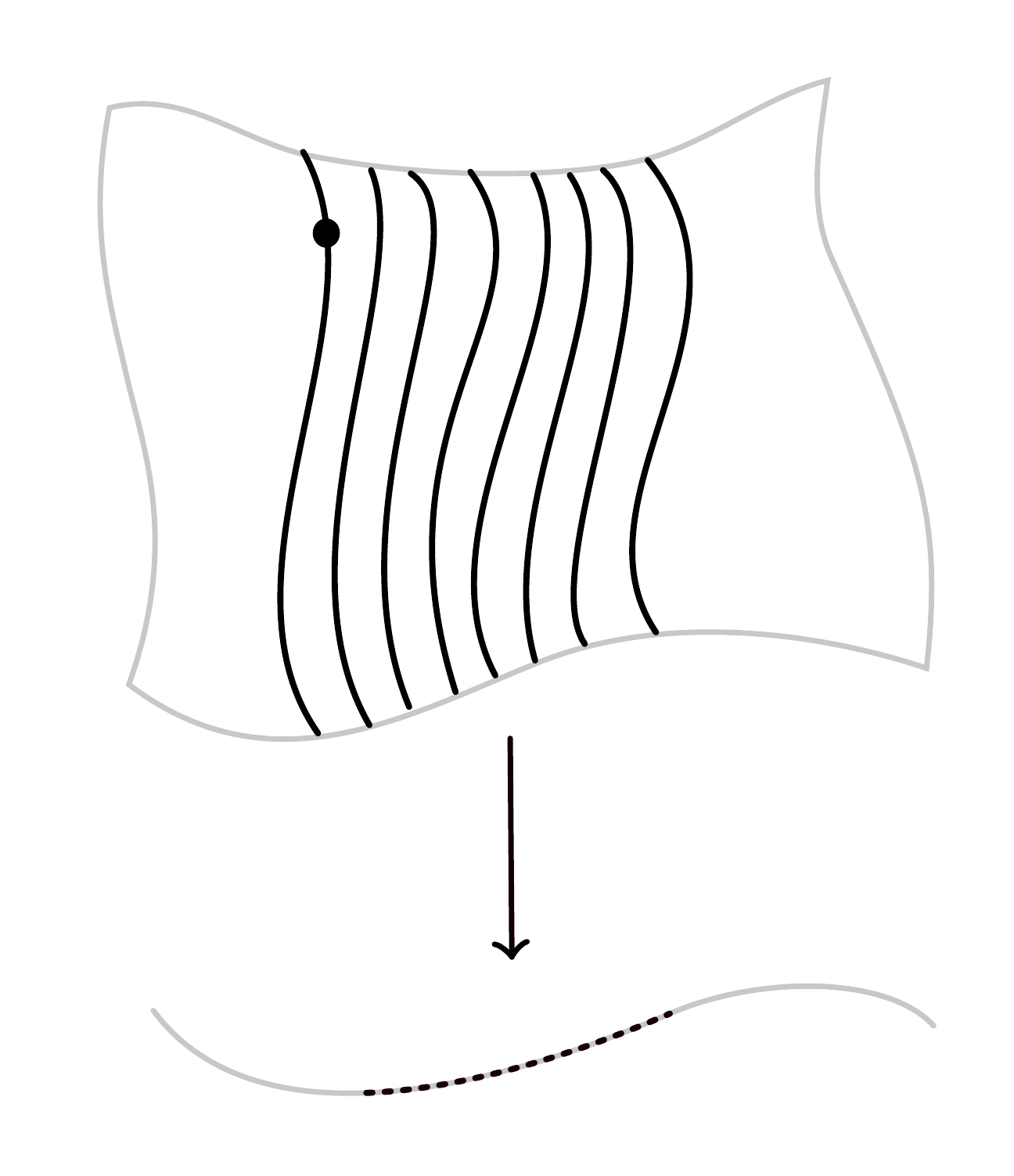}
		\put(82,10) {$R^m$}
		\put(82,47) {$S$}
		\put(20,77) {$E$}
	\end{overpic}

	\caption[Moduli of overconvergent isocrystals]{The moduli $S$ of modules with connections on $\tube{U}$. There is a natural ``Robba fibers'' map $S \to R^m$, and the preimage of the ``solvable locus'' of $R^m$ (indicated by the black dots) is the space of isocrystals on $U$ (indicated by the black fibers). Given $E \in \Isoc^\dagger(U/\bar{K})$, its tangent space in $S$ is $H^1_\rig(U, \End(E))$, and its tangent space in the fiber is $H^1_{p,\rig}(U, \End(E))$. Note that $S$ is formally smooth, but the fibers need not be. However, when $E$ is irreducible, it is a formally smooth point of the fiber.}
	\label{fig:moduli}
\end{figure}

In future work, I hope to explore ways in which this heuristic picture can be given rigorous geometric meaning. By realizing this picture in a sufficiently robust geometric setting, it may be possible to recreate the proof of Katz's criterion described earlier in the setting of overconvergent isocrystals, but this is not accomplished in this paper. 

\subsection{Contents}

\Cref{chapter:formal deformation theory} is background material on formal deformation theory. It is mostly expository, though at a few points we categorify and slightly expand some results that exist in the literature. This categorification makes for a cleaner story conceptually, and also permits us to prove a certain algebraization theorem for infinitesimal deformations of isocrystals later on (cf. \cref{result:algebraizing infinitesimal deformations}). 

In \cref{chapter:deformations of modules}, we apply the background material on formal deformation theory to studying deformations of a module over an arbitrary associative $K$-algebra $D$, where $K$ is a field of characteristic zero. The main result in this chapter is \cref{result:hochschild governs module deformations}, which shows that the Hochschild cochain complex $\Hoch_K(D, \End_K(E))$ governs the deformations of a left $D$-module $E$. This categorifies observations of Yau \cite{yau}. 

Then in \cref{chapter:deformations of differential modules}, we specialize further and study the deformations of a differential module $E$ over a differential commutative $K$-algebra $\O$, where $K$ is still a field of characteristic zero. The key here is \cref{result:de rham and hochschild}, which constructs an explicit quasi-isomorphism of differential graded algebras between the de Rham complex $\dR(\O, \End_\O(E))$ and the Hochschild cochain complex $\Hoch_K(D, \End_K(E))$, where $D$ is the associated ring of differential operators.\footnotemark

\footnotetext{The concepts involved in \cref{result:de rham and hochschild} bear some resemblance to concepts involved in the  Hochschild-Kostant-Rosenburg (HKR) theorem (cf. \cite[theorem 9.4.7]{weibel}, Kontsevich's formality theorem \cite{kontsevich}, and a version of the HKR theorem with coefficients \cite[theorem 4.5]{yekutieli_hochschild}). Note, however, that the Hochschild complexes involved in the HKR theorem are taken over a commutative algebra (and the coefficients involved in \cite{yekutieli_hochschild} are modules over said commutative algebra). In the result we discuss here, the relevant Hochschild complex is over a non-commutative algebra of differential operators, and the coefficients involved are two different endomorphism rings of a differential module over a differential ring. It is unclear if either the theorem here or the HKR theorem can be derived from the other.}

\Cref{isocrystal prelims} is mostly expository material on isocrystals on open subsets of the projective line, but there are some new results. \Cref{pid}, for instance, proves that the ring of functions on the tube of the complement of an effective Cartier divisor in $\P^1$ is a principal ideal domain; this is likely well-known to experts, but I know of no reference in the literature, and it has a number of important consequences that are described in \cref{isocrystal prelims}. Also, \cref{result:parabolic and middle} is the aforementioned comparison between Crew's parabolic cohomology and the intermediate extension operation in Berthelot's theory of arithmetic D-modules. Note that \cref{isocrystal prelims} is largely independent of everything that precedes it; starting in \cref{section:compactly supported and parabolic}, we do use some \emph{notation} that is introduced in \cref{chapter:deformations of differential modules}, but we do not yet use any of the \emph{results} of \cref{chapter:deformations of differential modules}. 

Finally, we put everything together in \cref{chapter:deformations of isocrystals}. \Cref{result:main part 1,result:main part 2} describe the infinitesimal deformation theory of isocrystals; combined, they provide a $p$-adic analog to \cref{result:tangent of trivializable} above. \Cref{result:algebraizing infinitesimal deformations} is a kind of ``algebraization'' theorem that falls out almost automatically as a result of our categorified approach to deformation theory. We conclude with a calculation in \cref{example:regular singularities} which is analogous to \cref{result:irreducible rank 2} above.

\subsection{Acknowledgements}

This work is based on my doctoral dissertation \cite{agrawal_thesis}. During the writing of my dissertation, I had countless conversations with my advisor, David Nadler, and also with Matthias Strauch and Christine Huyghe about this topic and many related ones; I cannot thank the three of them effusively enough. I would also like to thank Harrison Chen, Richard Crew, H\'el\`ene Esnault, Kiran Kedlaya, Adriano Marmora, Martin Olsson, Andrea Pulita, and Amnon Yekutieli for enlightening discussions. Part of this work was written at the University of Strasbourg while I was a Chateaubriand fellow; I would like to thank the University for their hospitality during my stay there, and the Embassy of France for their financial assistance. This work was also supported in part by NSF RTG grant DMS-1646385. Finally, I would like to thank the anonymous referee for many helpful remarks improving readability. 
	\section{Formal deformation theory} \label{chapter:formal deformation theory}

We fix a field $K$ of characteristic 0 and let $\Art_K$ denote the category of artinian local $K$-algebras with residue field $K$. In this section, we recount some generalities about deformation theory. The foundational work here is due to Schlessinger \cite{schlessinger}, but we use here the categorified version of Rim \cite[Expos\'e VI]{rim}, which considers opfibrations in groupoids over $\Art_K$ in place of functors $\Art_K \to \Set$. We generally\footnotemark\ follow terminology and notation set up for formal deformation theory in the Stacks Project \tag{06G7}, and we provide references therein whenever possible. 

\footnotetext{The only exceptions are that we use the slightly less sesquipedalian phrase \emph{opfibration in groupoids} in place of ``category cofibered in groupoids,'' and similarly \emph{hull} in place of ``miniversal formal object.''}

\subsection{Hulls and prorepresentability}
The definition of a deformation category is given in \tag{06J9}. To any deformation category $\scF$, one can associate a decategorified functor $\overline{\scF} : \Art_K \to \Set$ \tag{07W5}, as well as two vector spaces: its tangent space $T(\scF)$ \tag{06I1} and its space $\Inf(F)$ of infinitesimal automorphisms \tag{06JN}. Let us also draw attention to the fact that the 2-category of deformation categories has 2-fiber products \tag{06L4}, as we will use this several times. 

The following could be called the ``fundamental theorem of deformation theory.''

\begin{theorem} \label{deformation theory}
	Suppose $\scF$ is deformation category.
	\begin{enumerate}[label=(\alph{*})]
		\item $\scF$ has a hull if and only if its tangent space $T(\scF)$ is finite dimensional.
		\item The decategorified functor $\overline{\scF}$ is prorepresentable if and only if $T(\scF)$ is finite dimensional and $\Aut_{R'}(x') \to \Aut_R(x)$ is surjective whenever $x' \to x$ is a morphism in $\scF$ lying over a surjective homomorphism $R' \to R$ in $\Art_K$.
		\item $\scF$ is prorepresentable if and only if $T(\scF)$ is finite dimensional and $\Inf(\scF) = 0$.
	\end{enumerate}
\end{theorem}

\begin{proof}
	Since $\scF$ is a deformation category, it satisfies the Rim-Schlessinger condition (cf. \tag{06J2} for a definition) and therefore also axioms (S1) and (S2) (cf. \tag{06J7} for a definition). Thus (a) is precisely \tag{06IX}, and (b) follows from \tag{06JM} and \tag{06J8}. Finally, (c) follows from statement (b) together with \tag{06K0}.
\end{proof}

\subsection{Residual gerbes}

\begin{definition}
	Suppose $\scF$ is a predeformation category and fix an object $x_0$ lying above $K$. The \emph{residual gerbe} of $\scF$ is the full subcategory $\Gamma$ of $\scF$ spanned by objects $x$ such that there exists a morphism $x_0 \to x$. 
\end{definition}

In other words, the fiber of the residual gerbe $\Gamma$ over $R \in \Art_K$ is a connected groupoid. Moreover, if $x \in \Gamma$ lies over $R$, then the automorphism group $\Aut_R(x)$ is the same when is regarded as an object both of $\Gamma(R)$ and of $\scF(R)$. It is clear that $\Gamma$ is also a predeformation category.

\begin{lemma} \label{residual gerbe is deformation category}
	Let $\scF$ be a deformation category and $x_0$ an object lying above $K$. Suppose further that, for any morphism $x' \to x$ in the residual gerbe $\Gamma$ that lies above a surjective homomorphism in $\Art_K$, the map \[ \begin{tikzcd} \Hom(x_0, x') \ar{r} & \Hom(x_0, x) \end{tikzcd} \] is surjective. Then $\Gamma$ is a deformation category.
\end{lemma}

\begin{proof}
	Suppose we are given a diagram
	\[ \begin{tikzcd} & x_2 \ar{d} \\ x_1 \ar{r} & x \end{tikzcd} \]
	in $\Gamma$ where $x_2 \to x$ lies over a surjective homomorphism in $\Art_K$. Since $\scF$ is a deformation category, we can form the fiber product $x_1 \times_x x_2$ in $\scF$. Let us show that this fiber product lies in $\Gamma$. 
	
	Since $x_1$ is in $\Gamma$, there is a morphism $\tau_1 : x_0 \to x_1$ in $\scF$. Composing with the map $x_1 \to x$ gives us a map $x_0 \to x$, which, by our hypotheses, we can lift to a morphism $\tau_2 : x_0 \to x_2$ such that the following diagram commutes. 
	\[ \begin{tikzcd} x_0 \ar[bend left]{drr}{\tau_2} \ar[bend right]{ddr}[swap]{\tau_1} \\ & x_1 \times_x x_2 \ar{d} \ar{r} & x_2 \ar{d} \\ & x_1 \ar{r} & x \end{tikzcd} \]
	Thus $\tau = (\tau_1, \tau_2)$ defines a morphism $x_0 \to x_1 \times_x x_2$, proving that $x_1 \times_x x_2$ is in $\Gamma$. 
\end{proof}

\begin{remark}
	Observe that the hypothesis of the lemma is equivalent to the following: for every surjective homomorphism $R' \to R$ in $\Art_K$, there exist $x' \in \Gamma(R')$ and $x \in \Gamma(R)$ such that, for every morphism $x' \to x$, the induced group homomorphism $\Aut_{R'}(x') \to \Aut_R(x)$ is surjective. 
\end{remark}

\subsection{Obstruction theory}

We also need some basic definitions about obstruction theory. 

\begin{definition}
	Suppose $\scF$ is a deformation category. An \emph{obstruction space} for $\scF$ is a vector space $V$ over $K$ equipped with a collection of \emph{obstruction maps} $o(\pi, -) : \overline{\scF}(R) \to \ker(\pi) \otimes_K V$ for every small extension $\pi : R' \to R$ in $\Art_K$ \tag{06GD}, subject to the following conditions.
	\begin{enumerate}[label=(O\arabic{*})]
		\item Suppose we have a commutative diagram
		\[ \begin{tikzcd} R_2' \ar{r}{\pi_2} \ar{d}[swap]{\alpha'} & R_2 \ar{d}{\alpha} \\ R'_1 \ar{r}[swap]{\pi_1} & R_1 \end{tikzcd} \]
		in $\Art_K$ with $\pi_1$ and $\pi_2$ small extensions. Then for every $x \in \scF(R_2)$,
		\[ o(\pi_1, \alpha(x)) = (\alpha' \otimes 1_V)(o(\pi_2, x)). \]
		\item Suppose $\pi : R' \to R$ is a small extension and $x \in \scF(R)$. There exists a $x' \in \scF(R')$ and a morphism $x' \to x$ lying over $\pi$ if and only if $o(\pi, x) = 0$.
	\end{enumerate}
\end{definition}

\begin{lemma} \label{obstruction spaces exist}
	Any deformation category $\scF$ has an obstruction space.
\end{lemma}

\begin{proof}
	An obstruction space for $\scF$ as we have defined it is the same as a ``complete linear obstruction theory'' for the decategorified functor $\overline{\scF}$ in the sense of \cite[definitions 3.1, 4.1, and 4.7]{fantechi}. 
	
	Note that the natural map
	\[ \begin{tikzcd} \overline{\scF}(R \times_K R') \ar{r} & \overline{\scF}(R) \times \overline{\scF}(R') \end{tikzcd} \]
	is bijective for all $R, R' \in \Art_K$. This is a consequence of the Rim-Schlessinger condition; the proof is identical to the one in \tag{06I0}. By \cite[lemma 2.11]{fantechi}, it follows that $\scF$ is a ``Gdot functor'' in the sense of \cite[definition 2.10]{fantechi}. Combining \cite[theorem 3.2, corollary 4.4, and theorem 6.11]{fantechi}, we are done. 
\end{proof}

\begin{definition} \label{definition:compatible}
	Suppose that $\phi : \scF \to \textsc{G}$ is a morphism of deformation categories, that $V$ is an obstruction space for $\scF$ and $W$ is an obstruction space for $\textsc{G}$. A $K$-linear map $\gamma : V \to W$ is \emph{compatible} with $\phi$ if $o(\pi, \phi(x)) = (1_{\ker(\pi)} \otimes \gamma)(o(\pi, x))$ for all small extensions $\pi : R' \to R$ and all $x \in \scF(R)$.
\end{definition}

\begin{lemma}[Standard smoothness criterion] \label{result:standard smoothness criterion}
	Suppose $\phi : \scF \to \textsc{G}$ is a morphism of deformation categories such that
	\begin{enumerate}[label=(\roman{*})]
		\item $T(\scF) \to T(\textsc{G})$ is surjective, and
		\item there exists an injective compatible homomorphism $\gamma : V \to W$ of obstruction spaces, where $V$ is an obstruction space for $\scF$ and $W$ for $\textsc{G}$.
	\end{enumerate}
	Then $\phi$ is smooth \tag{06HG}.
\end{lemma}

\begin{proof}
	This is simply a rephrasing of \cite[proposition 2.17]{manetti_basics} in the setting of deformation categories. Suppose $\pi : R' \to R$ is a small extension, $y \in \textsc{G}(R')$, $x \in \scF(R)$ and $y \to \phi(x)$ is a morphism in $\textsc{G}$ lying over $\pi$. Then
	\[ 0 = o(\pi, \phi(x)) = (1_{\ker(\pi)} \otimes \gamma)(o(\pi, x)) \]
	so injectivity of $\gamma$ implies that $o(\pi, x) = 0$. Thus condition (O2) guarantees that there exists a morphism $x'_0 \to x$ lying over $\pi$.
	
	Observe that the morphisms $\phi(x'_0) \to \phi(x)$ and $y \to \phi(x)$ in $\textsc{G}$ lying over $\pi$ both define elements in the set $\Lift(\phi(x), \pi)$ \tag{06JE}. There is a free and transitive action of $T(\textsc{G}) \otimes_K \ker(\pi) = T(\textsc{G})$ on $\Lift(\phi(x), \pi)$ \tag{06JI}, so there exists $w \in T(\textsc{G})$ such that $(\phi(x'_0) \to \phi(x)) \cdot w = (y \to x)$. Since $T(\scF) \to T(\textsc{G})$ is surjective, we lift $w$ to some $v \in T(\scF)$ and then define $x' \to x$ to be a representative of the isomorphism class $(x'_0 \to x) \cdot v$. The functoriality of the action \tag{06JJ} guarantees that the isomorphism classes of $\phi(x') \to \phi(x)$ and $y \to \phi(x)$ in $\Lift(\phi(x), \pi)$ are equal. In other words, there exists an morphism $\phi(x') \to y$ in $\textsc{G}$ lying over $R$ and making the following diagram commute.
	\[ \begin{tikzcd} \phi(x') \ar{r} \ar[bend right]{dr} & y \ar{d}  \\ & \phi(x)  \end{tikzcd} \]
	This proves that $\phi$ is smooth \tag{06HH}.
\end{proof}

\begin{remark} \label{universal obstruction}
	For a deformation category $F$, consider the category whose objects are obstruction spaces for $F$ and whose morphisms are linear maps amongst obstruction spaces that are compatible with the identity on $F$ in the sense of \cref{definition:compatible} above. This category has an initial object $O_F$ and, for any obstruction space $V$, the unique map $O_F \to V$ is injective \cite[theorems 3.2 and 6.6]{fantechi}.  
\end{remark}

\subsection{Quotients by group actions}

\begin{definition}
	Suppose $F : \Art_K \to \Set$ and $G : \Art_K \to \Grp$ are deformation functors (in the sense of \tag{06JA}, which means in particular that they satisfy Schlessinger's homogeneity axiom (H4) \cite[page 213]{schlessinger}). Suppose further that $G$ acts on $F$. Let $[F/G]$ be the category whose objects are pairs $(R, x)$ where $R \in \Art_K$ and $x \in F(R)$, and whose morphisms $(R',x') \to (R,x)$ are pairs $(\pi, g)$ where $\pi : R' \to R$ is a homomorphism in $\Art_K$, $g \in G(R)$ and $g \cdot \pi(x') = x$, where $\pi(x')$ denotes the image of $x'$ under the map $F(\pi) : F(R') \to F(R)$. Composition is defined using the group operation on $G$. Specifically, given another morphism $(\pi', g') : (R'', x'') \to (R', x')$, we define 
	\[ (\pi, g) \circ (\pi', g') = (\pi \circ \pi', g \cdot \pi(g')). \]
	Evidently the functor $(R, x) \mapsto R$ presents $[F/G]$ as an opfibration in groupoids over $\Art_K$. 
\end{definition}

\begin{lemma} \label{quotient is deformation category}
	$[F/G]$ is a deformation category.
\end{lemma}

\begin{proof}
	Since $F$ and $G$ are both deformation functors, $F(K)$ is a singleton set and $G(K)$ is the trivial group. Thus $[F/G](K)$ has just one object and no nontrivial morphisms, so $[F/G]$ is a predeformation category \tag{06GS}. We need to verify the Rim-Schlessinger condition \tag{06J2}. This is straightforward once we recall that $G$ must be smooth \cite[theorem 7.19]{fantechi}. Suppose we have a diagram as follows in $[F/G]$, where $\pi_2 : R_2 \to R$ is surjective.
	\[ \begin{tikzcd} & (R_2, x_2) \ar{d}{(\pi_2, g_2)} \\ (R_1, x_1) \ar{r}[swap]{(\pi_1, g_1)} & (R, x) \end{tikzcd} \]
	Since $G$ is smooth, there exist $\tilde{g}_1, \tilde{g}_2 \in G(R_2)$ lifting $g_1, g_2 \in G(R)$ respectively. Let $S = R_1 \times_R R_2$. Since $F$ is a deformation category, we know that
	\[ \begin{tikzcd} F(S) \ar{r} & F(R_1) \times_{F(R)} F(R_2) \end{tikzcd} \]
	is bijective. The pair $(x_1, \tilde{g}_1^{-1} \tilde{g}_2 \cdot x_2)$ is an element of $F(R_1) \times_{F(R)} F(R_2)$, so let $y$ be the corresponding element of $F(S)$. Then we have a commutative diagram as follows, where $\rho_i : S \to R_i$ are the canonical maps in $\Art_K$.
	\[ \begin{tikzcd} (S, y) \ar{rr}{(\rho_2, \tilde{g}_2^{-1} \tilde{g}_1)} \ar{d}[swap]{(\rho_1, 1)} & & (R_2, x_2) \ar{d}{(\pi_2, g_2)} \\ (R_1, x_1) \ar{rr}[swap]{(\pi_1, g_1)} & & (R, x) \end{tikzcd} \]
	To check that this square is cartesian, suppose we have a diagram of solid arrows as follows.
	\[ \begin{tikzcd} (T, z) \ar[bend left]{drrr}{(\tau_2, h_2)} \ar[bend right]{ddr}[swap]{(\tau_1,  h_1)} \ar[dotted]{dr} \\
	& (S, y) \ar{rr}{(\rho_2, \tilde{g}_2^{-1} \tilde{g}_1)} \ar{d}[swap]{(\rho_1, 1)} & & (R_2, x_2) \ar{d}{(\pi_2, g_2)} \\
	& (R_1, x_1) \ar{rr}[swap]{(\pi_1, g_1)} & & (R, x) \end{tikzcd} \]
	Since $S = R_1 \times_R R_2$ in $\Art_K$, there exists $\tau :  T \to S$ such that $\tau \circ \rho_i = \tau_i$ for $i = 1, 2$. The commutativity of the large square tells us that $g_1 \pi_1(h_1) = g_2 \pi_2(h_2)$, so 
	\[ (h_1, \tilde{g}_1^{-1} \tilde{g}_2 h_2) \in G(R_1) \times_{G(R)} G(R_2). \]
	Since $G$ is a deformation functor, we let $h \in G(S)$ be the corresponding element under the bijection
	\[ \begin{tikzcd} G(S) \ar{r} & G(R_1) \times_{G(R)} G(R_2). \end{tikzcd} \]
	It is then easily verified that $(\tau, h)$ defines the desired dotted arrow $(T, z) \to (S, y)$.
\end{proof}

\begin{lemma} \label{smoothness of projection to quotient}
	The natural morphism $F \to [F/G]$ is smooth. Moreover, every obstruction space $V$ for $F$ is canonically an obstruction space for $[F/G]$ in such a way that the identity map on $V$ is compatible with $F \to [F/G]$. 
\end{lemma}

\begin{proof}
	Observe that $G$ is smooth \cite[theorem 7.19]{fantechi} and $\overline{[F/G]}$ coincides with the functor $\Art_K \to \Set$ that is denoted $F/G$ in \cite[page 570]{fantechi}, so $F \to \overline{[F/G]}$ is smooth \cite[proposition 7.5]{fantechi}. 
	\[ \begin{tikzcd} F \ar{r} \ar{dr} & \left[F/G\right] \ar{d} \\ & \overline{[F/G]} \end{tikzcd} \]
	Since $F \to [F/G]$ is essentially surjective and $[F/G] \to \overline{[F/G]}$ is smooth \tag{06HK}, it follows that $F \to [F/G]$ is smooth \tag{06HM}. 
	
	If $V$ is an obstruction space for $F$, there is a natural injective linear map $O_F \hookrightarrow V$ as in \cref{universal obstruction}. The natural map $O_F \to O_{[F/G]}$ is an isomorphism \cite[proposition 7.5]{fantechi}, so composing its inverse with $O_F \hookrightarrow V$ yields an injective linear map $O_{[F/G]} \hookrightarrow V$. This map makes $V$ an obstruction space for $[F/G]$. 
\end{proof}

\begin{definition}
	Let $0$ denote the unique element of $F(K)$ and also its image in $F(R)$ under the map $F(K) \to F(R)$ for every $R \in \Art_K$. The \emph{stabilizer} of the action of $G$ on $F$, denoted $\Stab_G$, is the subfunctor of $G$ that associates to each $R \in \Art_K$ the subgroup
	\[ \Stab_{G(R)}(0) := \{ g \in G(R) : g \cdot 0 = 0 \}. \]
\end{definition}

\begin{lemma} \label{stabilizer of quotient}
	The stabilizer $\Stab_G$ is a deformation functor. 
\end{lemma}

\begin{proof}
	Observe that we have an ``act on 0'' map $G \to F$ that carries $g \in G(R)$ to $g \cdot 0 \in F(R)$ for all $R \in \Art_K$. Suppose we have homomorphisms
	\[ \begin{tikzcd}& R_2 \ar{d} \\ R_1 \ar{r} & R \end{tikzcd} \]
	in $\Art_K$ with $R_2  \to R$ surjective. We then obtain a commutative diagram as follows, where we write $S := \Stab_G$ to ease notation. 
	\[ \begin{tikzcd} S(R_1 \times_R R_2) \ar{r} \ar[hookrightarrow]{d} & S(R_1) \times_{S(R)} S(R_2) \ar[hookrightarrow]{d} \\ G(R_1 \times_R R_2) \ar{r}{\sim} \ar{d} & G(R_1) \times_{G(R)} G(R_2) \ar{d} \\ F(R_1 \times_R R_2) \ar{r}{\sim} & F(R_1) \times_{F(R)} F(R_2) \end{tikzcd}  \]
	The lower two horizontal arrows are isomorphisms since $F$ and $G$ are deformation functors. 
	
	It is now an elementary diagram chase to prove that the horizontal arrow on top is also an isomorphism. Indeed, it follows immediately from injectivity of $G(R_1 \times_R R_2) \to G(R_1) \times_{G(R)} G(R_2)$ that $S(R_1 \times_R R_2) \to S(R_1) \times_{S(R)} S(R_2)$ is injective. Now suppose we have $(g_1, g_2) \in S(R_1) \times_{S(R)} S(R_2)$. Then there exists $g \in G(R_1 \times_R R_2)$ which maps to $(g_1, g_2)$. Since $g_1 \in S(R_1)$ and $g_2 \in S(R_2)$, we know that the image of $(g_1, g_2)$ in $F(R_1) \times_{F(R)} F(R_2)$ is $(0, 0)$. Thus the image of $g$ in $F(R_1 \times_R R_2)$ is $0$. It follows that $g \in S(R_1 \times_R R_2)$, proving that $S$ is a deformation functor.  
\end{proof}

For any local $K$-algebra $P \in \Art_K$, we write $h_P$ to denote the functor $\Art_K \to \Set$ where $h_P(T)$ is the set of local $K$-algebra homomorphisms $P \to T$. In particular, when $P = K$, observe that $h_K(T)$ is a singleton set for all $T$. 

\begin{definition}
	We set $BG := [h_K/G]$, where $G$ acts trivially on $h_K$. 
\end{definition}

\begin{lemma} \label{factor the point}
	The ``pick out 0'' morphism $h_K \to [F/G]$ factors as
	\[ \begin{tikzcd} h_K \ar{r} & B\!\Stab_G \ar{r} & \left[F/G \right], \end{tikzcd} \]
	and $B\!\Stab_G \to [F/G]$ is fully faithful with essential image the residual gerbe of $[F/G]$. \qed
\end{lemma}

\begin{corollary} \label{residual gerbe of quotient deformation category}
	The residual gerbe of $[F/G]$ is a deformation category. 
\end{corollary}

\begin{proof}
	$\Stab_G$ is a deformation functor by \cref{stabilizer of quotient}, so $B\Stab_G$ is a deformation category by \cref{quotient is deformation category}. Since the residual gerbe of $[F/G]$ is equivalent to $B\Stab_G$ by \cref{factor the point}, the result follows.
\end{proof}

\subsection{Differential graded Lie algebras}

A great deal of work has been done relating differential graded Lie algebras to deformation theory. The underlying philosophy (due originally to Deligne, Drinfeld, Feigin, Kontsevich, and others) is that ``reasonable'' deformation problems are governed by differential graded Lie algebras. Work of Lurie formalizes this philosophy in an $\infty$-categorical framework (cf. \cite{lurie_icm} for an overview of this work). 

Here, we recall just a few relevant portions of this theory, avoiding words like ``$\infty$-category.'' To further simplify the exposition of the theory, we will assume that all differential graded Lie algebras are concentrated in nonnegative degrees, as this is the only case that will be relevant for us. 

To lighten notation and decrease verbosity, all unadorned tensor products in this subsection are assumed to be over $K$, and algebras, Lie algberas, differential graded Lie algebras, etc, are also assumed to be over $K$, unless explicitly specified otherwise.

Fix a differential graded Lie algebra $L$ concentrated in nonnegative degrees. 

\begin{definition}[Gauge group]
	For $R \in \Art_K$, let $\m_R$ denote its maximal ideal. We can regard the nilpotent $R$-Lie algebra $G_L(R) = \m_R \otimes L^0$ as a group by defining a group operation $*$ using the Baker-Campbell-Hausdorff formula: we set
	\[ \eta * \eta' = \log\left( \exp(\eta) \exp(\eta')\right)  \]
	for all $\eta, \eta' \in \m_R \otimes L^0$. The formal power series on the right-hand side are computed using the $R$-algebra structure on the universal enveloping $R$-algebra $U(\m_R \otimes_K L^0)$, and \cite[theorem 7.4]{serre} guarantees that the result of these computations is actually in $\m_R \otimes L^0$. Observe moreover that 0 is the unit element of this group structure, and that the additive inverse $-\eta$ of $\eta$ is also the inverse of $\eta$ with respect to this group structure.
	
	This is all evidently natural in $R$, so we obtain a functor $G_L : \Art_K \to \Grp$, which is in fact a deformation functor \cite[section 3]{manetti_basics}. It is called the \emph{gauge group} of $L$. 
\end{definition}

\begin{remark}
	If $L^0$ is itself an algebra (a special case which will be important for us), we obtain a commutative diagram
	\[ \begin{tikzcd} \m_R \otimes L^0 \ar[hookrightarrow]{r} \ar[hookrightarrow, bend right]{dr} & U(\m_R \otimes L^0) \ar{d} \\
	& R \otimes L^0 \end{tikzcd} \]
	where the horizontal maps are the natural inclusions and the vertical map is the $R$-algebra homomorphism induced by the universal property of the universal enveloping algebra. 
	Thus, the power series on the right-hand side of the Baker-Campbell-Hausdorff formula can be computed using the natural $R$-algebra structure on $R \otimes L^0$.
\end{remark}

\begin{definition}[Gauge action] \label{gauge}
	Let $F_L : \Art_K \to \Set$ be the functor $R \mapsto \m_R \otimes L^1$. The \emph{gauge action} of $G_L$ on $F_L$ is defined by the formula
	\[ \eta * x = x + \sum_{n = 0}^\infty \frac{[\eta, -]^n}{(n+1)!} ([\eta, x] - d\eta) \]
	for $\eta \in \m_R \otimes L^0$ and $x \in \m_R \otimes L^1$. Since $\m_R$ is a nilpotent ideal in $R$, the endomorphism $[\eta, -]$ is nilpotent, so the sum in the above formula is finite. One checks that this is, in fact, a group action: in other words, we have 
	\[ \eta' * (\eta * x) = (\eta' * \eta) * x. \]
\end{definition}

\begin{definition}[Maurer-Cartan elements]
	For $x \in \m_R \otimes L^1$, we define
	\[ Q(x) := dx + \frac{1}{2}[x,x]. \]
	If $Q(x) = 0$, then $x$ is a \emph{Maurer-Cartan element} of $\m_R \otimes L^1$. We define $\MC_L : \Art_K \to \Set$ to be the functor $\Art_K \to \Set$ sending $R$ to the set of Maurer-Cartan elements of $\m_R \otimes L^1$. This is a deformation functor \cite[section 3]{manetti_basics}. 
\end{definition}

\begin{definition}[Deformation category associated to a differential graded Lie algebra]
	With $F_L$ as in \cref{gauge} above, the action of $G_L$ on $F_L$ stabilizes $\MC_L$ \cite[section 1]{manetti_basics}, so we can define
	\[ \Def_L := [\MC_L/G_L]. \]
	By \cref{quotient is deformation category}, this is a deformation category.
\end{definition}

\begin{remark}
	In \cite{yekutieli}, the deformation category $\Def_L$ is called the \emph{(reduced) Deligne groupoid} of $L$. Note that, since we have assumed that $L$ is concentrated in nonnegative degrees, the distinction between the reduced Deligne groupoid and the Deligne groupoid is neutralized. 
\end{remark}

\begin{theorem}[{\cite[proposition 2.6]{goldman}}] \label{result:fundamental theorem of dglas}
	The deformation category $\Def_L$ has infinitesimal automorphisms $H^0(L)$, tangent space $H^1(L)$, and obstruction space $H^2(L)$.
\end{theorem}

\begin{definition}
	If $\scF$ is a category over $\Art_K$ and $L$ is a differential graded Lie algebra such that there exists an equivalence $\Def_L \to \scF$ of categories over $\Art_K$, we then say that $L$ \emph{governs} $\scF$.
\end{definition}

\begin{example} \label{easy dgla}
	Suppose $V$ is a finite dimensional vector space and consider the differential graded Lie algebra $V[-1]$. In other words, this is nonzero in only degree 1, where it is $V$, and the Lie bracket is necessarily trivial. It is then straightforward to construct a natural equivalence $\Def_{V[-1]} = h_{\hat{P}}$, where $\hat{P}$ denotes the completion of $P = \Sym(V^\vee)$ along the maximal ideal generated by $V^\vee$ and $h_{\hat{P}}$ is the functor $\Art_K \to \Set$ sending $T \in \Art_K$ to the set of local $K$-algebra homomorphisms $\hat{P} \to T$.
\end{example}

\begin{example} \label{residual gerbe governed by h0}
	The residual gerbe $\Gamma$ of $\Def_L$ is also a deformation category by \cref{residual gerbe of quotient deformation category}. In fact, it is easy to see that 
	\[ \Stab_{G_L(R)}(0) = \m_R \otimes Z^0(L) = \m_R \otimes H^0(L) \]
	for every $R \in \Art_K$, so there is a natural equivalence $\Def_{H^0(L)} \simeq \Gamma$. In other words, $\Gamma$ is governed by $H^0(L)$.  
\end{example}

\subsection{Homomorphisms of differential graded Lie algebras}

\subsubsection{Functoriality} \label{section:def construction is 2-functorial}

The construction $L \mapsto \Def_L$ is functorial: i.e., any homomorphism $L \to M$ of differential graded Lie algebras concentrated in nonnegative degrees induces a natural functor $\Def_L \to \Def_M$ \cite[paragraph 2.3]{goldman}.

\begin{lemma}
	The identifications of \cref{result:fundamental theorem of dglas} fit into commutative diagrams as follows.
	\[ \begin{tikzcd} \Inf(\Def_L) \ar{r} \ar{d} & \Inf(\Def_M) \ar{d} & & T(\Def_L) \ar{r} \ar{d} & T(\Def_M) \ar{d} \\ H^0(L) \ar{r} & H^0(M) & & H^1(L) \ar{r} & H^1(M) \end{tikzcd} \]
	Moreover, the  map of obstruction spaces $H^2(L) \to H^2(M)$ is compatible with $\Def_L \to \Def_M$ in the sense of \cref{definition:compatible}. 
\end{lemma}

\begin{proof}
	The identifications of the infinitesimal deformations and of the tangent space in \cref{result:fundamental theorem of dglas} are given by ``deleting $\epsilon$.'' More precisely, they are induced by the isomorphism $\m_{K[\epsilon]} \simeq K$ given by $\epsilon \mapsto 1$. The commutativity of the two squares follows from this. The fact that $H^2(L) \to H^2(M)$ is compatible with $\Def_L \to \Def_M$ follows from the construction of obstruction classes; see \cite[section 2]{manetti_basics}. 
\end{proof}

\subsubsection{Quasi-isomorphism invariance}

Since $L \mapsto \Def_L$ is 2-functorial, certainly the functor $\Def_L \to \Def_M$ must be an equivalence whenever $L \to M$ is an isomorphism of differential graded Lie algebras. In fact, the same is true when $\phi$ is only a \emph{quasi}-isomorphism as well.  

\begin{theorem} \label{result:quasi-isomorphism invariance}
	If $L \to M$ is a quasi-isomorphism of differential graded Lie algebras, then $\Def_L \to \Def_{M}$ is an equivalence of deformation categories.
\end{theorem}

\begin{remark}
	Proofs of this can be found in \cite[theorem 2.4]{goldman} or \cite[theorem 4.2]{yekutieli}. The former reference calls this the ``equivalence theorem'' and attributes it to Deligne; the latter generalizes the former in addition to correcting a mistake in the former (cf. \cite[proof of lemma 3.11]{yekutieli}). In fact, it is not necessary to assume that $L \to M$ is a quasi-isomorphism: it is sufficient to assume that it induces isomorphisms on cohomology in degrees 0 and 1, and an injective map on cohomology in degree 2. We will not have need to use the theorem under these weaker hypotheses. 
\end{remark}

\subsubsection{Fiber products}

In what follows, we fix homomorphisms $\phi_1 : L_1 \to M$ and $\phi_2 : L_2 \to M$ of differential graded Lie algebras concentrated in nonnegative degrees, and we let $\phi$ denote the pair $(\phi_1, \phi_2)$. The following generalizes and categorifies the main construction of \cite{manetti_morphisms}. 

\begin{definition}
	Define the functor $\MC_{\phi} : \Art_K \to \Set$ where $\MC_\phi(R)$ is the set of triples $(x_1, x_2, \tau)$ where $x_i \in \MC_{L_i}(R)$ for $i = 1, 2$, $\tau \in G_M(R)$, and 
	\[ \tau * \phi_1(x_1) = \phi_2(x_2). \]
	There is an action of $G_{L_1} \times G_{L_2}$ on $\MC_\phi$ where $(\eta_1, \eta_2) \in G_{L_1}(R) \times G_{L_2}(R)$ acts on $(x_1, x_2, \tau) \in \MC_\phi(R)$ by 
	\[ (\eta_1, \eta_2) * (x_1, x_2, \tau) = (\eta_1 * x_1, \eta_2 * x_2, \phi_2(\eta_2) * \tau * (-\phi_1(\eta_1)). \]
	We then define $\Def_\phi := [\MC_\phi/(G_{L_1} \times G_{L_2})]$. 
\end{definition}

The above definition is cooked up precisely so that we have the following.

\begin{lemma} \label{fiber product}
	The forgetful functors $\Def_\phi \to \Def_{L_1}$ and $\Def_\phi \to \Def_{L_2}$ fit into a 2-cartesian diagram as follows.
	\[ \begin{tikzcd} \Def_\phi \ar{r} \ar{d} &  \Def_{L_2} \ar{d} \\ \Def_{L_1} \ar{r} &  \Def_M \end{tikzcd} \]
\end{lemma}

\begin{proof}
	Unwinding the construction of 2-fiber products in the $(2,1)$-category of categories over $\Art_K$ described in \tag{0040}, we find exactly the category $\Def_\phi$ described above.
\end{proof}

The pair $\phi = (\phi_1, \phi_2)$ defines a homomorphism $\phi_1 - \phi_2 : L_1 \oplus L_2 \to M$ of differential graded Lie algebras. We set
\[ C := \Cone(\phi_1 - \phi_2 : L_1 \oplus L_2 \to M)[-1]. \]

\begin{theorem} \label{deformations and fiber product}
	The deformation category $\Def_\phi$ has infinitesimal automorphisms $H^0(C)$ and tangent space $H^1(C)$. Moreover, if 
	\[ \phi_2(\m_R \otimes L_2^1) \subseteq \MC_M(R) \]
	for all $R \in \Art_K$, then $H^2(C)$ is an obstruction space for $\Def_\phi$. 
\end{theorem}

\begin{proof}
	The argument is essentially identical to the one \cite[section 2]{manetti_morphisms}, but we record it here for completeness. We will write elements of $\m_{K[\epsilon]} \otimes V$ as $\epsilon v$ where $v \in V$. First, the stabilizer of $(0,0,0) \in \MC_\phi(K[\epsilon])$ is given by pairs 
	\[ (\epsilon \eta_1, \epsilon \eta_2) \in G_{L_1}(K[\epsilon]) \times G_{L_2}(K[\epsilon]) \]
	such that
	\[ (\epsilon \eta_1, \epsilon \eta_2) * (0,0,0) = (\epsilon \eta_1 * 0, \epsilon \eta_2 * 0,  \phi_2(\epsilon \eta_2) * (- \phi_1(\epsilon \eta_1))) \]
	equals $(0,0,0)$. Using the fact that $\epsilon^2 = 0$, we find that this condition is equivalent to $\eta_1 \in Z^0(L_1)$, $\eta_2 \in Z^0(L_2)$, and 
	\[ \phi_2(\eta_2) = \phi_1(\eta_1), \]
	so $(\epsilon \eta_1, \epsilon \eta_2)$ stabilizes $(0,0,0)$ if and only if
	\[ (\eta_1, \eta_2) \in (Z^0(L_1) \oplus Z^0(L_2)) \cap \ker(\phi_1 - \phi_2) = Z^0(C) = H^0(C) \]
	where we have used the fact that $L_1, L_2$ and $M$ all vanish in negative degrees. This shows that ``deleting $\epsilon$'' defines an isomorphism
	\[ \Inf(\Def_\phi) \simeq H^0(C). \]
	
	Next, let us compute the tangent space. Since $\epsilon^2 = 0$, we have
	\[ \MC_{L_i}(K[\epsilon]) = Z^1(\m_{K[\epsilon]} \otimes L_i) = \m_{K[\epsilon]} \otimes Z^1(L_i). \]
	Suppose now that $\epsilon x_i \in \MC_{L_i}(K[\epsilon])$. Then $\epsilon \tau \in G_M(K[\epsilon])$ satisfies
	\[ (\epsilon \tau) * \phi_1(\epsilon x_1) = \phi_2(\epsilon x_2) \]
	if and only if
	\[ \phi_1(x_1) - d\tau = \phi_2(x_2). \]
	Thus we see that
	\[ \MC_\phi(K[\epsilon]) \simeq \{ (x_1, x_2, \tau) : x_i \in Z^1(L_i) \text{ and } \phi_1(x_1) - \phi_2(x_2) = d\tau \} = Z^1(C). \]
	Now note that an element $(\epsilon x_1, \epsilon x_2, \epsilon \tau)$ is gauge equivalent to $(0,0,0)$ precisely if there exists 
	\[ (\epsilon \eta_1, \epsilon \eta_2) \in G_{L_1}(K[\epsilon]) \times G_{L_2}(K[\epsilon]) \]
	such that
	\[ (\epsilon \eta_1 * 0, \epsilon \eta_2 * 0, \phi_2(\epsilon \eta_2) * (-\phi_1(\epsilon \eta_1))) = (\epsilon x_1, \epsilon x_2, \epsilon \tau). \]
	Note that $\epsilon \eta_i * 0 = -\epsilon d\eta_i$ and 
	\[ \phi_2(\epsilon \eta_2) * (-\phi_1(\epsilon \eta_1)) = \epsilon(\phi_2(\eta_2) - \phi_1(\eta_1)), \]
	so $(\epsilon x_1, \epsilon x_2, \epsilon \tau)$ is gauge equivalent to $(0,0,0)$ precisely if there exists $(\eta_1, \eta_2)$ such that
	\[ d_C(-\eta_1, -\eta_2) = (x_1, x_2, \tau). \]
	Thus ``deleting $\epsilon$'' defines an isomorphism
	\[ T(\Def_\phi) \simeq H^1(C).  \]
	
	Finally, we want to define on $H^2(C)$ the structure of an obstruction space for $\Def_\phi$. By \cref{smoothness of projection to quotient}, it is sufficient to define the structure of an obstruction space for $\MC_\phi$. Let $\pi : R' \to R$ be a small extension in $\Art_K$ and suppose $(x_1, x_2,\tau) \in \MC_\phi(R)$. Since the functor $\Art_K \to \Set$ given by $R \mapsto \m_R \otimes L_i^1$ is evidently smooth, there exists $x_i' \in \m_{R'} \otimes L^1_i$ such that $\pi(x_i') = x_i$ for $i = 1,2$. Also, since $G_M$ is smooth, there exists $\tau' \in G_M(R')$ such that $\pi(\tau') = \tau$. We now define
	\[ h_i := Q(x_i') = dx_i' + \frac{1}{2}[x_i', x_i'] \text{ for $i = 1, 2$, and } s := \tau' * \phi_1(x_1') - \phi_2(x_2'). \]
	Since $(x_1, x_2, \tau) \in \MC_\phi$, we have $\pi(h_i) = 0$ and $\pi(s) = 0$, so $(h_1, h_2, s)$ is an element of \[ \ker(\pi) \otimes C^2 = \ker(\pi) \otimes (L^2_1 \oplus L^2_2 \oplus M^1), \]
	and $(x'_1, x'_2, \tau') \in \MC_\phi(R')$ if and only if $(h_1, h_2, s) = 0$. We will show the following. 
	\begin{enumerate}[label=(\alph{*})]
		\item Replacing $(x_1', x_2', \tau')$ with different lifts corresponds precisely to shifting $(h_1, h_2, s)$ by a 2-coboundary in $\ker(\pi) \otimes C^2$.
		\item $(h_1, h_2, s)$ is a 2-cocycle in $\ker(\pi) \otimes C^2$.
	\end{enumerate}
	Once we have proved these facts, we can then define $o(\pi, (x_1, x_2, \tau))$ to be the class in $\ker(\pi) \otimes H^2(C)$ represented by $(h_1, h_2, s)$. This class is independent of choices and measures exactly the obstruction to lifting $(x_1, x_2, \tau)$.
	
	Let us first look at point (a). Given two lifts $(x_1', x_2', \tau')$ and $(x_1'', x_2'', \tau'')$ of $(x_1, x_2, \tau)$, their difference
	\[ (\epsilon_1, \epsilon_2, \delta) := (x_1'', x_2'', \tau'') - (x_1', x_2', \tau') \]
	is an element $\ker(\pi) \otimes C^1$. Then the difference between the corresponding $h_i$'s is exactly $d\epsilon_i$. The proof of this is identical to the arguments in \cite[paragraph 2.7]{goldman} or \cite[section 3]{manetti_basics}. Now consider the difference between the corresponding $s$'s. 
	\begin{equation} \label{difference of s's} \left( \tau'' * \phi_1(x_1'') - \phi_2(x_2'') \right) - \left(\tau' * \phi_1(x_1') - \phi_2(x_2') \right) \end{equation}
	Since $\tau'' = \tau' + \delta$, we find by applying \cite[lemma 2.8]{goldman} that 
	\[ \tau'' * \phi_1(x_1'') = \tau' * \phi_1(x_1'') - d\delta. \]
	Now note that
	\[ \tau' * \phi_1(x_1'') - \tau' * \phi_1(x_1') = \phi_1(\epsilon_1) + \sum_{n = 0}^\infty \frac{[\tau', -]^n}{(n+1)!}(\phi(\epsilon_1)) = \phi_1(\epsilon_1), \]
	where the sum vanishes because $\tau' \in \m_{R'} \otimes M^0$, $\phi(\epsilon_1) \in \ker(\pi) \otimes M^1$, and $\ker(\pi)\m_R = 0$. 
	Putting all of this together, we find
	\[ \eqref{difference of s's} = \phi_1(\epsilon_1) - \phi_2(\epsilon_2) - d\delta. \]
	This shows that replacing $(x_1', x_2', \tau')$ with $(x_1'', x_2'', \tau'')$ corresponds to replacing $(h_1, h_2, s)$ with 
	\[ (h_1 + d\epsilon_1, h_2 + d\epsilon_2, s + \phi_1(\epsilon_1) - \phi_2(\epsilon_2) - d\delta) = (h_1, h_2, \delta) + d(\epsilon_1, \epsilon_2, \delta). \]
	This concludes the proof of point (a).
	
	For point (b), we compute
	\[ d(h_1, h_2, s) = (dh_1, dh_2, \phi_1(h_1) - \phi_2(h_2) - ds). \]
	The proof that $dh_i = 0$ is identical to the arguments in \cite[paragraph 2.7]{goldman} or \cite[section 3]{manetti_basics}, so we just need to show that
	\[ \phi_1(h_1) - \phi_2(h_2) = ds. \]
	Since $\tau' * \phi_1(x_1') = \phi_2(x_2') + s$ by definition of $s$, we have
	\[ \begin{aligned} \phi_1(x_1') &= (-\tau') * (s + \phi_2(x_2')) \\
	&= \exp([-\tau', -])(s) + (-\tau') * \phi_2(x_2') \\
	&= s + (-\tau')*\phi_2(x_2') \end{aligned} \]
	where for the last step, we have used the fact that $[-\tau', s] = 0$ since $\ker(\pi) \m_{R'} = 0$. Then
	\[ \begin{aligned} \phi_1(h_1) &= \phi_1 \left(dx_1' + \frac{1}{2}[x_1', x_1'] \right) \\
	&= d\phi_1(x_1') + \frac{1}{2}[\phi_1(x_1'), \phi_1(x_1')]  \\
	&= d\left(s + (-\tau') * \phi_2(x_2') \right) + \frac{1}{2}\left[s + (-\tau')*\phi_2(x_2'), s + (-\tau')*\phi_2(x_2') \right] \\
	&= ds + d((-\tau') * \phi_2(x_2')) + \frac{1}{2}[(-\tau')*\phi_2(x_2'), (-\tau')*\phi_2(x_2')] \\
	&= ds + Q((-\tau') * \phi_2(x_2')) \end{aligned} \]
	so it follows that
	\[ \phi_1(h_1) - \phi_2(h_2) - ds = Q((-\tau') * \phi_2(x_2')) - Q(\phi_2(x_2')). \]
	Since we have assumed that $\phi_2(\m_{R'} \otimes L^1_2) \subseteq \MC_{M}(R')$, we see that the right-hand side of the above equation vanishes, completing the proof.  
\end{proof}

\begin{remark}
	In the theorem above, we prove that $H^2(C)$ is an obstruction space for $\Def_\phi$ only under the technical hypothesis that $\phi_2(\m_R \otimes L^1_2) \subseteq \MC_M(R)$ for all $R \in \Art_K$. We do not know if $H^2(C)$ is an obstruction space for $\Def_\phi$ without this assumption, but the proof above does at least show that $C^2/B^2(C)$ is always an obstruction space even without this technical hypothesis. In any case, we will only use the above result when $L_2^1 = 0$, in which case this hypothesis is automatically satisfied. 
\end{remark}

\begin{remark} \label{producing a long exact sequence}
	With $C = \Cone(\phi_1 - \phi_2)[-1]$ and $\phi_2(\m_R \otimes L^1_2) \subseteq \MC_M(R)$ for all $R \in  \Art_K$ as above, we have a distinguished triangle
	\[ \begin{tikzcd} C \ar{r} & L_1 \oplus L_2 \ar{r}{\phi_1 - \phi_2} & M \ar{r}{+} & \mbox{} \end{tikzcd} \]
	in the derived category of vector spaces. The long exact sequence on cohomology associated to this distinguished triangle then relates the infinitesimal automorphisms, tangent spaces, and obstruction spaces of $\Def_\phi$, $\Def_{L_1}$, $\Def_{L_2}$, and $\Def_M$.
\end{remark}

\begin{definition}
	If $\scF$ is category over $\Art_K$, we say that the pair $\phi = (\phi_1, \phi_2)$ \emph{governs} $\scF$ if there exists an equivalence $\Def_\phi \to \scF$ of categories over $\Art_K$.
\end{definition}

\begin{remark}
	One could say that $\Def_\phi$ is ``governed'' by $C = \Cone(\phi_1 - \phi_2)[-1]$, but $C$ is not a differential graded Lie algebra. Readers with inclinations towards higher category theory may appreciate knowing that $C$ is naturally an $L_\infty$-algebra \cite{fiorenza}, but we will not need this. It is also possible to find a differential graded Lie algebra that does govern $\Def_\phi$ \cite[section 7]{manetti_morphisms}, but we will not need this either.
\end{remark}

\begin{example} \label{backbone example}
	We will now discuss an extended example that will serve as the backbone for the discussion in \cref{section:trivialized and trivializable}. Let $\alpha : L \to M$ be a homomorphism of differential graded Lie algebras concentrated in nonnegative degrees, and let $\Gamma$ denote the residual gerbe of $\Def_M$. We then form the following diagram of deformation categories.
	\[ \begin{tikzcd} \Def_{(\alpha, 0)} \ar{r} \ar{d} & \Def_{(\alpha, i)} \ar{r} \ar{d} & \Def_L \ar{d} \\ h_K \ar{r} & \Gamma \ar{r} & \Def_M \end{tikzcd} \]
	Here $i$ denotes the inclusion $H^0(M) \hookrightarrow M$. Note that $\Gamma = \Def_{H^0(M)}$ by \cref{residual gerbe governed by h0} and clearly $h_K = \Def_0$, so \cref{fiber product} implies that
	\[ \Def_{(\alpha, 0)} = \Def_L \times_{\Def_M} h_K \text{ and } \Def_{(\alpha, i)} = \Def_L \times_{\Def_M} \Gamma. \]
	Since $h_K \to \Gamma$ is essentially surjective, its pullback $\Def_{(\alpha, 0)} \to \Def_{(\alpha, i)}$ is also essentially surjective. 
	
	Now define the following.
	\[ \begin{aligned}
	C^+ &= \Cone(\alpha : L \to M)[-1] \\
	C &= \Cone(\alpha - i : L \oplus H^0(M) \to M)[-1] \\
	\end{aligned} \]
	Then we have two distinguished triangles as in \cref{producing a long exact sequence}, and a morphism between them as follows.
	\[ \begin{tikzcd} C^+ \ar{r} \ar[dotted]{d} & L \ar{r}{\alpha} \ar{d}[swap]{(1, 0)} & M \ar{d}{1} \ar{r}{+} & \mbox{} \\
	C \ar{r} & L \oplus H^0(M) \ar{r}{\alpha - i} & M \ar{r}{+} & \mbox{} \end{tikzcd} \]
	We then get a morphism of long exact sequences.
	\[ \begin{aligned} \begin{tikzcd} 0 \ar{r} & H^0(C^+) \ar{r} \ar{d} & H^0(L) \ar{r}{\alpha} \ar{d}{(1,0)} & H^0(M) \ar{r} \ar{d}{1} & \dotsb \\
	0 \ar{r} & H^0(C) \ar{r} & H^0(L) \oplus H^0(M) \ar{r}{\alpha - 1} & H^0(M) \ar{r} & \dotsb \end{tikzcd} \\
	\mbox{} \\
	\begin{tikzcd}
	\dotsb \ar{r} & H^0(M) \ar{r} \ar{d}{1} & H^1(C^+) \ar{d} \ar{r} & H^1(L) \ar{d}{1} \ar{r}{\alpha} & H^1(M) \ar{r} \ar{d}{1} & \dotsb  \\
	\dotsb \ar{r} & H^0(M) \ar{r} & H^1(C) \ar{r} & H^1(L) \ar{r}{\alpha} & H^1(M) \ar{r} & \dotsb
	\end{tikzcd}
	\end{aligned} \]
	Clearly we have $H^i(C^+) \simeq H^i(C)$ for all $i \geq 2$. Note moreover that $i : H^0(M) \hookrightarrow M$ and $0 \hookrightarrow M$ are both zero in degree 1, so $H^2(C) \simeq H^2(C^+)$ is an obstruction space for $\Def_{(\alpha, i)}$ and $\Def_{(\alpha, 0)}$, respectively, using \cref{deformations and fiber product}. 
	
	Consider the map $\alpha - 1 : H^0(L) \oplus H^0(M) \to H^0(M)$. Its kernel is $H^0(C)$, so
	\[ \begin{aligned} \Inf(\Def_{(\alpha, \iota)}) &= H^0(C) \\
	&= \ker(\alpha - 1 : H^0(L) \oplus H^0(M) \to H^0(M)) \\
	&= \{ (x, \alpha(x)) : x \in H^0(L) \} \\
	&= H^0(L) \\
	&= \Inf(\Def_L). \end{aligned} \]
	Moreover, clearly $\alpha - 1$ is surjective. This means that the connecting map $H^0(M) \to H^1(C)$ is the zero map, so
	\[ \begin{aligned} T(\Def_{(\alpha, \iota)}) &= H^1(C) \\
	&= \ker(\alpha : H^1(L) \to H^1(M)) \\
	&= \im(H^1(C^+) \to H^1(L)) \\
	&= \im(T(\Def_{(\alpha, 0)}) \to T(\Def_L)). \end{aligned} \]
	In particular, the map on tangent spaces induced by $\Def_{(\alpha, 0)} \to \Def_{(\alpha, i)}$ is surjective. It is also straightforward to check directly from the construction of obstruction classes that the isomorphism $H^2(C^+) \simeq H^2(C)$ is compatible with $\Def_{(\alpha, 0)} \to \Def_{(\alpha, i)}$, so the standard smoothness criterion \cref{result:standard smoothness criterion} implies that $\Def_{(\alpha, 0)} \to \Def_{(\alpha, i)}$ is smooth.
\end{example}

	\section{Deformations of modules}  \label{chapter:deformations of modules}

Throughout this section, let $K$ be a field of characteristic 0, $D$ an associative $K$-algebra, and $E$ a left $D$-module. For any $R \in \Art_K$, we let $D_R := R \otimes_K D$ and $E_R := R \otimes_K E$. In this chapter, we study the following deformation problem. 

\begin{definition}
	Let $\Def_{D,E}$ be the category of tuples $(R, F, \theta)$ where $R \in \Art_K$, $F$ is a left $D_R$-module flat over $R$, and $\theta : F \to E$ is a left $D_R$-module homomorphism such that the induced homomorphism $K \otimes_R F \to E$ of left $D$-modules is an isomorphism.
	Morphisms $(R', F', \theta') \to (R, F, \theta)$ in $\Def_{D,E}$ are pairs $(\pi, u)$ consisting of a homomorphism $\pi : R' \to R$ in $\Art_K$ and a homomorphism $u : F' \to F$ of left $D_{R'}$-modules such that the corresponding left $D_R$-module homomorphism $R \otimes_{R'} F' \to F$ is an isomorphism, and such that $\theta \circ u = \theta'$.
	\[ \begin{tikzcd} F' \ar{r}{u} \ar[bend right]{dr}[swap]{\theta'} & F \ar{d}{\theta} \\ & E \end{tikzcd} \]
	When $D$ can be inferred from context, we will write $\Def_E$ instead of $\Def_{D,E}$. The forgetful functor $\Def_E \to \Art_K$ defined by $(R, E) \mapsto R$ presents $\Def_{E}$ as an opfibration in groupoids over $\Art_K$. It is straightforward to check directly that $\Def_E$ is a deformation category, but in any case this is a consequence of \cref{result:hochschild governs module deformations} below. For $R \in \Art_K$, we will abusively write 1 for the canonical map $E_R \to E$. Regarding $E_R$ as a left $D_R$-module in the natural way, the triple $(R, E_R, 1)$ becomes an object of $\Def_E$.
\end{definition}

\subsection{Hochschild complex}

For any $D$-bimodule $P$, let $\Hoch_K(D,P)$ denote the Hochschild cochain complex, so
\[ \Hoch^p_K(D, P) = \Hom_K(D^{\otimes p}, P) \]
for all non-negative integers $p$, where $D^{\otimes p}$ is the $p$-fold tensor product of $D$ over $K$. The differential is defined via a simplicial construction (cf. \cite[chapter 9]{weibel} for details). When $P$ is a $K$-algebra equipped with a $K$-algebra homomorphism $D \to P$, the cup product makes $\Hoch_K(D,P)$ a differential graded $K$-algebra (cf. \cite[section 7, page 278]{gerstenhaber}). 

In particular, $P = \End_K(E)$ is a $K$-algebra under composition equipped with a homomorphism $D \to \End_K(E)$, so $\Hoch_K(D, \End_K(E))$ is a differential graded $K$-algebra. We can then regard this as a differential graded $K$-Lie algebra under the graded commutator bracket.

\begin{theorem} \label{result:hochschild governs module deformations}
	$\Hoch_K(D, \End_K(E))$ governs $\Def_{E}$.
\end{theorem}

\begin{proof}
	Let $L := \Hoch_K(D, \End_K(E))$. We are trying to construct an equivalence of categories  
	\[ \begin{tikzcd} \Def_{L} \ar{r}{\Theta} & \Def_{E} \end{tikzcd} \]
	over $\Art_K$. Observe that, for any $R \in \Art_K$, 
	\[ L_R := R \otimes_K L = \Hoch_R(D_R, \End_R(E_R)). \]
	In degree 1, this complex contains a canonical element $s_R : D_R \to \End_R(E_R)$ which is the $R$-algebra homomorphism defining the natural $D_R$-module structure on $E_R$. We can compute that $\phi \in L_R^1$ satisfies the Maurer-Cartan equation $d\phi + \frac{1}{2}[\phi,\phi] = 0$ if and only if
	\[ \phi(d_1)d_2 + d_1\phi(d_2) + \phi(d_1)\phi(d_2) = \phi(d_1d_2) \]
	for all $d_1, d_2 \in D_R$, if and only if the $R$-module homomorphism $\phi + s_R$ is actually an $R$-algebra homomorphism $D_R \to \End_R(E_R)$. Then note that
	\[ \m_R \otimes_K L = \ker(L_R \to L) \]
	so $\phi \in \m_R \otimes_K L^1$ if and only if $\phi + s_R$ maps to the $D$-module structure map $s_K \in L^1 = \Hom_K(D, \End_K(E))$.
	
	In other words, we conclude that objects $\Def_L(R)$ are in bijection with left $D_R$-module structures on $E_R$ which reduce to the given left $D$-module structure on $E$. If $\phi \in \Def_L(R)$, we write $E_{R,\phi}$ for $E_R$ regarded as a left $D_R$-module via $\phi$. Then $(R, \phi) \mapsto (R, E_{R,\phi}, 1)$ defines the functor $\Theta : \Def_L \to \Def_{E}$ on the level of objects.
	Since any $(R, F, \theta) \in \Def_E$ must have $F$ isomorphic \emph{as an $R$-module} to $E_R$, this also shows that the functor $\Theta$, once we have finished constructing it, must be essentially surjective.
	
	Next up, let's compute the gauge action. Observe that
	\[ \m_R \otimes_K L^0 = \m_R \otimes_K \End_K(E) = \ker(\End_R(E_R) \to \End_K(E)). \]
	If $\eta \in \m_R \otimes_K L^0$, then one checks that $d\eta = -[\eta, s_R]$, so for $\phi \in L^1_R$, we have
	\begin{align*}
	\eta * \phi &= \phi + \sum_{n = 0}^{\infty} \frac{[\eta, -]^n}{(n+1)!}([\eta, \phi] - d\eta) \\
	&= \phi + \sum_{n = 0}^\infty \frac{[\eta, -]^n}{(n+1)!} ([\eta, \phi + s_R]) \\
	&= \phi + \sum_{n = 0}^\infty \frac{[\eta, -]^{n+1}}{(n+1)!}(\phi + s_R) \\
	&= e^{[\eta, -]}(\phi + s_R) - s_R \\
	&= e^\eta(\phi + s_R)e^{-\eta} - s_R
	\end{align*}
	where $e^{[\eta, -]}$ and $e^\eta$ denote exponentiation of nilpotent endomorphisms. Note that $e^\eta$ is an $R$-module automorphism of $E_R$ with inverse $e^{-\eta}$, and the last equality is \cref{result:exponential of commutator} below.
	
	Now suppose that $(\pi, \eta)$ is a morphism $(R', \phi') \to (R, \phi)$ in $\Def_L$. In other words, $\pi : R' \to R$ is a homomorphism in $\Art_K$ and $\eta \in \m_R \otimes_K L^0$ satisfies
	\[ \eta * \pi(\phi') = \phi. \]
	We define $\Theta(\pi, \eta)$ to be the pair consisting of $\pi$ together with the following composite.
	\[ \begin{tikzcd} E_{R',\phi'} \ar{r}{\pi \otimes 1} & E_{R,\pi(\phi')} \ar{r}{e^\eta} & E_{R,\phi}. \end{tikzcd} \]
	This composite is a $D_R$-module homomorphism if and only if $e^\eta : E_{R,\pi(\phi')} \to E_{R,\phi}$ is a left $D_R$-module homomorphism, if and only if the following diagram commutes.
	\[ \begin{tikzcd} D_R \ar{r}{\pi(\phi') + s_R} \ar{d}[swap]{\phi + s_R} & \End_R(E_R) \ar{d}{e^\eta \cdot } \\ \End_R(E_R) \ar{r}[swap]{\cdot e^\eta} & \End_R(E_R) \end{tikzcd} \]
	This commutativity is equivalent to
	\[ (\phi + s_R)e^\eta = e^\eta (\pi(\phi') + s_R) \]
	which is exactly the condition $\eta * \pi(\phi') = \phi$. Finally, we observe that $e^\eta \equiv 1 \bmod{\m_R}$, so $\Theta(\pi, \eta)$ is indeed a morphism $\Theta(R', \phi') \to \Theta(R, \phi)$ in $\Def_{E}$. Now if $(\pi', \eta') : (R'', \phi'') \to (R', \phi')$ is another morphism in $\Def_L$, we need to check that the following diagram commutes.
	\[ \begin{tikzcd} E_{R, \pi'(\pi(\phi''))} \ar{r}{e^{\pi(\eta')}} \ar[bend right]{dr}[swap]{e^{\eta * \pi(\eta')}} & E_{R, \pi(\phi')} \ar{d}{e^{\eta}} \\ & E_{R, \phi} \end{tikzcd} \]
	But note the equality
	\[ \log(e^\eta e^{\pi(\eta')}) = \eta *  \pi(\eta') \]
	so taking the exponential of both sides yields the desired equality. This completes our construction of the functor $\Theta : \Def_L \to \Def_{E}$.
	
	To complete the proof, we must check that $\Theta$ is fully faithful. This requires showing that, for any $R \in \Art_K$ and any $\phi, \phi' \in \Def_L(R)$, every left $D_R$-module isomorphism $\alpha : E_{R,\phi'} \to E_{R,\phi}$ reducing to the identity on $E$ modulo $\m_R$ is of the form $e^\eta$ for some $\eta \in \End_R(E_R)$ satisfying $\eta * \phi' = \phi$, and moreover that there is only one such $\eta$. We clearly must take
	\[ \eta = \log(\alpha) = \log(1 + (\alpha - 1)) = \sum_{n = 1}^\infty (-1)^{n+1} \frac{(\alpha - 1)^n}{n}, \]
	which is well-defined since $\alpha - 1$ is nilpotent. The fact that $\eta* \phi' = \phi$ is then a translation of the fact that $\alpha$ is a left $D_R$-module homomorphism, as we saw above.
\end{proof}

\begin{lemma} \label{result:exponential of commutator}
	Let $R$ be a commutative ring and $A$ an $R$-algebra. Suppose that $a, b \in A$ and that $a$ is nilpotent. Then $[a, -]$ is a nilpotent $R$-module endomorphism of $A$ and
	\[ e^a b a^{-a}  = \sum_{n = 0}^\infty \sum_{k = 0}^n \frac{(-1)^k a^{n-k}ba^k}{(n-k)!k!} = e^{[a, -]}(b).  \]
\end{lemma}

\begin{proof}
	Both equalities are direct calculations.
\end{proof}

\begin{corollary} \label{result:fundamental theorem of deformation of modules}
	\begin{enumerate}[label=(\alph{*})]
		\item $\Inf(\Def_{E}) = \End_D(E)$.
		\item $T(\Def_{E}) = \Ext^1_D(E,E)$.
		\item $\Ext^2_D(E,E)$ is an obstruction space for $\Def_{E}$.
	\end{enumerate}
\end{corollary}

\begin{proof}
	This follows immediately from \cref{result:hochschild governs module deformations,result:fundamental theorem of dglas}, and the fact that
	\[ H^i(\Hoch_K(D, \End_K(E))) = \Ext^i_{D/K}(E, E) = \Ext^i_D(E, E) \]
	for all $i$. Here, we use \cite[lemma 9.1.9]{weibel} to calculate the cohomology of the Hochschild complex, and the fact that $K$ is a field to identify $\Ext^i_{D/K}$ with $\Ext^i_D$. 
\end{proof}

\begin{remark} \label{remark:sheaf theoretic upgrade}
	In light of the fundamental theorem of deformation theory \ref{deformation theory}, we see that $\Def_E$ has a hull as long as $\Ext^1_D(E, E)$ is finite dimensional. We will also have a prorepresentability result when $\End_D(E) = K$; cf. \cref{automorphisms lift for endomorphically simples} below. Taking $D$ to be commutative, these results would give us statements about deformations of coherent sheaves on affine $K$-schemes, but we rarely have finite dimensionality of $\Ext^1_D(E,E)$ over $K$ when $D$ is commutative. 
	
	However, it is quite likely that there is a generalization of \cref{result:hochschild governs module deformations} with $D$ an associative algebra in a general topos. Then, taking $D$ to be the structure sheaf of a proper $K$-scheme would allow us to recover standard results about the deformation theory of coherent sheaves on proper $K$-schemes (cf. \cite[theorem 3.6--8]{nitsure_deformation}). We do not pursue this generalization here since this will not be necessary for our intended applications. 
\end{remark}

\begin{remark}  \label{remarks on dimension 1}
	Suppose that
	\[ H^i( \Hoch_K(D, \End_K(E))) = \Ext^i_D(E,E) = 0 \]
	for all $i \geq 2$. For example, this could be because $E$ has projective dimension at most 1 as a left $D$-module, or, more strongly, because $D$ has left global dimension at most 1. In this situation, evidently the natural inclusion $\tau_{\leq 1}\Hoch_K(D, \End_K(E)) \to \Hoch_K(D, \End_K(E))$ is a quasi-isomorphism. Moreover, $\tau_{\leq 1}\Hoch_K(D, \End_K(E))$ is a differential graded subalgebra of $\Hoch_K(D, \End_K(E))$. Since $L \mapsto \Def_L$ factors through quasi-isomorphisms of differential graded Lie algberas by the equivalence \cref{result:quasi-isomorphism invariance}, we see that the two-term differential graded Lie algebra
	\[ \tau_{\leq 1}\Hoch_K(D, \End_K(E)) \]
	governs $\Def_{E}$. Recall the explicit description of $\tau_{\leq 1}\Hoch_K(D, \End_K(E))$. 
	\[ \begin{tikzcd} \End_K(E) \ar{r}{d} & \Der_K(D, \End_K(E)) \ar{r} & 0 \ar{r} & \dotsb \end{tikzcd} \]
	Here $\Der_K(D, \End_K(E))$ are \emph{derivations}, i.e. $K$-linear maps $s : D \to \End_K(E)$ satisfying the Leibniz rule $s(ab) = s(a)b + as(b)$ for all $a, b \in D$. The differential $d$ is given by 
	\[ d(\rho)(a) = a\rho - \rho a = [a, \rho] \]
	for $\rho \in \End_K(E)$ and $a \in D$  \cite[section 9.2]{weibel}. The image of the differential $d$ is the set of \emph{principal derivations}, denoted $\PDer_K(D, \End_K(E))$ \cite[definition 6.4.2]{weibel}. Multiplication is simply composition: in degree 0 it is composition of $K$-endomorphisms of $E$, and the multiplication maps
	\[ \begin{tikzcd} \End_K(E) \times \Der_K(D, \End_K(E)) \ar{r} & \Der_K(D, \End_K(E))  \\
	\Der_K(D, \End_K(E)) \times \End_K(E) \ar{r} & \Der_K(D, \End_K(E)) \end{tikzcd} \]
	are also composition: the first takes $(\rho, \delta)$ to the derivation $d \mapsto \rho \circ \delta(d)$, and the second takes $(\delta, \rho)$ to the derivation $d \mapsto \delta(d) \circ \rho$.
\end{remark}

\subsection{Derivations and extensions} \label{derivations and extensions}

If $E$ and $F$ are left $D$-modules, note that $\Hom_K(E, F)$ is a $D$-bimodule and 
\[ H^1(\Hoch_K(D, \Hom_K(E, F))) = \Ext^1_D(E, F) \]
using \cite[lemma 9.1.9]{weibel}. Let us work out explicitly how to regard Hochschild cohomology classes on the left-hand side as extensions of $E$ by $F$ on the right-hand side. We will apply this in the case when $E = F$, but it is less confusing to work in greater generality. 

For any $s \in \Hoch_K^1(D, \Hom_K(E, F)) = \Hom_K(D, \Hom_K(E, F))$, let $F \oplus_s E$ denote the vector space $F \oplus E$ endowed with an ``action'' of $D$ by the formula 
\[ a \cdot (f, e) = (af + s(a)(e), ae) \]
for all $a \in D$, $f \in F$ and $e \in E$. If $b$ is another element of $D$, then 
\[ ab \cdot (f, e) = a \cdot (b \cdot (f, e)) \text{ if and only if } s(ab) = s(a)b + a s(b) \]
so $F \oplus_s E$ is a left $D$-module if and only if $s$ is a derivation. When this is the case, note that the inclusion $F \hookrightarrow F \oplus_s E$ and the projection $F \oplus_s E \to E$ are both $D$-linear, so $F \oplus_s E$ is an extension of $E$ by $F$. In other words, we have defined a map
\[ \begin{tikzcd} \Der_K(D, \Hom_K(E, F)) \ar{r}{\alpha} & \Ext^1_D(E, F). \end{tikzcd} \]
We now claim that $\alpha$ is $K$-linear, that it is surjective, and that its kernel is exactly the set of principal derivations $\PDer_K(D, \Hom_K(E, F))$.

\subsubsection{Linearity}

If $s$ and $t$ are two derivations, we first need to show that the Baer sum of $F \oplus_s E$ and $F \oplus_t E$ is isomorphic to $F \oplus_{s + t} E$. Note that the pullback of
\[ \begin{tikzcd} & F \oplus_s E \ar{d} \\ F \oplus_t E \ar{r} &  E \end{tikzcd} \]
is isomorphic to $F \oplus F \oplus E$ as a vector space, and under this isomorphism, the action of $D$ is defined by
\[ a \cdot (f, f', e) = (af + s(a)(e), af' + t(a)(e), ae). \]
The Baer sum is defined to be the quotient of this pullback by the submodule \[ \{(f, -f, 0) : f \in F\}. \] Observe that the map from $F \oplus_{s + t} E$ into this quotient given by $(f, e) \mapsto (f, 0, e)$ is clearly $K$-linear and bijective; in fact, it is even $D$-linear, because
\[ \begin{aligned} a \cdot (f, 0, e) &= (af + s(a)(e), t(a)(e), e) \\
&= (af +  s(a)(e), 0, e) + (0, t(a)(e), 0) \\
&= (af + s(a)(e), 0, e) + (t(a)(e), 0, 0) \\
&= (af + (s+t)(a)(e), 0, e). \end{aligned} \]
This proves that $\alpha$ commutes with addition. 

To show that it commutes with scalar multiplication, we need to show that for any $\lambda \in K$, $F \oplus_{\lambda s} E$ is isomorphic to the pushout of the following. 
\[ \begin{tikzcd} F \ar{r} \ar{d}[swap]{\lambda} & F \oplus_s E \\ F \end{tikzcd} \]
More precisely, consider the map from $F \oplus_{\lambda s} E$ into this pushout given by $(f, e)  \mapsto (f, 0, e)$. This map is certainly $K$-linear and bijective. If $\lambda = 0$, $D$-linearity is clear. For $\lambda$ nonzero, we note the following.
\[ \begin{aligned} a \cdot (f, 0, e) &= a \cdot (0, \lambda^{-1}f, e) \\ 
&= (0, \lambda^{-1}af + s(a)(e), ae) \\
&= (af + (\lambda s)(a)(e), 0, ae). \end{aligned}  \]

\subsubsection{Surjectivity} 

Suppose $Q$ is an extension of $E$ by $F$. 
\[ \begin{tikzcd}
0 \ar{r} & F \ar{r} & Q \ar{r} & E \ar{r} & 0
\end{tikzcd} \]
This exact sequence splits in the category of vector spaces, so, after fixing $K$-linear splittings, we have $Q = F \oplus E$ as a vector space. For any $a \in D$ and $e \in E$, define $s(a)(e)$ to be the projection of $a \cdot (0, e)$ onto $F$. Then for any $f \in F$ we see that
\[ a \cdot (f, e) = a \cdot (f, 0) + a \cdot (0, e) = (af, 0) + (s(a)(e), ae) = (af + s(a)(e), ae) \]
using the fact that $F \to Q$ and $Q \to E$ are $D$-linear. Since $Q$ is a left $D$-module, we know from above that $s$ must be a derivation and $Q = F \oplus_s E$ as extensions of $E$ by $F$. This proves surjectivity. 

\subsubsection{Kernel is principal derivations}

Suppose that $s \in \Der_K(D, \Hom_K(E, F))$ and $F \oplus_s E$ is a trivial extension of $E$ by $F$. Choose an isomorphism $\phi :  F \oplus_s E \to F \oplus E$ of extensions. 
\[ \begin{tikzcd} 0 \ar{r} & F \ar{r} \ar[equals]{d} & F \oplus_s E \ar{r} \ar{d}{\phi} & E \ar{r} \ar[equals]{d} & 0 \\ 0 \ar{r} & F \ar{r} & F \oplus E \ar{r} & E \ar{r} & 0 \end{tikzcd} \]
The commutativity of this diagram means that $\phi$ must be given by $\phi(f, e) = (f + \sigma(e), e)$ for some $\sigma \in \Hom_K(E, F)$. Since $\phi$ is $D$-linear, for any $a \in D$ we have
\begin{align*}
(a\sigma(e), ae) &= a(\sigma(e), e) \\
&= a\phi(0, e) \\
&= \phi(a \cdot (0, e)) \\
&= \phi(s(a)(e), ae) \\
&= (s(a)(e) + \sigma(ae), ae)
\end{align*}
which means that
\[ s(a)(e) = a\sigma(e) - \sigma(ae) \]
or, in other words,
\[ s(a) = a\sigma - \sigma a = [a, \sigma] = d\sigma(a) \]
for all $a \in D$. In other words, we have $s = d\sigma$, so $s$ is principal. Conversely, it is also clear from this calculation that any principal derivation does in fact give rise to a trivial extension of $E$ by $F$. 

\subsection{Prorepresentability for irreducibles}

\begin{proposition} \label{ends-and-auts}
	Suppose $\End_D(E) = K$. Then for any $(R, F, \theta) \in \Def_E$, we have 
	\[ \End_{D_R}(F) = R \text{ and } \Aut_R(F, \theta) = 1 + \m_R. \]
\end{proposition}

\begin{proof}
	Note that the latter assertion follows from the former. Since the quotient map $R \to R/\m_R = K$ in $\Art_K$ can be factored into a series of small surjections \tag{06GE}, it suffices to prove the following: whenever $\pi : R \to R'$ is a small surjection and $(R, F, \theta) \to (R', F', \theta')$ is a map in $\Def_E$ lying over $\pi$ such that $\End_{D_{R'}}(F') = R'$, then $\End_{D_R}(F) = R$. 
	
	Suppose $\phi \in \End_{D_R}(F)$. Then the induced $D_{R'}$-linear endomorphism of $F'$ is a scalar in $R'$, which means that there exists $a \in R$ such that $\phi - a \in \End_{D_R}(F)$ induces the zero endomorphism of $F'$. In other words, if we define $\psi := \phi - a$, we have $\im(\psi) \subseteq IF$, where $I := \ker(\pi)$. This means that
	\[ \psi(\m_R F) \subseteq \m_R \im(\psi) \subseteq \m_R IF = 0 \]
	so $\psi$ naturally factors through a $\bar{\psi} : E = F/\m_R F \to IF$. 
	
	Since $\pi : R \to R'$ is a small surjection, its kernel $I$ is principally generated by some $\epsilon \in R$ such that $\epsilon \m_R = 0$. Since every element of $R$ can be written uniquely as $a + \eta$ with $a \in K$ and $\eta \in \m_R$, we have 
	\[ I = \{a\epsilon : a \in K \}. \]
	In other words, the map $K \to I$ given by $1 \mapsto \epsilon$ is an isomorphism of $R$-modules. Tensoring with $F$, we have an isomorphism of $D$-modules
	\[ \begin{tikzcd} E = K \otimes_R F \ar{r}{\sigma} & I \otimes_R F = IF. \end{tikzcd} \]
	Then $\sigma^{-1} \circ \bar{\psi} \in \End_D(E) = K$, so define $b := \sigma^{-1} \circ \bar{\psi}$. Observe that
	\[ \psi(f) = \bar{\psi}(f \bmod {\m_R F}) = (\sigma \circ \sigma^{-1} \circ \bar{\psi})(f \bmod{\m_R F}) = \sigma(b f \bmod{\m_R F}) = b\epsilon f  \]
	for any $f \in F$, which means that $\psi = b\epsilon$ and therefore $\phi = a + b\epsilon \in R$. 
\end{proof}

\begin{corollary} \label{automorphisms lift for endomorphically simples}
	Suppose $\End_D(E) = K$. If $(R', F', \theta') \to (R, F, \theta)$ in $\Def_E$ lies over a surjective map $R' \to R$, then $\Aut(F', \theta') \to \Aut(F, \theta)$ is also surjective. Thus, if $\End_D(E) = K$ and $\Ext^1_D(E,E)$ is finite dimensional, then $\overline{\Def}_E$ is prorepresentable. 
\end{corollary}

\begin{proof}
	The assertion that $\Aut(F', \theta') \to \Aut(F, \theta)$ is surjective follows from the characterization of the automorphism groups in \cref{ends-and-auts}. Then, under the addition hypothesis that $\Ext^1_D(E,E)$ is finite dimensional, prorepresentability follows from the fundamental theorem of deformation theory \ref{deformation theory}. 
\end{proof}

Usually, the condition that $\End_D(E) = K$ is a consequence of the following. 

\begin{definition}
	$E$ is \emph{absolutely irreducible} if $E_{\bar{K}}$ is irreducible over $D_{\bar{K}}$, where $\bar{K}$ is an algebraic closure of $K$. 
\end{definition}

\begin{lemma} \label{endomorphisms of absolutely irreducibles}
	Suppose that $E$ has all of the following properties. 
	\begin{enumerate}[label=(\roman{*})]
		\item $\End_D(E)$ is finite dimensional over $K$. 
		\item $E$ is finitely presented over $D$. 
		\item $E$ is absolutely irreducible.
	\end{enumerate}
	Then $\End_D(E) = K$. 
\end{lemma}

\begin{proof}
	We know that $\End_{D_{\bar{K}}}(E_{\bar{K}}) = \bar{K}$ by Schur's lemma, so we have
	\[ \dim_K \End_D(E) = \dim_{\bar{K}}( \bar{K} \otimes_K \End_D(E)) = \dim_{\bar{K}} \End_{D_{\bar{K}}}(E_{\bar{K}}) = 1 \]
	using \cref{result:finite presentation and endomorphisms} below for the second equality. Thus $\End_D(E) = K$. 
\end{proof}

\begin{lemma} \label{result:finite presentation and endomorphisms}
	Suppose $E$ is finitely presented and $F$ is a left $D$-module. For every commutative $K$-algebra $P$, we have 
	\[ P \otimes_K \Hom_D(E, F) = \Hom_{D_P}(E_P, F_P). \]
\end{lemma}

\begin{proof}
	Observe that $\Hom_{D_P}(E_P, F_P) = \Hom_D(E, F_P)$. The map $F \to F_P$ induces a morphism $\Hom_D(-, F) \to \Hom_D(-, F_P)$ of contravariant left-exact functors $\Mod_D \to \Mod_K$. But observe that $\Hom_D(-, F_P)$ naturally takes values in $\Mod_P$, so in fact we have a morphism $\eta : P \otimes_K \Hom_D(-, F) \to \Hom_D(-, F_P)$. Since $P \otimes_K -$ is exact, this is a morphism of contravariant left-exact functors $\Mod_D \to \Mod_P$. We want to show that $\eta_E$ is an isomorphism, but using left-exactness of both functors and the fact that $E$ is finitely presented, it suffices to show that $\eta_D$ is an isomorphism. This is clear.
\end{proof}

\begin{lemma} \label{geometric irreducibility finite extensions}
	The following are equivalent.
	\begin{enumerate}[label=(\alph{*})]
		\item $E_{\bar{K}}$ is irreducible over $D_{\bar{K}}$, where $\bar{K}$ is an algebraic closure of $K$. 
		\item $E_L$ is irreducible over $D_L$ for any finite extension $L$ of $K$. 
	\end{enumerate}
\end{lemma}

\begin{proof}
	Suppose $L$ is a finite extension of $K$, and choose an embedding $L \hookrightarrow \bar{K}$. Suppose $F$ is a nonzero $D_L$-submodule of $E_L$. Since $\bar{K}$ is faithfully flat over $L$, we see that $F_{\bar{K}}$ is a nonzero submodule of $E_{\bar{K}}$. This means that $E_{\bar{K}}/F_{\bar{K}} = (E_L/F)_{\bar{K}} = 0$, which means that $F = E_L$. This shows (a) implies (b).
	
	For the converse, note that $\bar{K}$ is the colimit of all subextensions $L$ finite over $K$, so $D_{\bar{K}} = \colim D_L$ and $E_{\bar{K}} = \colim E_L$. Suppose $e \in E_{\bar{K}}$ is nonzero. Then there exists an $L_0$ finite over  $K$ such that $e \in E_{L_0}$. Then for every $L$ finite over $L_0$, note that $L$ is finite over $K$ also, so $E_L$ is irreducible and $D_L e = E_L$. Since the finite extensions of $L_0$ are cofinal in the partially ordered set of finite extensions of $K$ in $\bar{K}$, this shows that $D_{\bar{K}} e = E_{\bar{K}}$. 
\end{proof}

	\section{Deformations of differential modules} \label{chapter:deformations of differential modules}

In this section, $K$ is a field of characteristic 0 and $\O$ a commutative $K$-algebra equipped with a $K$-linear derivation $\partial$. A \emph{differential module} over $\O$ is an $\O$-module $E$ equipped with a $K$-linear map $\partial : E \to E$ satisfying the Leibniz rule \[ \partial(ae) = \partial(a)e + a\partial(e). \] 
A morphism between differential modules $\phi : E \to F$ is an $\O$-module homomorphism such that $\partial \circ \phi = \phi \circ \partial$. 
\[ \begin{tikzcd} E \ar{r}{\partial} \ar{d}[swap]{\phi} & E \ar{d}{\phi} \\ F \ar{r}[swap]{\partial} & F \end{tikzcd} \]
Let $\DMod_\O$ denote the category of differential modules over $\O$. We study the deformations of a finite free differential $\O$-module. 

\subsection{Preliminaries} \label{section:deformations of differential modules preliminiaries}

The category $\DMod_\O$ of differential $\O$-modules is naturally a $K$-linear tensor category \cite[definition 5.3.2]{kedlaya_pde}.\footnotemark\ Given two differential $\O$-modules $E$ and $E'$, the internal hom $\Hom(E,E')$ is just $\Hom_\O(E, E')$, with differential $\O$-module structure determined by the equation
\[ (\partial \cdot \phi)(e) = \partial(\phi(e)) - \phi(\partial(e)) \]
for $\phi \in \Hom_\O(E,E')$ and $e \in E$. 

\footnotetext{\label{footnote:tensor category}For us, \emph{tensor category} will mean a closed symmetric monoidal abelian category in which the monoidal product is right exact in each argument. We will say that it is \emph{compact closed} if each of its objects is dualizable and the monoidal product is exact in both arguments.}

Let $D = \O[\partial]$ be the corresponding ring of differential operators: elements of $D$ can be written uniquely in the form $f\partial^i$ where $f \in \O$ and $i$ is a  nonnegative integer, and multiplication is determined by the equation
\[ [\partial, f] = \partial(f) \]
for all $f \in \O$. Then $\DMod_\O$ is naturally equivalent to the category $\Mod_D$ of left $D$-modules. 

For more about all of this, see \cite[chapter 5]{kedlaya_pde}.

\begin{definition}[de Rham complex]
	If $E$ is a differential $\O$-module, we define the \emph{de Rham complex} of $E$, denoted $\dR(\O,E)$, be the following two-term chain complex in nonnegative degrees.
	\[ \begin{tikzcd} E \ar{r}{\partial} & E \ar{r} & 0 \ar{r} & \dotsb \end{tikzcd} \]
	We then define $H^i_{\dR}(E) := H^i(\dR(\O,E))$. 
\end{definition}

\begin{example}
	It follows from definitions that $H^0_\dR(\Hom(E, E')) = \Hom_D(E,E')$. 
\end{example}

\begin{lemma}[{\cite[section 6.6]{christol}, \cite[proof of proposition 2]{goodearl}}] \label{result:resolve a finite free differential module}
	Let $E$ be a finite free differential $\O$-module and $(e_1, \dotsc, e_n)$ an $\O$-basis for $E$. Let $N$ be the corresponding matrix of $\partial$. Then we have an exact sequence of left $D$-modules
	\[ \begin{tikzcd} 0 \ar{r} & D^{\oplus n} \ar{r}{\partial I - N} & D^{\oplus n} \ar{r} &  E \ar{r} & 0 \end{tikzcd} \]
	where the map $D^{\oplus n} \to E$ carries the standard $D$-basis of $D^{\oplus n}$ onto the $\O$-basis $(e_1, \dotsc, e_n)$ of $E$, and $\partial I$ denotes the diagonal matrix with $\partial$ in all of the diagonal entries
\end{lemma}

\begin{corollary}
	Suppose $E$ is a finite free differential $\O$-module. Then $E$ is finitely presented as a left $D$-module and has projective dimension at most 1.  \qed
\end{corollary}

\begin{corollary} \label{result:de rham complex is rhom}
	$\rightderived\!\Hom_D(\O, E) = \dR(\O, E)$ for any differential $\O$-module $E$. 
\end{corollary}

\begin{proof}
	Consider the free resolution of $\O$ as a $D$-module provided by \cref{result:resolve a finite free differential module}. Applying the functor $\Hom_D(-, E)$ to this free resolution gives exactly the de Rham complex $\dR(\O, E)$.  
\end{proof}

\subsection{de Rham and Hochschild complexes} \label{de rham and hochschild}

For the remainder, we fix a finite free differential $\O$-module $E$. When we write $\Def_E$, we will mean $\Def_{D,E}$ (as opposed to $\Def_{\O,E}$).

\begin{remark}[Lifting bases] \label{lifting bases}
	Suppose $(e_1, \dotsc, e_n)$ is an $\O$-basis for $E$ and suppose $(R, F, \theta) \in \Def_{E}$. Choose $f_i \in F$ such that $\theta(f_i) = e_i$ for all $i = 1, \dotsc, n$. Then the tuple $(f_1, \dotsc, f_n)$ defines an $\O_R$-module homomorphism $\phi : \O_R^{\oplus n} \to F$ and the composite
	\[ \begin{tikzcd} \O^{\oplus n} \ar{r}{1 \otimes \phi} & K \otimes_R F \ar{r}{1 \otimes \theta} & E \end{tikzcd} \]
	is clearly an isomorphism. Since $K \otimes_R F \to E$ is an isomorphism, we see that $\O^{\oplus n} \to K \otimes_R F$ must also be an isomorphism. Since $F$ is $R$-flat, $\phi$ is itself an isomorphism \cite[lemma 3.3]{schlessinger}. In other words, $F$ is a finite free differential $\O_R$-module with $\O_R$-basis $(f_1, \dotsc, f_n)$. Hereafter, a basis $(f_1, \dotsc, f_n)$ of $F$ obtained in this way will be said to be a \emph{lift} of the basis $(e_1, \dotsc, e_n)$ of $E$.
\end{remark}

Observe that $\dR(\O, \End(E))$ a differential graded $K$-algebra under composition. Our goal now is to show that $\dR(\O, \End(E))$, regarded as a differential graded $K$-Lie algebra, governs $\Def_E$. We will do this by producing a quasi-isomorphism of differential graded algebras between it and $\Hoch_K(D, \End_K(E))$. 

\begin{remark}
	It is easy to use general abstract nonsense to produce an identification \[ \dR(\O, \End(E)) = \Hoch_K(D, \End_K(E)) \]
	in the derived category of vector spaces, as we will work out momentarily, but it is not apparent that the resulting identification preserves the differential graded Lie algbera structures on both complexes. Thus, we will instead construct an explicit quasi-isomorphism, so that we can directly verify that it preserves the differential graded Lie algebra structures. 
	
	To see how to produce this identification using abstract nonsense, note that the adjunction isomorphism 
	$\Hom(\O, \Hom(E, -)) = \Hom(E,-)$ is an isomorphism of functors $\Mod_D \to \Mod_D$, so, by taking horizontal sections on both sides, we get an isomorphism \[ \Hom_D(\O, \Hom(E, -)) = \Hom_D(E, -) \] of functors $\Mod_D \to \Mod_K$. Since $\Hom(E, -)$ is exact, we right derive both sides and get
	\[ \rightderived\!\Hom_D(\O, \Hom(E,-)) = \rightderived\!\Hom_D(E,-). \]
	From \cref{result:de rham complex is rhom}, we know the left hand side is $\dR(\O, \Hom(E, -))$. Putting these two identifications together and evaluating at $E$ shows that
	\[ \dR(\O, \End(E)) = \rightderived\!\Hom_D(E, E), \]
	and the identification of the right-hand side with the Hochschild complex is a consequence of the proof of \cite[lemma 9.1.9]{weibel}. 
\end{remark}

\begin{lemma} \label{derivations on D}
	The map $s \mapsto (s|_\O, s(\partial))$ defines an injective map of vector spaces
	\[ \begin{tikzcd} \Der_K(D, \End_K(E)) \ar[hookrightarrow]{r} & \Der_K(\O, \End_K(E)) \times \End_K(E) \end{tikzcd} \]
	whose image is the subspace of $(r, v)$ such that 
	\begin{equation} \label{image}
	[r(f), \partial] + [f, v] + r(\partial(f)) = 0
	\end{equation}
	for all $f \in \O$. 
\end{lemma}

\begin{proof}
	This proof is fairly excruciating, but there are no surprises. Observe that $D$ is a free left $\O$-module on the basis $1, \partial, \partial^2, \dotsc$. If we have $s \in \Der_K(D, \End_K(E))$, then an easy inductive argument shows that
	\begin{equation} \label{extended leibniz}
	s(f\partial^k) = s(f)\partial^k + \sum_{i = 0}^{k-1} f \partial^i s(\partial) \partial^{k-1-i}
	\end{equation}
	for any $f \in \O$ and $k \in \N$. Since the right-hand side depends only on the pair $(s|_\O, s(\partial))$, this proves injectivity. Moreover, we have
	\begin{align*}
	0 &= s(\partial f) - s(\partial)f - \partial s(f) \\
	&= s(f\partial) + s(\partial(f)) - s(\partial)f - \partial s(f) \\
	&= s(f)\partial + fs(\partial) + s(\partial(f)) - s(\partial)f - \partial s(f) \\
	&= [s(f), \partial] + [f, s(\partial)] + s(\partial(f))
	\end{align*}
	which shows that the pair $(s|_\O, s(\partial))$ satisfies \cref{image} for all $f \in \O$. 
	
	Conversely, suppose that we have $(r, v) \in \Der_K(\O, \End_K(E)) \times \End_K(E)$ satisfying \cref{image}. Motivated by \cref{extended leibniz}, we define a function $s : D \to \End_K(E)$ by declaring
	\[ s(f\partial^k) := r(f)\partial^k + \sum_{i = 0}^{k-1} f \partial^i v \partial^{k-1-i}   \]
	for all $f \in \O$ and $k \in \N$ and then extending additively. Then we certainly have $(s|_\O, s(\partial)) = (r, v)$, so we only need to check that $s$ is actually a derivation. In other words, we need to check that 
	\begin{equation} \label{leibniz} 
	s(PQ) = s(P)Q + P s(Q)
	\end{equation}
	for all pairs $(P, Q)$ of elements of $D$. 
	
	We will do this by ``inducting'' on the complexity of an element of $D$. The key to this is the following trivial observation: if $P = P' + P''$ for some $P', P'' \in D$ and \cref{leibniz} holds for $(P', Q)$ and $(P'', Q)$, then \cref{leibniz} also holds for $(P, Q)$. There is an analogous statement when $Q$ decomposes as a sum of two elements. We will use this observation tacitly throughout. 
	
	An element $P \in D$ is a \emph{monomial} if it is of the form $f\partial^k$ for some $f \in \O$ and $k \in \N$. We say that $k$ is the \emph{degree} of the monomial $P$, and that $P$ is a \emph{monic} monomial if $f = 1$. We say that $P$ is \emph{left Leibniz} if $(P, Q)$ satisfies \cref{leibniz} for all $Q \in D$ (equivalently, all monomials $Q$). Dually, we say that $Q$ is \emph{right Leibniz} if $(P, Q)$ satisfies \cref{leibniz} for all $P \in D$ (equivalently, all monomials $P$). To complete the proof, it suffices to show that every monomial is left Leibniz. Let us proceed incrementally towards this assertion, in 6 steps.
	
	\bigskip
	
	\noindent \textit{Step 1}. First, let us show that any monic monomial $Q = \partial^k$ is right Leibniz. Let $P = f\partial^j$. 
	\[ \begin{aligned}
	s(f\partial^{j+k}) &= r(f)\partial^{j+k} + \sum_{i = 0}^{j+k-1} f\partial^i v\partial^{j+k-1-i} \\
	s(f\partial^j)\partial^k &= r(f)\partial^{j+k} + \sum_{i = 0}^{j-1} f\partial^i v \partial^{j+k-1-i} \\
	f\partial^j s(\partial^k) &= \sum_{i = j}^{j+k-1} f\partial^{i} v \partial^{j+k-1-i}
	\end{aligned} \]
	Thus \cref{leibniz} holds for $(P, \partial^k)$.
	
	\bigskip
	
	\noindent \textit{Step 2}. Let us next show that every degree 0 monomial $f \in \O$ is left Leibniz. Suppose $Q = g\partial^k$. 
	\[ \begin{aligned}
	s(fg\partial^k) &= r(fg) \partial^k + \sum_{i = 0}^{k-1} fg\partial^i v \partial^{k-1-i} \\
	&= r(f)g\partial^k + fr(g)\partial^k + \sum_{i = 0}^{k-1} fg\partial^i v \partial^{k-1-i} \\
	s(f)g\partial^k &= r(f)g\partial^k \\
	fs(g\partial^k) &= fr(g)\partial^k + \sum_{i = 0}^{k-1} fg\partial^i v \partial^{k-1-i}
	\end{aligned} \]
	Thus \cref{leibniz} holds for $(f, Q)$.
	
	\bigskip
	
	\noindent \textit{Step 3}. A calculation identical to one we did earlier shows that
	\[ s(\partial f) - s(\partial)f - \partial s(f) = [r, \partial] + [f, v] + r(\partial(f))  \]
	which means that $(r, v)$ satisfying \cref{image} is equivalent to $(\partial, f)$ satisfying \cref{leibniz}.
	
	\bigskip
	
	\noindent \textit{Step 4}. Let us now show that $\partial$ is left Leibniz. Suppose $Q = f\partial^k$. Then 
	\[ \begin{aligned}
	s(\partial f \partial^k) &= s(\partial f)\partial^k + \partial f s(\partial^k) \\
	&= s(\partial)f\partial^k + \partial s(f)\partial^k + \partial f s(\partial^k) \\
	&= s(\partial)f\partial^k + \partial s(f\partial^k)
	\end{aligned} \]
	using the fact that $\partial^k$ is right Leibniz for the first equality (step 1), the fact that $(\partial, f)$ satisfies \cref{leibniz} for the second (step 3), and then the fact that $\partial^k$ is right Leibniz again for the third equality (step 1).
	
	\bigskip
	
	\noindent \textit{Step 5}. Let us now show by induction on degree that any monic monomial is left Leibniz. Suppose that $\partial^j$ is left Leibniz for some $j \in \N$. Let $Q = f\partial^k$. 
	Note that we have
	\[ \partial^{j+1}f\partial^k = \partial^j f \partial^{k+1} + \partial^j \partial(f) \partial^k \]
	in $D$. Using this, we have
	\[  \begin{aligned}
	s(\partial^{j+1}f\partial^k) &= s(\partial^j f \partial^{k+1} + \partial^j \partial(f) \partial^k) \\
	&= s(\partial^j)f\partial^{k+1} + \partial^j s(f\partial^{k+1}) + s(\partial^j)\partial(f)\partial^k + \partial^j s(\partial(f)\partial^k) \\
	&= s(\partial^j)(f\partial^{k+1} + \partial(f)\partial^{k}) + \partial^j s(f\partial^{k+1} + \partial(f)\partial^k) \\
	&= s(\partial^j) \partial f \partial^k + \partial^j s(\partial f \partial^k) \\
	&= s(\partial^j) \partial f \partial^k + \partial^j (s(\partial) f\partial^k + \partial s(f\partial^k)) \\
	&= (s(\partial^j)\partial + \partial^j s(\partial)) f\partial^k + \partial^{j+1} s(f\partial^k) \\
	&= s(\partial^{j+1}) f\partial^k + \partial^{j+1} s(f\partial^k)
	\end{aligned} \]
	using the inductive hypothesis that $\partial^j$ is left Leibniz twice for the second equality, the fact that $\partial$ is left Leibniz for the fifth (step 4), and the fact that $\partial$ is right Leibniz for the final (step 1).
	
	\bigskip
	
	\noindent \textit{Step 6}. Finally, we show that an arbitrary monomial $f\partial^k$ is left Leibniz. For any $Q \in D$, observe that
	\[ \begin{aligned} s(f\partial^k Q) &= s(f)\partial^k Q + fs(\partial^k Q) \\
	&= s(f)\partial^k Q + fs(\partial^k)Q + f\partial^k s(Q) \\
	&= s(f\partial^k)Q + f\partial^k s(Q)
	\end{aligned} \]
	where the first and third equalities are because $f$ is left Leibniz (step 2) and the second because $\partial^k$ is left Leibniz (step 5).
\end{proof}

\begin{corollary} \label{zeta}
	There is a unique injective $K$-linear map $\zeta : \End(E) \to \Der_K(D, \End_K(E))$ such that $\zeta(\rho)(f) = 0$ for all $f \in \O$ and $\zeta(\rho)(\partial) = \rho$. Moreover, $\im(\zeta)$ is precisely the set of derivations $D \to \End_K(E)$ which annihilate $\O$. 
\end{corollary}

\begin{proof}
	If $\rho \in \End_K(E)$, note that $(0, \rho)$ satisfies \cref{image} if and only if $\rho \in \End(E)$. So, for $\rho \in \End(E)$, let $\zeta(\rho)$ be the unique $s \in \Der_K(D, \End_K(E))$ such that $(s|_\O, s(\partial)) = (0, \rho)$.
\end{proof}

\begin{theorem} \label{result:de rham and hochschild}
	The map $\zeta$ of \cref{zeta} defines a quasi-isomorphism of differential graded $K$-algebras \[ \begin{tikzcd} \dR(\O, \End(E)) \ar{r} & \Hoch_K(D, \End_K(E)). \end{tikzcd} \]
\end{theorem}

\begin{proof}
	Since $E$ has projective dimension at most 1 over $D$, we have $\Ext^i_D(E,E) = 0$ for all $i \geq 2$. Thus, as in \cref{remarks on dimension 1}, the natural inclusion $\tau_{\leq 1}\Hoch_K(D, \End_K(E)) \to \Hoch_K(D, \End_K(E))$ is a quasi-isomorphism of differential graded $K$-algebras. The map $\zeta$ of \cref{zeta} defines a map of complexes $\dR(\O, \End(E)) \to \tau_{\leq 1}\Hoch_K(D, \End_K(E))$ as follows (where we are using the description of the truncated Hochschild complex from \cref{remarks on dimension 1}). 
	\[ \begin{tikzcd} \End(E) \ar{r}{\partial} \ar[hookrightarrow]{d} & \End(E) \ar{r} \ar{d}{\zeta} & 0 \ar{r} \ar{d} & \dotsb \\
	\End_K(E) \ar{r}[swap]{d}  & \Der_K(D, \End_K(E)) \ar{r} & 0 \ar{r} & \dotsb \end{tikzcd} \]
	To see that this diagram commutes, note that if $\sigma \in \End(E)$, note that 
	\[ \zeta(\partial \sigma)(f) = 0 = [f, \sigma] = d\sigma(f) \]
	for any $f \in \O$ and that
	\[ \zeta(\partial \sigma)(\partial) = \zeta([\partial, \sigma])(\partial) = [\partial, \sigma] = d\sigma(\partial). \]
	In other words, $d\sigma$ and $\zeta(\partial\sigma)$ agree on $\O$ and on $\partial$, so the injectivity assertion of \cref{derivations on D} proves commutativity of the diagram. 
	
	A similar argument using the same injectivity assertion shows that $\zeta$ is a homomorphism of differential graded $K$-algebras. Suppose $\sigma, \rho \in \End(E)$, where $\sigma$ is regarded as living in degree 0 and $\rho$ in degree 1. We want to show that
	\[ \zeta(\sigma \rho) = \sigma \zeta(\rho) = \zeta(\sigma)\rho, \]
	but clearly all three annihilate $\O$ and take the value $\sigma \rho$ on $\partial$. 
	
	We now want to show that $\zeta$ induces isomorphisms on cohomology. This is clear in degree 0, so we focus on degree 1. Note that we have a diagram as follows. 
	\[ \begin{tikzcd} H^1_{\dR}(\End(E)) \ar{r} \ar[bend right]{dr}[swap]{\zeta} & \Ext^1_D(E, E) \ar{d} \\ & H^1(\Hoch_K(D, \End_K(E))) \end{tikzcd}  \]
	Here, the vertical map is the isomorphism detailed in \cref{derivations and extensions}, and the horizontal map is the isomorphism of \cite[lemma 5.3]{kedlaya_pde}. To show that $\zeta$ is an isomorphism, it suffices to prove that this diagram is commutative. 
	
	Let us begin by recalling the explicit construction of the horizontal isomorphism displayed above as it is described in \cite[proof of lemma 5.3]{kedlaya_pde}. Suppose we have $\rho \in \End(E)$ representing a cohomology class in $H^1_{\dR}(\End(E))$. Its image in $\Ext^1_D(E, E)$ is denoted $E \oplus_\rho E$. As an $\O$-module, it is just $E \oplus E$, but $\partial$ acts by
	\[ \partial (e, e') := (\partial e + \rho(e'), \partial e'). \]
	It is straightforward to verify that the Leibniz rule 
	\[ \partial (a \cdot (e, e')) = \partial(a) \cdot (e,e') + a \cdot \partial (e, e') \]
	is satisfied, and that the inclusion $E \to E \oplus_\rho E$ into the first coordinate and the projection $E \oplus_\rho E \to E$ onto the second coordinate are both homomorphisms of differential $\O$-modules. In other words, $E \oplus_\rho E$ is in fact an extension of $E$ by itself. 
	
	Now let us find the image of this extension in 
	\[ H^1(\Hoch_K(D, \End_K(E))) = \Der_K(D, \End_K(E))/\PDer_K(D, \End_K(E)). \]
	To do this, we use the description in \cref{derivations and extensions}. Note that the underlying $\O$-module of $E \oplus_\rho E$ is $E \oplus E$. In other words, there is a natural pair of $K$-linear (even $\O$-linear) splittings for this extension. Thus the image of $E \oplus_\rho E$ in $H^1(\Hoch_K((D, \End_K(E))))$ is the class of the derivation $s \in \Der_K(D, \End_K(E))$ where, for any $P \in D$ and $e \in E$, $s(P)(e)$ is the projection of $P \cdot (0, e)$ onto the first coordinate. 
	
	Note that taking $P = f$ for some $f \in \O$, then $f \cdot (0, e) = (0, fe)$. Thus $s|_\O = 0$. Moreover, taking $P = \partial$ shows that $s(\partial) = \rho$. It follows from the injectivity assertion of \cref{derivations on D} that $s = \zeta(\rho)$. This proves that the diagram is commutative. \qedhere 
\end{proof}

\begin{corollary} \label{result:governing complex connection}
	$\dR(\O, \End(E))$ governs $\Def_{E}$. \qed
\end{corollary}

\begin{remark} \label{explicit governance}
	It is worth describing the equivalence $\Theta : \Def_{\dR(\O, \End(E))} \to \Def_E$ explicitly. For $R \in \Art_K$, the objects of $\Def_{\dR(\O, \End(E))}(R)$ are elements of
	\[ \m_R \otimes \End(E) = \ker(\End(E_R) \to \End(E)). \]
	Given $\mu \in \ker(\End(E_R) \to \End(E))$, let $E_{R,\mu}$ denote the differential $\O_R$-module whose underlying $\O_R$-module is $E_R$, and where $\partial$ acts by $1 \otimes \partial + \mu$. Since $\mu$ reduces to 0 modulo $\m_R$, the natural $\O_R$-module homomorphism $\theta : E_{R,\mu} \to E$ is actually a $D_R$-module homomorphism. 

	The equivalence $\Theta : \Def_{\dR(\O, \End(E))} \to \Def_E$ is given by $\mu \mapsto (R, E_{R,\mu}, \theta)$ on the level of objects. On morphisms, $\Theta$ acts ``by exponentiation.'' See the proof of \cref{result:hochschild governs module deformations} for details about this. 
\end{remark}

\subsection{Trace and determinant} \label{trace and determinant}

\begin{lemma} \label{trace preserves algebra structure}
	The trace map $\tr :  \End(E) \to \O$ is a split surjective homomorphism of differential $\O$-modules. Moreover, the induced map $\tr : \dR(\O, \End(E)) \to \dR(\O,\O)$ is a homomorphism of differential graded $K$-Lie algebras.\footnote{Note that it is \emph{not} a homomorphism of differential graded $K$-algebras: it preserves the commutator bracket, but not multiplication itself.}
\end{lemma}

\begin{proof}
	Observe that the natural embedding $\O \to \End(E)$, carrying an element $f \in \O$ to the multiplication by $f$ map, is a homomorphism of differential $\O$-modules. Indeed, if we let $\mu_f$ denote the multiplication by $f$ map, then
	\[ (\partial \cdot \mu_f)(e) = \partial \mu_f (e) - \mu_f \partial(e) = \partial(fe) - f\partial(e) = \partial(f)e = \mu_{\partial(f)}(e). \]
	Now let us show that the trace map also preserves the differential module structure. We choose a basis $(e_1, \dotsc, e_n)$ for $E$, and then observe that if $\phi \in \End(E)$ has matrix $M$ with respect to this basis, then $\partial\phi$ has matrix $\partial(M) + [N, M]$ where $N$ is the matrix of action of $\partial$ and $\partial(M)$ denotes entry-wise application of $\partial$ to $M$. Then
	\[ \tr(\partial \phi) = \tr(\partial(M) + [N, M]) = \tr(\partial(M)) = \partial(\tr(M)) = \partial(\tr(\phi)). \]
	Clearly the map $\O \to \End(E)$ splits the trace map, so this completes the proof of the first assertion. The second assertion follows from the observation that 
	\[ \tr([\alpha, \beta]) = 0 = [\tr(\alpha), \tr(\beta)] \]
	for any $\alpha, \beta \in \End(E)$. 
\end{proof}

\begin{lemma} \label{matrix of determinant}
	If $N$ is the matrix of $\partial$ on $E$ with respect to an $\O$-basis $(e_1, \dotsc, e_n)$, then $\tr(N)$ is the matrix of $\partial$ on $\det(E)$ with respect to $e_1 \wedge \dotsb \wedge e_n$.
\end{lemma}

\begin{proof}
	Recall from \cite[definition 5.3.2]{kedlaya_pde} that
	\[ \partial (e_1 \wedge \dotsb \wedge e_n) = \sum_{i = 1}^n e_1 \wedge \dotsb \wedge e_{i-1} \wedge \partial e_i \wedge e_{i+1} \wedge \dotsb \wedge e_n. \]
	By definition of $N$, we have 
	\[ \partial e_i = \sum_{j = 1}^n N_{j,i} e_j, \]
	so
	\[ \partial(e_1 \wedge \dotsb \wedge e_n) = \sum_{i = 1}^n N_{i,i} (e_1 \wedge \dotsb \wedge e_n) = \tr(N) (e_1 \wedge \dotsb \wedge e_n). \qedhere \]
\end{proof}

\begin{lemma} \label{result:trace and determinant}
	The following diagram 2-commutes. 
	\[ \begin{tikzcd}
	\Def_{\dR(\O, \End(E))} \ar{r}{\tr} \ar{d} & \Def_{\dR(\O, \O)} \ar{d} \\ \Def_{E} \ar{r}[swap]{\det} & \Def_{\det(E)}
	\end{tikzcd} \]
\end{lemma}

\begin{proof}
	Since $\det(E)$ is of rank 1, we have $\End(\det(E)) = \O$. The vertical arrows are the equivalences of \cref{result:governing complex connection}, which are described in more detail in \cref{explicit governance}. The horizontal arrow on top is induced by the homomorphism $\tr : \dR(\O, \End(E)) \to \dR(\O, \O)$ of \cref{trace preserves algebra structure}, and the horizontal arrow on the bottom is given by 
	\[ (R, F, \theta) \mapsto (R, \det(F), \det(\theta)). \]
	The 2-commutativity of the square is straightforward to verify at this point; once we choose bases, the key observation is \cref{matrix of determinant} above. The details follow.  
	
	For any $R \in \Art_K$, observe that we have a canonical isomorphism $\eta_R : \det(E)_R \to \det(E_R)$ of $\O_R$-modules. Explicitly, if we choose a basis $(e_1, \dotsc, e_n)$ of $E$, then 
	\[ \eta_R(1 \otimes (e_1 \wedge \dotsb \wedge e_n)) = (1 \otimes e_1) \wedge \dotsb \wedge (1 \otimes e_n), \]
	but $\eta_R$ does not depend on the choice of basis. Chasing the basis $1 \otimes (e_1 \wedge \dotsb \wedge e_n)$ around shows that $\eta_R$ makes the following diagram commute.  
	\begin{equation} \label{eqn:compatible with reduction}
	\begin{tikzcd} \det(E)_{R} \ar{r}{\eta_{R}} \ar{d}[swap]{\theta} & \det(E_{R}) \ar{d}{\det(\theta)} \\ \det(E) \ar{r}[swap]{1} & \det(E) \end{tikzcd}
	\end{equation}
	Moreover, the isomorphism $\eta_R$ is evidently natural in $R$ in the sense that, if $\pi : R' \to R$ is a homomorphism in $\Art_K$, we have a commutative diagram as follows.
	\begin{equation} \label{eqn:naturality of isomorphism}
	\begin{tikzcd} \det(E)_{R'} \ar{r}{\eta_{R'}} \ar[swap]{d}{\pi \otimes 1} & \det(E_{R'}) \ar{d}{\det(\pi \otimes 1)} \\ \det(E)_R \ar{r}[swap]{\eta_R} & \det(E_R) \end{tikzcd}
	\end{equation}
	Now suppose $\mu \in \ker(\End(E_R) \to \End(E))$. We claim that $(1, \eta_R)$ is an isomorphism 
	\[ \begin{tikzcd} (R, \det(E)_{R,\tr(\mu)}, \theta) \ar{r} & (R, \det(E_{R,\mu}), \det(\theta)) \end{tikzcd} \]
	in $\Def_{\det(R)}$. In fact, in light of the commutative diagram \eqref{eqn:compatible with reduction}, it is sufficient to show that $\eta_R$ is an isomorphism of differential $\O_R$-modules $\det(E)_{R,\tr(\mu)} \to \det(E_{R,\mu})$. Once we show this, it follows immediately from the diagram \eqref{eqn:naturality of isomorphism} that the collection $(1, \eta_R)$ defines a 2-morphism that makes the square in the statement of the lemma 2-commute.
	
	To prove that $\eta_R : \det(E)_{R,\tr(\mu)} \to \det(E_{R,\mu})$ is an isomorphism of differential modules, we choose a basis $(e_1, \dotsc, e_n)$ for $E$. Let $N$ be the matrix of $\partial$ acting on $E$ and let $M$ be the $n \times n$ matrix with coefficients in $\m_R \otimes \O$ representing $\mu$. Observe that $(1 \otimes e_1, \dotsc, 1 \otimes e_n)$ is a basis for $E_{R,M}$ and $1 \otimes N + M$ is the matrix of action of $\partial$ on this basis. Applying \cref{matrix of determinant}, we see that $\partial$ acts on $(1 \otimes e_1) \wedge \dotsb \wedge (1 \otimes e_n)$ by 
	\[ \tr(1 \otimes N + M) = 1 \otimes \tr(N) + \tr(M). \]
	This is precisely the matrix with which $\partial$ acts on the basis $1 \otimes (e_1 \wedge \dotsb \wedge e_n)$ of $\det(E)_{R,\tr(M)}$, proving the claim.
\end{proof}

\subsection{Trivialized and trivializable deformations} \label{section:trivialized and trivializable}

We continue to fix $\O$ and $E$ as above. Moreover, we assume in addition that $\O^\sharp$ is an another commutative $K$-algebra, $\O \to \O^\sharp$ is a homomorphism of $K$-algebras, and that there is a $K$-linear derivation on $\O^\sharp$, again denoted $\partial$, which extends the action of $\partial$ on $\O$.

We let $D^\sharp := \O^\sharp[\partial]$ be the corresponding ring of differential operators, and we regard $E^\sharp := \O^\sharp \otimes_\O E$ as a finite free differential $\O^\sharp$-module. In other words, $\Def_{E^\sharp}$ means $\Def_{D^\sharp, E^\sharp}$. For any $(R, F, \theta) \in \Def_{E}$, we let \[ F^\sharp := \O^\sharp \otimes_\O F \]
regarded as a differential $\O_R^\sharp$-module. Since $F$ is $R$-flat, so is $F^\sharp$. Letting $\theta^\sharp : F^\sharp \to E^\sharp$ be the natural map induced by $\theta$, we observe that $(R, F, \theta) \mapsto (R, F^\sharp, \theta^\sharp)$ defines a functor $\Def_{E} \to \Def_{E^\sharp}$.

We can give a description of this functor in terms of the differential graded Lie algebras which govern the deformation categories $\Def_E$ and $\Def_{E^\sharp}$. Observe that the map $E \to E^\sharp$ induces a homomorphism $\dR(\O, E) \to \dR(\O^\sharp, E^\sharp)$ of complexes.
\[ \begin{tikzcd} E \ar{d} \ar{r}{\partial} & E \ar{d} \\
E^\sharp \ar{r}[swap]{\partial} &  E^\sharp \end{tikzcd} \]
In particular, since 
\[ \End(E^\sharp) = \O^\sharp \otimes_\O \End(E), \]
there is a natural map $\dR(\O, \End(E)) \to \dR(\O^\sharp, \End(E^\sharp))$, which is a homomorphism of differential graded $K$-algebras. It is this map that induces the functor $\Def_E \to \Def_{E^\sharp}$, in the following sense. 

\begin{lemma} \label{result:equivalences and base change}
	The following diagram 2-commutes.
	\[ \begin{tikzcd} \Def_{\dR(\O, \End(E))} \ar{r} \ar{d} & \Def_{\dR(\O^\sharp, \End(E^\sharp))} \ar{d} \\ \Def_E \ar{r} & \Def_{E^\sharp} \end{tikzcd} \]
\end{lemma}

\begin{proof}
	This is very pedantic and the proof is very similar in structure to that of \cref{result:trace and determinant}, so we write down fewer details. Observe that there is a canonical isomorphism $\eta_R : (E^\sharp)_R \to (E_R)^\sharp$ of $\O^\sharp_R$-modules coming from the symmetry of the tensor product:
	\[ \begin{tikzcd} (E^\sharp)_R = R \otimes_K (\O^\sharp \otimes_\O E) \ar{r}{\eta_R} & \O^\sharp \otimes_\O (R \otimes_K E) = (E_R)^\sharp. \end{tikzcd} \]
	The content is to show that $\eta_R$ is actually an isomorphism of differential $\O_R^\sharp$-modules $(E^{\sharp})_{R,\mu^\sharp} \to (E_{R,\mu})^\sharp$ for any $\mu \in \ker(\End(E_R) \to \End(E))$. Note that $\mu^\sharp = 1_{\O^\sharp} \otimes \mu$ and $\partial$ acts on the domain by
	\[ 1_R \otimes 1_{\O^\sharp} \otimes \partial + \mu^\sharp. \]
	Under the symmetry isomorphism $\eta_R$, this corresponds precisely to
	\[ 1_{\O^\sharp} \otimes (1_R \otimes \partial + \mu) \]
	which is precisely how $\partial$ acts on $(E_{R,\mu})^\sharp$. 
\end{proof}

In other words, the functor $\Def_E \to \Def_{E^\sharp}$ is the one induced by the homomorphism of differential graded algebras $\dR(\O, \End(E)) \to \dR(\O^\sharp, \End(E^\sharp))$. Thus, we can set up a diagram as in \cref{backbone example}, where $\Gamma$ denotes the residual gerbe of $\Def_{E^\sharp}$. 
\[ \begin{tikzcd} \Def_E^{\sharp, +} \ar{r} \ar{d} & \Def_E^{\sharp} \ar{r} \ar{d} & \Def_E \ar{d} \\ h_K \ar{r} & \Gamma \ar{r} & \Def_{E^\sharp}  \end{tikzcd} \]
We can describe the deformation categories $\Def_E^{\sharp,+}$ and $\Def_E^\sharp$ explicitly. First, let's look at $\Def_E^{\sharp,+}$. Its objects are tuples $(R, F, \theta, \tau)$ where $(R, F, \theta) \in \Def_{E}$ and $\tau : (E^\sharp_R, 1) \to (F^\sharp, \theta^\sharp)$ is in $\Def_{E^\sharp}(R)$.
Morphisms $(R', F', \theta', \tau') \to (R, F, \theta, \tau)$ in $\Def_{E}^{\sharp,+}$ are morphisms $(\pi, u) : (R', F', \theta') \to (R, F, \theta)$ in $\Def_E$ such that the following diagram commutes.
\[ \begin{tikzcd} E_{R'}^\sharp \ar{r} \ar{d}[swap]{\tau'} & E_R^\sharp \ar{d}{\tau} \\ F'^\sharp \ar{r}[swap]{u^\sharp} & F^\sharp \end{tikzcd} \]

\begin{definition}[Trivialized deformations]
	Objects of $\Def_E^{\sharp,+}(R)$ will be called \emph{$\O^\sharp$-trivialized deformations} of $E$ over $R$.
\end{definition}

Now $\Def_E^\sharp$ is the full subcategory of $\Def_{E}$ consisting of objects $(R, F, \theta)$ such that there exists a morphism $\tau : (E^\sharp_R, 1) \to (F^\sharp, \theta^\sharp)$ in $\Def_{E^\sharp}(R)$, but we do \emph{not} fix $\tau$ as a part of the data.

\begin{definition}[Trivializable deformations]
	The objects of $\Def_{E}^{\sharp}(R)$ will be called \emph{$\O^\sharp$-trivializable deformations} of $E$ over $R$.
\end{definition}

\begin{lemma} \label{smoothness of trivialized and trivializable}
	$\Def_E^{\sharp,+}$ is smooth if and only if $\Def_E^{\sharp}$ is smooth.
\end{lemma}

\begin{proof}
	As we noted in \cref{backbone example}, the map $\Def_E^{\sharp,+} \to \Def_E^{\sharp}$ is smooth and essentially surjective. Apply \tag{06HM}.
\end{proof}

\begin{proposition} \label{rank 1 smoothness}
	Suppose $H^0_{\dR}(\O) = K$ and $H^1_{\dR}(\O)$ is finite dimensional. If $E$ is of rank 1, then $\Def^\sharp_E$ and $\Def^{\sharp,+}_E$ are both smooth. 
\end{proposition}

\begin{proof}
	Note that $\End(E) = \O$ since $E$ is of rank 1, so $\dR(\O, \O)$ governs $\Def_E$ by \cref{result:governing complex connection}. We have assumed that 
	\[ \End_D(E) = H^0_{\dR}(\End(E)) = H^0_{\dR}(\O) = K, \]
	so automorphisms lift over small extensions by \cref{automorphisms lift for endomorphically simples}. The tangent space $T(\Def_E) = H^1_{\dR}(\O)$ is also finite dimensional by assumption, so in fact $\Def_E$ is prorepresentable by the fundamental theorem of deformation theory \ref{deformation theory}. It is smooth, since the governing complex $\dR(\O, \O)$ vanishes outside degrees 0 and 1. 
	
	Choose a list $(h_1, \dotsc, h_s)$ in $\O$ whose image in $H^1_{\dR}(\O)$ is a basis for $\ker(H^1_{\dR}(\O) \to H^1_{\dR}(\O^\sharp))$. Then choose $(h_{s+1}, \dotsc, h_r)$ in $\O$ whose images in $H^1_{\dR}(\O^\sharp)$ form a basis for $\im(H^1_{\dR}(\O) \to H^1_{\dR}(\O^\sharp))$. Then the image of $(h_1, \dotsc, h_s, h_{s+1}, \dotsc, h_r)$ in $H^1_{\dR}(\O)$ is a basis. Let $e \in E$ be an $\O$-basis and suppose $N \in \O$ is such that $\partial e = Ne$. 
	
	Since $\Def_{E}$ is smooth and prorepresentable and $(h_1, \dotsc, h_r)$ gives a basis for the tangent space $H^1_{\dR}(\O)$, for any $(R, F, \theta) \in \Def_E$ there is a lift $f \in F$ of $e$ such that
	\[ \partial(f) = \left( 1 \otimes N +  \sum_i \alpha_i \otimes h_i \right) f \]
	for some $\alpha_1, \dotsc, \alpha_r \in \m_R$. We claim that $(R, F, \theta)$ is trivializable if and only if 
	\[ \alpha_{s+1} = \dotsb = \alpha_r = 0. \]
	To see this, observe that an $\O_R^\sharp$-module isomorphism $\tau : E_R^\sharp \to F^\sharp$ compatible with 1 and $\theta^\sharp$ must be given by $1 \otimes e \mapsto (1+s)f$ for some $s \in \m_R \otimes_K \O^\sharp$, and it is easy to check  that $\tau$ is $D_R^\sharp$-linear if and only if
	\[ \partial(s) + (1+s)\sum_i \alpha_i \otimes h_i = 0. \]
	Note that $\alpha_i \equiv 0 \bmod{\m_R}$ for all $i$. Since we can factor the surjection $R \to K$ into a composite of small extensions, let us assume inductively that there exists some $\epsilon \in \m_R$ such that $\epsilon \m_R = 0$ and $\alpha_i \equiv 0 \bmod{\epsilon R}$ for all $i > s$. Since $s \in \m_R \otimes_K \O^\sharp$, we have $\alpha_i s = 0$ for all $i > s$, so the above equation reads
	\[ \partial(s) + (1+s)\left( \sum_{i \leq s} \alpha_i \otimes h_i \right) + \sum_{i > s} \alpha_i \otimes h_i = 0. \]
	Choose a direct sum complement $Q$ for $\epsilon R$ inside $\m_R$, so that any element $\eta \in \m_R$ can be written uniquely as $\eta_0 \epsilon + \tilde{\eta}$ where $\eta_0 \in K$ and $\tilde{\eta} \in Q$. Then we have $\alpha_i = \alpha_{i,0}\epsilon + \tilde{\alpha}_i$ for all $i$, and similarly we can write $s = \epsilon \otimes s_0 + \tilde{\eta} \otimes \tilde{s}$ for some $\tilde{\eta} \in Q$ and $s_0, \tilde{s} \in \O^\sharp$. Again recalling that $\epsilon \m_R = 0$, the ``epsilon part'' of the above equation is
	\[ \partial(s_0) + \sum_i \alpha_{i,0} h_i = 0, \]
	which is an equation in $\O^\sharp$. Passing to $H^1_{\dR}(\O^\sharp)$, clearly $\partial(s_0)$ disappears, as do the terms corresponding to $h_1, \dotsc, h_s$ since these terms were chosen to be in $\ker(H^1_{\dR}(\O) \to H^1_{\dR}(\O^\sharp))$. But the image of $(h_{s+1}, \dotsc, h_r)$ is linearly independent in $H^1_{\dR}(\O^\sharp)$ by construction, so we have $\alpha_{i,0} = 0$ for all $i > s$. But we also had assumed inductively that $\alpha_i \equiv 0 \bmod{\epsilon R}$ for all $i > s$, so in fact $\alpha_i = 0$ for all $i > s$.
	
	Now suppose $(R, F, \theta)$ is trivializable and choose a basis $f \in F$ as above. Let $F'$ be a free $\O_{R'}$-module of rank 1 with basis $f'$. Choose lifts $\alpha'_i \in \m_{R'}$ of $\alpha_i$ and define an action of $\partial$ on $F'$ by
	\[ \partial f' = \left( 1 \otimes N + \sum_{i \leq s} \alpha_i' \otimes h_i \right) f'. \]
	Let $u : F' \to F$ be given by $f' \mapsto f$ and $\theta' = \theta \circ u$. Then $(R', F', \theta')$ is trivializable by our observations above and defines a lift of $(R, F, \theta)$. Thus $\Def_E^{\sharp}$ is smooth. Smoothness of $\Def_E^{\sharp,+}$ follows from \cref{smoothness of trivialized and trivializable}.  
\end{proof}


\subsection{Compactly supported and parabolic cohomology}

We continue to fix $\O^\sharp$ as above. 

\begin{definition} \label{definition:abstract compactly supported and parabolic cohomologies}
	We define
	\[ \begin{aligned} 
	C^+(E) &= \Cone(\dR(\O, E) \to \dR(\O^\sharp, E^\sharp))[-1] \\
	C(E) &= \Cone(\dR(\O, E) \oplus H^0_{\dR}(E^\sharp) \to \dR(\O^\sharp, E^\sharp))[-1]
	\end{aligned} \]
	and then set
	\[ H^i_c(E) = H^i(C^+(E))  \text{ and } H^i_p(E) = H^i(C(E)) \]
	for all $i$.
\end{definition}

It follows from the associated long exact sequences as in \cref{backbone example} that $H^i_c(E) = H^i_p(E)$ for all $i \geq 2$, that $H^0_p(E) = H^0_{\dR}(E)$, and that
\[ H^1_p(E) = \im(H^1_c(E) \to H^1_{\dR}(E)) = \ker(H^1_{\dR}(E) \to H^1_{\dR}(E^\sharp)). \]
Moreover, in the cases of interest to us, we will have $H^0_c(E) = 0$ for the following reason. 

\begin{lemma} \label{result:zeroth compactly supported cohomology vanishes}
	If $\O \to \O^\sharp$ is injective, then $H^0_c(E) = 0$. 
\end{lemma}

\begin{proof}
	Since $E$ is free over $\O$ and therefore flat, the map $E \to E^\sharp = \O^\sharp \otimes_\O E$ is also injective. The distinguished triangle
	\[ \begin{tikzcd} C^+ \ar{r} & \dR(\O, E) \ar{r} & \dR(\O^\sharp, E^\sharp) \ar{r}{+} & \mbox{} \end{tikzcd} \]
	then induces a long exact sequence as follows.
	\[ \begin{tikzcd} 0 \ar{r} & H^0_c(E) \ar{r} & H^0_{\dR}(E) \ar{r} \arrow[d, phantom, ""{coordinate, name=Z}] & H^0_{\dR}(E^\sharp) \arrow[dll,
	rounded corners,
	to path=
	{ -- ([xshift=2ex]\tikztostart.east)
		|- (Z) [near end]\tikztonodes
		-| ([xshift=-2ex]\tikztotarget.west)
		-- (\tikztotarget)}
	] \\
	& H^1_c(E) \ar{r} & H^1_{\dR}(E) \ar{r} \arrow[d, phantom, ""{coordinate, name=Y}] & H^1_{\dR}(E^\sharp) \arrow[dll,
	rounded corners,
	to path=
	{ -- ([xshift=2ex]\tikztostart.east)
		|- (Y) [near end]\tikztonodes
		-| ([xshift=-2ex]\tikztotarget.west)
		-- (\tikztotarget)}
	] \\
	& H^2_c(E) \ar{r} & 0 \end{tikzcd} \]
	It follows that $H^0_c(E) = 0$.
\end{proof}

The following is an immediate consequence of the definitions plus our observations in \cref{backbone example}. 

\begin{lemma} \label{result:infinitesimal deformation theory of sharp categories}
	We have the following.
	\[ \begin{aligned} 
	& \Inf(\Def_E^{\sharp,+}) = H^0_c(\End(E)) & & T(\Def_E^{\sharp,+}) = H^1_c(\End(E)) \\ 
	& \Inf(\Def_E^\sharp) = H^0_{\dR}(\End(E)) & & T(\Def_E^\sharp) = H^1_p(\End(E)) \end{aligned} \]
	Moreover, $H^2_c(\End(E))$ is compatibly an obstruction space for both $\Def_E^{\sharp,+}$ and $\Def_E^\sharp$. \qed
\end{lemma}

\subsection{Duality pairing}

\begin{definition}[Duality pairing] \label{definition:duality pairing}
	Observe that $H^2_c(E)$ is the cokernel of the sum of the differential $\partial : E^\sharp \to E^\sharp$ with the natural map $E \to E^\sharp$, which we will denote $e \mapsto e^\sharp$. There is a $K$-bilinear map
	\begin{equation} \label{cup product}
	\begin{tikzcd} H^0_\dR(E^\vee) \times H^2_c(E) \ar{r} & H^2_c(\O). \end{tikzcd}
	\end{equation}
	which we call the \emph{duality pairing}, given by 
	\[ (\phi, e) \mapsto \langle \phi, e\rangle := \phi^\sharp(e), \]
	where $e \in E^\sharp$ represents an element in $H^2_c(E)$. 
\end{definition}

\begin{proof}[Proof that the above pairing is well-defined]
	If $e = f^\sharp$ for some $f \in E$, then 
	\[ \phi^\sharp(e) = \phi^\sharp(f^\sharp) = \phi(f)^\sharp \]
	is in the image of $\O \to \O^\sharp$ and therefore vanishes in $H^2_c(\O)$. If $e = \partial f$ for some $f \in E^\sharp$, then 
	\[ \phi^\sharp(e) = \phi^\sharp(\partial f) = \partial \phi^\sharp(f), \] since $\phi \in H^0_{\dR}(E^\vee)$, so we see that $\phi^\sharp(e)$ is in the image of $\partial : \O^\sharp \to \O^\sharp$ and therefore vanishes in $H^2_c(\O)$. 
\end{proof}

\begin{example}[Trace pairing] \label{example:trace pairing}
	Note that have a perfect\footnotemark\ pairing
	\[ \begin{tikzcd}
	\End(E) \times \End(E) \ar{r} & \O
	\end{tikzcd} \]
	given by $(\alpha, \beta) \mapsto \tr(\alpha \circ \beta)$. This yields an identification $\End(E)^\vee = \End(E)$. 
	Combining this identification and the natural identification $\O^\vee = \O$, the dual of the trace map $\tr : \End(E) \to \O$ is exactly the inclusion $\O \to \End(E)$ carrying $f \in E$ to the multiplication by $f$ map. Using the identification $\End(E)^\vee = \End(E)$, the duality pairing \eqref{cup product} becomes the $K$-bilinear map
	\begin{equation} \label{cup product endomorphisms}
	\begin{tikzcd} H^0_\dR(\End(E)) \times H^2_c(\End(E)) \ar{r} & H^2_c(\O). \end{tikzcd}
	\end{equation}
	given by $\langle \alpha, \beta \rangle = \tr(\alpha^\sharp \circ \beta)$.
\end{example}

\begin{example}
	Applying \cref{example:trace pairing} with $E = \O$, we have a pairing $H^0_{\dR}(\O) \times H^2_c(\O) \to H^2_c(\O)$ given by 
	\[ \langle f, g \rangle = f^\sharp g. \]
	This pairing is always non-degenerate.\textsuperscript{\ref{footnote:nondeg perfect}} Indeed, suppose we have some $g \in H^2_c(\O)$ such that $\langle f, g \rangle = f^\sharp g = 0$ for all $f \in H^0_{\dR}(\O)$. Taking $f = 1$ shows that we must have $g = 0$. If $H^2_c(\O)$ is nonzero and the pairing is perfect, then we must have $H^0_{\dR}(\O) = K$. 
	
	\footnotetext{\label{footnote:nondeg perfect} Let $A$ be a commutative ring and $\mu : M \times N \to L$ an $A$-bilinear map. Then $\mu$ is \emph{non-degenerate} (resp. \emph{perfect}) if the induced map $N \to \Hom_A(M, L)$ is injective (resp. bijective).} 
\end{example}

The following tells us that the trace pairing annihilates all of the obstruction classes for $\Def^\sharp$ in $H^2_c(\End(E))$.

\begin{lemma} \label{obstructions annihilated}
	Suppose $H^0_{\dR}(\O) = K$ and $H^1_{\dR}(\O)$ is finite dimensional. If $\pi : R' \to R$ is a small extension in $\Art_K$ and $(F, \theta) \in \Def_E^\sharp(R)$, then
	\[ \langle 1, o(\pi, (F, \theta)) \rangle = 0. \]
\end{lemma}

\begin{proof}
	Observe that we have a commutative diagram of differential graded $K$-algebras as follows.
	\[ \begin{tikzcd} \dR(\O, \End(E)) \ar{r}{\tr} \ar{d} & \dR(\O, \O) \ar{d} \\ \dR(\O^\sharp, \End(E^\sharp)) \ar{r}[swap]{\tr} & \dR(\O, \O^\sharp) \end{tikzcd} \]
	Since $L \mapsto \Def_L$ is 2-functorial, this induces a 2-commutative square that becomes the back of the following 2-commuatative cube. 
	\setlength{\perspective}{10pt}
	\[ \begin{tikzcd}[row sep={40,between origins}, column sep={40,between origins}]
	&[-\perspective] &[-\perspective] \Def_{\dR(\O, \End(E))} \ar{rrr}{\tr} \ar{dd} \ar{dll} &[\perspective] &[-\perspective] &[-\perspective] \Def_{\dR(\O, \O)} \ar{dd} \ar{dll} \\[-\perspective]
	\Def_E \ar[crossing over]{rrr}{\det} & & & \Def_{\det(E)} \\[\perspective]
	& & \Def_{\dR(\O^\sharp, \End(E^\sharp))}  \ar{rrr}[swap]{\tr} \ar{dll} & & & \Def_{\dR(\O^\sharp, \O^\sharp)} \ar{dll} \\[-\perspective]
	\Def_{E^\sharp} \ar{rrr}[swap]{\det} \ar[from=uu,crossing over] & & & \Def_{\det(E^\sharp)} \ar[from=uu,crossing over]
	\end{tikzcd}\]
	The front face of this cube evidently 2-commutes; the top and bottom faces 2-commute by \cref{result:trace and determinant}, and the left and right faces 2-commute by \cref{result:equivalences and base change}. 
	
	The conclusion of this is that the map $H^2_c(\End(E)) \to H^2_c(\O)$ induced by the trace map $\tr : \dR(\O^\sharp, \End(E^\sharp)) \to  \dR(\O^\sharp, \O^\sharp)$ is compatible with the morphism of deformation categories $\det : \Def_E^\sharp \to \Def_{\det(E)}^\sharp$ in the sense of \cref{definition:compatible}. Clearly the map $H^2_c(\End(E)) \to H^2_c(\O)$ induced by the trace map is exactly $\langle 1, - \rangle$, as one can see from our explicit description of the pairing above in \cref{example:trace pairing}. Compatibility therefore says precisely that
	\[ \langle 1, o(\pi, (F, \theta)) \rangle = o(\pi, \det(F, \theta)). \]
	We know from \cref{rank 1 smoothness} that $\Def_{\det(E)}^{\sharp}$ is smooth, so $o(\pi, \det(F, \theta)) = 0$. 
\end{proof}

\begin{definition}[Duality pairing] \label{definition:duality pairing parabolic}
	One can define a \emph{duality pairing}
	\[ \begin{tikzcd} H^1_p(E^\vee) \times H^1_p(E) \ar{r} & H^2_c(\O). \end{tikzcd} \]
	as follows. Suppose $\phi \in H^1_p(E^\vee)$ and $e \in H^1_p(E)$. Then there exists $\alpha \in (E^\sharp)^\vee$ and $f \in E^\sharp$ such that $\partial \alpha = \phi^\sharp$ and $\partial f = e^\sharp$, and we define
	\begin{equation} \label{eqn:duality pairing parabolic definition} \langle \phi, e \rangle = \alpha(e^\sharp) - \phi^\sharp(f) = \alpha(\partial f) - (\partial \alpha)(f).
	\end{equation}
\end{definition}

\begin{proof}[Proof that this pairing is well-defined]
	Suppose first that we have $\alpha, \alpha'$ such that $\partial \alpha = \partial \alpha' = \phi^\sharp$. Then $\alpha - \alpha'$ is horizontal, so
	\[ \begin{aligned} (\alpha(\partial f) - (\partial \alpha)(f)) - (\alpha'(\partial f) - (\partial \alpha')(f)) &= (\alpha - \alpha')(\partial f) \\
	&= \partial ((\alpha - \alpha')(f) )  \end{aligned} \]
	which is in the image of $\partial : \O^\sharp \to \O^\sharp$. Similarly, if we have $f, f'$ such that $\partial f = \partial f' = e^\sharp$, then $f - f'$ is horizontal and 
	\[ \begin{aligned} (\alpha(\partial f) - (\partial \alpha)(f)) - (\alpha(\partial f') - (\partial \alpha)(f')) &= (\partial \alpha)(f' - f) \\ &= \partial (\alpha(f' - f)), \end{aligned} \]
	which again is in the image of $\partial : \O^\sharp \to \O^\sharp$. In other words, formula \eqref{eqn:duality pairing parabolic definition} yields a well-defined pairing 
	\[ \begin{tikzcd} \{ \phi \in E^\vee : \phi^\sharp \in \im(\partial) \} \times \{ e \in E : e^\sharp \in \im(\partial) \} \ar{r} & H^1_{\dR}(\O^\sharp). \end{tikzcd} \]
	Now observe that if $\phi = \partial \psi$ for some $\psi \in E^\vee$, then 
	\[ \begin{aligned} \langle \phi, e \rangle &= \psi^\sharp(\partial f) - (\partial \psi^\sharp)(f) \\
	&= 2\psi(e)^\sharp - \partial(\psi^\sharp(f)), \end{aligned} \]
	where the first term is in the image of $\O \to \O^\sharp$ and the second in the image of $\partial :  \O^\sharp \to \O^\sharp$. Thus $\langle \phi, e \rangle$ vanishes in $H^2_c(\O)$. Similarly, if $e = \partial h$ for some $h \in E$, then 
	\[ \begin{aligned} \langle \phi, e \rangle &= \alpha(\partial h^\sharp) - (\partial \alpha)(h^\sharp) \\
	&= \partial(\alpha(h^\sharp)) - 2\phi(h)^\sharp
	 \end{aligned} \]
	and now the first term is in the image of $\partial : \O^\sharp \to \O^\sharp$ and the second is in the image of $\O \to \O^\sharp$, so again $\langle \phi, e \rangle$ vanishes in $H^2_c(\O)$. 
\end{proof}

\begin{lemma} \label{result:duality pairing is alternating}
	Under the identification $\End(E)^\vee = \End(E)$ resulting from the trace pairing as in \cref{example:trace pairing}, the duality pairing
	\[ \begin{tikzcd} H^1_p(\End(E)) \times H^1_p(\End(E)) \ar{r} & H^2_c(\O) \end{tikzcd} \]
	is alternating. 
\end{lemma}

\begin{proof}
	Unwinding everything, we find that the duality pairing in this case is described as follows. Given $\phi, \psi \in H^1_p(\End(E))$, we find $\alpha, \beta \in \End(E)$ such that $\partial \alpha = \phi^\sharp$ and $\partial \beta = \psi^\sharp$, and then
	\[ \langle \phi, \psi \rangle = \tr( \alpha \circ (\partial \beta) - (\partial \alpha) \circ \beta). \]
	If $\phi = \psi$, then we can take $\alpha = \beta$ and we see that $\langle \phi, \phi \rangle$ is the trace of a commutator, which must vanish. 
\end{proof}
	
	\section{Preliminaries on isocrystals} \label{isocrystal prelims}

\begin{notation} \label{main notation}
	Suppose now that $K$ is a complete discrete valuation field of characteristic 0 with perfect residue field $k$ of positive characteristic $p$. Let $X$ denote the projective line over $k$, $Z$ an effective Cartier divisor in $X$ and $U = X \setminus Z$ its complement. Let $\frX$ denote the formal projective line over the ring of integers $K^\circ$ in $K$. We define 
	\[ \O := \Gamma(\frX, \O_{\frX}(^\dagger Z)_\Q). \]
	In other words, $\O$ is the ring of overconvergent functions on the tube $]U[$ inside the generic fiber $\frX_K$. Fix a coordinate $t$ on $\frX$ such that $\infty \in Z$. If $h \in K^\circ[t]$ is a separable polynomial whose reduction in $k[t]$ vanishes along $Z$, then $\O = K\langle t, h^{-1}\rangle^\dagger$. More precisely, if $k\langle t, u\rangle^\dagger$ is the two variable Washnitzer algebra (as in \cite[section 7.5]{fresnel} or \cite[section 1.2]{grosse-klonne}), then $\O$ is the quotient of $k\langle t, u \rangle^\dagger$ by the ideal generated by $hu - 1$. 
	
	Let $\partial$ denote the derivation on $\O$ dual to $dt$. Then $\ker(\partial) = K$. Let $D$ denote the corresponding ring of differential operators. 
\end{notation}

\subsection{Properties of $\O$}

\begin{proposition} \label{pid}
	$\O$ is a principal ideal domain, and every maximal ideal of $\O$ is generated by an irreducible $f \in K[t]$. 
\end{proposition}

\begin{proof}
	We first check that $\O$ is a Dedekind domain. Note that $\O$ is noetherian \cite[discussion following lemma 7.5.1]{fresnel}, so it suffices to show that $\O_\m$ is regular of Krull dimension 1 for every maximal ideal $\m \subseteq \O$. The noetherian local ring $\O_\m$ is regular of Krull dimension 1 if and only if its $\m$-adic completion $(\O_\m)^\wedge$ is regular of Krull dimension 1 \cite[corollary 11.19, proposition 11.24]{atiyah}. But $(\O_\m)^\wedge = (\O'_\m)^\wedge$, where $\O'$ denotes the completion of $\O$ for the Gauss norm \cite[theorem 1.7(2)]{grosse-klonne}. Note that $\O' = K\langle t, h^{-1}\rangle = K\langle t, u \rangle/(hu-1)$ for some $h$, so $\m$ corresponds to a maximal ideal of $K\langle t \rangle$ which we abusively denote by $\m$ again \cite[section 4.1]{fresnel}. Then note that $(\O'_\m)^\wedge = (K\langle t \rangle_{\m})^\wedge$ \cite[remark  4.1.5(2)]{fresnel}. Now $(K\langle t \rangle_\m)^\wedge$ is regular of Krull dimension 1 since $K\langle t \rangle$ is \cite[theorem 3.2.1(2)]{fresnel}, and this completes the proof that $\O$ is a Dedekind domain.
	
	To show that $\O$ is a principal ideal domain, it suffices to show that every prime ideal is principal (cf. \cite[proposition 3.17]{lam}). Clearly the zero ideal is principal. Since $\O$ is a Dedekind domain, every nonzero prime ideal is maximal, so it suffices to show that every maximal ideal $\m$ of $\O$ is principal. Note that $\m \cap K[t]^\dagger$ is a maximal ideal of $K[t]^\dagger$, so it must be generated by a polynomial $f \in K[t]$ \cite[proposition 1.5]{grosse-klonne}. Injectivity of the map $\text{MaxSpec}(\O) \to \text{MaxSpec}(K[t]^\dagger)$ (cf. \cite[remark 4.1.5(2)]{fresnel}) implies $\m = f\O$. Clearly $f$ must be irreducible.
\end{proof}

\begin{corollary}
 $\O$ is differentially simple.
\end{corollary}

\begin{proof}
	Since $\O$ is a principal ideal domain by \ref{pid}, it suffices to check that no maximal ideal $\m$ of $\O$ is a $D$-submodule \cite[4.4]{block}. But we know that $\m = f\O$ for an irreducible $f \in K[t]$. If $\m$ were a $D$-submodule, then we would have $\partial(f) \in \m$ also. But since $f$ is irreducible, there exist $a, b \in K[t]$ such that $af + b\partial(f) = 1$, which then forces $\m$ to be the unit ideal, yielding a contradiction.
\end{proof}

This has a number of consequences. 
\begin{itemize}
	\item The ring $D$ has left global dimension 1 \cite[5]{goodearl}.
	
	\item Every differential $\O$-module finite over $\O$ is finite \emph{free} over $\O$. Indeed, projectivity follows from \cite[4.3]{maurischat}, and then freeness follows from \ref{pid}. It follows immediately that the category of finite free differential $\O$-modules is abelian.  
	
	\item For any finite free differential $\O$-module $E$, we have \[ \dim H^0_{\dR}(E) \leq \rank E. \]
	Indeed, suppose $S$ is a $K$-basis for $H^0_{\dR}(E)$. Then $1 \otimes S$ is still $K$-linearly independent in $H^0_{\dR}(L \otimes_\O E)$, where $L := \Frac(\O)$. But note that $H^0_{\dR}(L) = K$ \cite[3.1]{maurischat}. Thus
	\[ |S| = |1 \otimes S| \leq \dim_K H^0_{\dR}(L \otimes_\O E) \leq \dim_{L} (L \otimes_\O E) = \rank E, \]
	where the first equality is because $E \to L \otimes_\O E$ is injective, the first inequality is because $1 \otimes S$ is linearly independent in $H^0_{\dR}(L \otimes_\O E)$, and the second inequality is a consequence of \cite[5.1.5]{kedlaya_pde}.
\end{itemize}

\begin{definition}
	For any $z \in Z$, let $\Robba_z$ denote the Robba ring at $z$ (see, for instance, \cite[definition 3.1]{lestum}). Observe that there are natural flat restrictions maps $\O \to \Robba_z$ for each $z$. If $E$ is a differential $\O$-module, then 
	\[ E_z := \Robba_z \otimes_\O E \]
	is the \emph{Robba fiber} of $E$ at $z$. It is a differential $\Robba_z$-module, and the functor $E \mapsto E_z$ is exact. 
\end{definition}

\begin{notation}
	We define $\MC^f(\tube{U})$ to be the category of finite free differential $\O$-modules, and $\MC^\dagger(\tube{U})$ to be the full subcategory of $\MC^f(\tube{U})$ whose objects are the differential $\O$-modules whose Robba fibers are solvable.\footnote{More precisely, ``solvability'' here means ``solvability at 1'' in the sense of \cite[d\'efinition 4.1--1]{christol_mebkhout3} or \cite[definition 12.6.1]{kedlaya_pde}.}
\end{notation}

\begin{proposition} \label{result:translate isocrystals}
	If $\Isoc^\dagger(U)$ denotes the category of overconvergent isocrystals on $U$ over $K$ as in \cite[d\'efinition 2.3.6]{berthelot_rigide2}, then there is a tensor equivalence  \[ \begin{tikzcd} \Isoc^\dagger(U) \ar{r}{\iota} & \MC^\dagger(\tube{U}). \end{tikzcd} \] If $E \in \Isoc^\dagger(U)$, then $E$ is irreducible if and only if $\iota(E)$ is irreducible in $\DMod_\O$, and \[ H^i_\rig(U, E) = H^i_{\dR}(\iota(E)) \]
	where $H^i_\rig(U, E)$ is absolute rigid cohomology \cite[definition 8.2.5]{lestum_book}.
\end{proposition}

\begin{proof}
	The first statement is more or less well-known (cf. \cite[6.7--6.8]{lestum}). The only point to note here is that, since $\O$ is a principal ideal domain by \ref{pid}, finite differential $\O$-modules must actually be \emph{free} (and not just projective, as is the case in \cite{lestum}). The second assertion follows from \cite[2.2.7(c)]{berthelot_rigide2}, and the third from \cite[equation (8.1.1)]{crew_finiteness}.
\end{proof}

Hereafter, we will freely regard overconvergent isocrystals on $U$ as differential $\O$-modules with solvable Robba fibers using the equivalence $\iota$ above. 

\subsection{Compactly supported and parabolic cohomology} 
\label{section:compactly supported and parabolic}

Let 
\[ \O^\sharp := \prod_{z \in Z} \Robba_z.  \]
We have a natural homomorphism of differential rings $\O \to \O^\sharp$. This allows us to use the notation introduced in \cref{definition:abstract compactly supported and parabolic cohomologies}. 

\begin{lemma}[{\cite[section 8.1]{crew_finiteness}}]
	Suppose $E \in \Isoc^\dagger(U)$. Then 
	\[ H^i_{\rig, c}(U, E) = H^i_c(\iota(E)) \]
	where $\iota$ is the equivalence of \cref{result:translate isocrystals} and $H^i_{\rig, c}(U, E)$ is absolute rigid cohomology with compact supports \cite[definition 8.2.5]{lestum_book}. 
\end{lemma}

\begin{definition}[{\cite[equation (8.1.5)]{crew_finiteness}}]
	For $E \in \Isoc^\dagger(U)$, we define the \emph{parabolic cohomology} of $E$, denoted $H^1_{\rig,p}(U, E)$, to be $H^1_p(\iota(E))$, where $\iota$ is the equivalence of \cref{result:translate isocrystals}.
\end{definition}

Observe that, for any finite free differential $\O$-module $E$, we have a distinguished triangle
\[ \begin{tikzcd} C^+(E) \ar{r} &  \dR(\O, E) \ar{r} & \displaystyle \prod_{z \in Z} \dR(\Robba_z, E_z) \ar{r}{+} & \mbox{} \end{tikzcd} \]
where $C^+(E)$ is as in \cref{definition:abstract compactly supported and parabolic cohomologies}. This gives rise to a long exact sequence as follows, where $H^0_c(E) = 0$ by \cref{result:zeroth compactly supported cohomology vanishes} since the homomorphism $\O \to \O^\sharp = \prod_{z \in Z} \Robba_z$ is injective. 
\begin{equation} \label{eqn:crew six term exact sequence}
\begin{tikzcd} 0 \ar{r} & \cancel{H^0_c(E)} \ar{r} & H^0_{\dR}(E) \ar{r} \arrow[d, phantom, ""{coordinate, name=Z}] & \displaystyle \prod_{z \in Z} H^0_{\dR}(E_z) \arrow[dll,
rounded corners,
to path=
{ -- ([xshift=2ex]\tikztostart.east)
	|- (Z) [near end]\tikztonodes
	-| ([xshift=-2ex]\tikztotarget.west)
	-- (\tikztotarget)}
] \\
& H^1_c(E) \ar{r} & H^1_{\dR}(E) \ar{r} \arrow[d, phantom, ""{coordinate, name=Y}] & \displaystyle \prod_{z \in Z} H^1_{\dR}(E_z) \arrow[dll,
rounded corners,
to path=
{ -- ([xshift=2ex]\tikztostart.east)
	|- (Y) [near end]\tikztonodes
	-| ([xshift=-2ex]\tikztotarget.west)
	-- (\tikztotarget)}
] \\
& H^2_c(E) \ar{r} & 0 \end{tikzcd}
\end{equation}
This is Crew's six-term exact sequence \cite[equation (8.1.4)]{crew_finiteness}.

\subsection{Duality pairing}

There is a trace map $H^2_c(\O) \to K$ \cite[equation (8.1.7)]{crew_finiteness} which is an isomorphism \cite[discussion following theorem 9.5]{crew_finiteness}. For $E \in \MC^f(\tube{U})$, composing the duality pairing of \cref{definition:duality pairing} with the trace map yields exactly the pairing
\begin{equation} \label{eqn:duality pairing} \begin{tikzcd} H^0_{\dR}(E^\vee) \times H^2_c(E) \ar{r} & K \end{tikzcd} \end{equation}
of \cite[equation (8.1.8)]{crew_finiteness}. Similarly, composing the duality pairing of \cref{definition:duality pairing parabolic} with the trace map yields exactly the pairing
\begin{equation} \label{eqn:duality pairing parabolic} \begin{tikzcd} H^1_p(E^\vee) \times H^1_p(E) \ar{r} & K \end{tikzcd} \end{equation}
of \cite[equation (8.1.9)]{crew_finiteness}. In order for these pairings to be perfect, we need the following definition. 

\begin{definition}
	We say that $E$ is \emph{strict} if the Robba fiber $E_z$ is strict\footnotemark\ for all $z \in Z$. We write $\MC^s(\tube{U})$ for the full subcategory of $\MC^f(\tube{U})$ whose objects are the strict $E \in \MC^f(\tube{U})$. 
\end{definition}

\footnotetext{A differential module $E$ over the Robba ring is \emph{strict} if $H^1_{\dR}(E)$ is finite dimensional. A more technical definition is given by Crew in \cite[discussion preceding theorem 6.3]{crew_finiteness}, but \cite[theorem 6.3]{crew_finiteness} combines with the more recent observation of Crew \cite[lemma 1]{crew_rigidity} to show that the two definitions are equivalent.}

\begin{lemma} \label{result:stability of strictness}
	$\MC^s(\tube{U})$ is a Serre subcategory \tag{02MN} of $\MC^f(\tube{U})$ which contains the unit object $\O$ and is stable under duality. 
\end{lemma}

\begin{proof}
	Since $E \mapsto E_z$ is an exact tensor functor, it is sufficient to show that the category $\DMod^s_{\Robba}$ of strict differential modules over the Robba ring $\Robba$ is a Serre subcategory of the category $\DMod^f_{\Robba}$ of finite free differential modules over $\Robba$ which contains the unit object $\Robba$ and is stable under duality. The fact that $\Robba$ is strict is clear, since $\dim H^1_{\dR}(\Robba) = 1$. The fact that $\DMod_\Robba^s$ is stable under duality is \cite[theorem 6.3]{crew_finiteness}. If 
	\[ \begin{tikzcd} 0 \ar{r} & E' \ar{r} & E \ar{r} & E'' \ar{r} & 0 \end{tikzcd} \]
	is an exact sequence in $\DMod^f_{\Robba}$, it induces an exact sequence as follows. 
	\[ \begin{tikzcd} 0 \ar{r} & H^0_{\dR}(E') \ar{r} & H^0_{\dR}(E) \ar{r} \arrow[d, phantom, ""{coordinate, name=Z}] & H^0_{\dR}(E'') \arrow[dll,
	rounded corners,
	to path=
	{ -- ([xshift=2ex]\tikztostart.east)
		|- (Z) [near end]\tikztonodes
		-| ([xshift=-2ex]\tikztotarget.west)
		-- (\tikztotarget)}
	] \\
	& H^1_{\dR}(E') \ar{r} & H^1_{\dR}(E) \ar{r} \arrow[d, phantom, ""{coordinate, name=Y}] & H^1_{\dR}(E'') \arrow[dll,
	rounded corners,
	to path=
	{ -- ([xshift=2ex]\tikztostart.east)
		|- (Y) [near end]\tikztonodes
		-| ([xshift=-2ex]\tikztotarget.west)
		-- (\tikztotarget)}
	] \\
	& 0 & \mbox{} \end{tikzcd} \]
	Since $H^0_{\dR}(E'')$ is always finite dimensional \cite[proposition 6.2]{crew_finiteness}, it follows that $H^1_{\dR}(E)$ is finite dimensional if and only if $H^1_{\dR}(E')$ and $H^1_{\dR}(E'')$ are finite dimensional. Thus $\DMod^s_\Robba$ is a Serre subcategory. 
\end{proof}

\begin{remark} \label{remark:strict robba modules}
	Let us make some further observations about the category $\DMod^s_\Robba$ of strict differential modules over the Robba ring $\Robba$. For any $\alpha \in K$, let $E_\alpha$ denote the differential module defined by the differential equation $t\partial - \alpha$. Then $E_\alpha \in \DMod_\Robba^s$ if and only if $\alpha$ is $p$-adic non-Liouville \cite[proposition 6.10]{crew_finiteness}. 
	\begin{itemize}
		\item $\DMod^s_\Robba$ is \emph{not} stable under tensor products: if we choose $p$-adic non-Liouville numbers $\alpha, \beta$ whose sum $\alpha + \beta$ is $p$-adic Liouville, then $E_{\alpha}$ and $E_{\beta}$ are both strict but $E_{\alpha} \otimes E_{\beta} \simeq E_{\alpha + \beta}$ is not. 
		
		\item $\DMod^s_\Robba$ is incomparable with the category $\DMod^\dagger_\Robba$ of solvable differential modules over $\Robba$ \cite[d\'efinition 4.1--1]{christol_mebkhout3}. Indeed, $E_\alpha \in \DMod^\dagger_\Robba$ if and only if $\alpha \in \Z_p$ \cite[example 9.5.2]{kedlaya_pde}. So, for example, if $\alpha \in \Z_p$ is $p$-adic Liouville, then $E_\alpha$ is solvable but not strict. Conversely, if $\alpha \in K \setminus \Z_p$, then $E_\alpha$ is strict but not solvable. 
	\end{itemize}
\end{remark}

\begin{remark} \label{remark:frobenius robba modules}
	Let $\DMod_\Robba^F$ denote the category of finite differential modules over the Robba ring $\Robba$ that can be equipped with Frobenius structures potentially after a finite extension of $K$ \cite[d\'efinition 2.5--2]{christol_mebkhout4}. By way of example, if $E_\alpha$ is as in \cref{remark:strict robba modules} above, then $E_\alpha \in \DMod_\Robba^F$ if and only if $\alpha \in \Z_{(p)}$ \cite[corollaire 6.0--23]{christol_mebkhout4}. 
	
	It follows immediately from \cite[corollaire 6.0--20]{christol_mebkhout4} that $\DMod^F_\Robba$ is a tannakian subcategory of the tannakian\footnotemark\ category $\DMod_\Robba^f$ over $K$. Moreover, we have
	\[ \DMod_\Robba^F \subseteq \DMod_\Robba^\dagger \cap \DMod_\Robba^s. \]
	The inclusion $\DMod_\Robba^F \subseteq \DMod_\Robba^\dagger$ is a theorem of Dwork's \cite[theorem 17.2.1]{kedlaya_pde}. The inclusion $\DMod_\Robba^F \subseteq \DMod_\Robba^s$ is a consequence of the $p$-adic local monodromy theorem (which is due independently to Andr\'e \cite{andre_hasse-arf}, Kedlaya \cite[theorem 20.1.4]{kedlaya_pde}, and Mebkhout \cite{mebkhout}).
	
	\footnotetext{\label{footnote:finite diff robba modules tannakian}It follows from \cite[proposition 6.1]{crew_finiteness} that the category $\DMod_\Robba^f$ of finite differential modules over $\Robba$ is stable under subquotients, extensions, tensor products, and internal homs, and it evidently contains the unit object $\Robba$. The endomorphisms of the unit object $\Robba$ are $H^0_{\dR}(\Robba) = K$, and the ``tannakian dimension'' \cite[section 7]{deligne_tannakiennes} coincides with the rank as a finite free module over $\Robba$, which is always a nonnegative integer. Thus $\DMod_\Robba^f$ is tannakian over $K$ \cite[theorem 7.1]{deligne_tannakiennes}.}
\end{remark}

\begin{remark} \label{remark:frobenius implies strict}
	In particular, it follows from \cref{remark:frobenius robba modules} that if $E \in \MC^f(\tube{U})$ can be equipped with a Frobenius structure, it is overconvergent and its Robba fibers along $Z$ are strict. This is the most important case. 
\end{remark}

\begin{theorem}[Crew's finiteness theorem {\cite[theorem 9.5]{crew_finiteness}}] \label{result:crew finiteness}
	Suppose $E \in \MC^s(\tube{U})$.\footnotemark\ Then all of the terms in the exact sequence \eqref{eqn:crew six term exact sequence} are finite dimensional, and both of the duality pairings \eqref{eqn:duality pairing} and \eqref{eqn:duality pairing parabolic} are perfect.   
\end{theorem}

\footnotetext{Technically, \cite[theorem 9.5]{crew_finiteness} is only stated for $E \in \MC^f(\tube{U})$ both strict and overconvergent, but overconvergence is not used anywhere in the proof.}

\subsection{Dimension of parabolic cohomology} \label{section:dimension parabolic}

\newcommand{\artin}{\textnormal{Ar}}
\newcommand{\irr}{\textnormal{Irr}}

\begin{definition}
	For $E \in \MC^\dagger(\tube{U})$ and $z \in Z$, we define the \emph{Artin conductor} of $E$ at $z$, denoted $\artin_z(E)$, by the formula
	\[ \artin_z(E) = \irr_z(E) + \rank(E) - \dim H^0_{\dR}(E_z), \]
	where $\irr_z(E)$ is the $p$-adic irregularity of the Robba fiber $E_z$ \cite[d\'efinition 8.3--8]{christol_mebkhout3}. Observe that we always have $\irr_z(E) \geq 0$ and $\dim H^0_{\dR}(E_z) \leq \rank(E)$, so $\artin_z(E) \geq 0$. 
\end{definition}

\begin{lemma} \label{result:dimension parabolic}
	If $E \in \MC^\dagger(\tube{U})$ admits a Frobenius structure, we have
	\[ \dim H^1_p(E) = \dim H^0_{\dR}(E) + \dim H^2_c(E) - 2 \rank(E) + \sum_{z \in Z} \artin_z(E). \]
\end{lemma}

\begin{proof}
	Crew's six-term exact sequence \eqref{eqn:crew six term exact sequence} and the definition of parabolic cohomology give us an exact sequence as follows.
	\[ \begin{tikzcd} 0 \ar{r} & H^0_{\dR}(E) \ar{r} & \displaystyle \prod_{z \in Z} H^0_{\dR}(E_z) \ar{r} & H^1_c(E) \ar{r} & H^1_p(E) \ar{r} & 0 \end{tikzcd} \]
	Taking dimensions, we find that
	\[ \dim H^1_p(E) = \dim H^1_c(E) - \sum_{z \in Z} \dim H^0_{\dR}(E_z) + \dim H^0_{\dR}(E). \]
	The Christol-Mebkhout index formula \cite[th\'eor\`eme 8.4--1]{christol_mebkhout3} says that
	\[ -\dim H^1_c(E) + \dim H^2_c(E) = \chi_c(U, E) = \chi_c(U) \rank(E) - \sum_{z \in Z} \irr_z(E) \]
	where $\chi_c(U, E) = 2 - \# Z$. We now put these equations together. 
\end{proof}

\begin{remark}
	Note that the above calculation also applies more generally under the same ``non-Liouville hypotheses'' that are necessary for the Christol-Mebkhout index formula \cite[th\'eor\`eme 8.4--1]{christol_mebkhout3}.
\end{remark}

\subsection{Parabolic cohomology and restriction to open subsets} \label{section:parabolic restriction}

The following shows that parabolic cohomology is an invariant of the ``generic fiber'' of an isocrystal. In particular, if $E \in \Isoc^\dagger(U)$ is cohomologically rigid in the sense of \cref{cohomologically rigid isocrystal} below, then so is $E|_V$ for any dense open $V \subseteq U$. 

\begin{proposition} \label{result:parabolic restriction}
	For any dense open subset $V \subseteq U$ and $E \in \MC^\dagger(\tube{U})$, there is a natural isomorphism 
	\[ H^1_p(E) = H^1_p(E|_V). \]
\end{proposition}

\begin{proof}
	We will use the identification $\MC^\dagger(\tube{U}) = \Isoc^\dagger(U)$ of \cref{result:translate isocrystals} in order to apply cohomological machinery like excision and so forth. Note that we have a commutative diagram as follows.
	\begin{equation} \label{square}
	\begin{tikzcd} H^1_{\rig,c}(V, E|_V) \ar[twoheadrightarrow]{r} \ar{d} & H^1_p(E|_V) \ar[hookrightarrow]{r} & H^1_\rig(V, E|_V) \\ 
	H^1_{\rig,c}(U, E) \ar[twoheadrightarrow]{r} & H^1_p(E) \ar[hookrightarrow]{r} & H^1_\rig(U, E) \ar{u} \end{tikzcd}
	\end{equation}
	We construct an isomorphism $\tau : H^1_p(E|_V) \to H^1_p(E)$ by doing the only thing one could think to do in this setting: we define $\tau(\alpha)$ for $\alpha \in H^1_p(E|_V)$ to be the image in $H^1_p(E)$ of a lift $\alpha' \in H^1_{\rig,c}(V, E|_V)$ of $\alpha$. The content of this proof is to check that this actually defines a bijection, and then $K$-linearity follows immediately. 
	
	Let $S := U \setminus V$, and note that we have an excision exact sequence for compactly supported rigid cohomology \cite[proposition 2.5.1]{tsuzuki_gysin}, and that $H^i_{\rig,c}(S, E|_S) = 0$ for all $i > 0$ since $\dim(S) = 0$. 
	\[ \begin{tikzcd} \dotsb \ar{r} & H^1_{\rig,c}(S, E|_S) \ar{r} & H^1_{\rig,c}(U, E) \ar{r} & \cancel{H^1_{\rig,c}(S, E|_S)} \ar{r} & \dotsb \end{tikzcd} \]
	It follows that the vertical map $H^1_{\rig,c}(V, E|_V) \to H^1_{\rig,c}(U, E)$ on the left-hand side of the commutative square \eqref{square} is surjective. 
	
	We can analogously show that the vertical map $H^1_\rig(U, E) \to H^1_\rig(V/K, E|_V)$ on the right-hand side of the square \eqref{square} is injective. We again have an excision exact sequence \cite[proposition 2.1.1(3)]{tsuzuki_gysin}, and $H^i_{S,\rig}(U, E) = 0$ for all $i \neq 2$ \cite[corollary 4.1.2]{tsuzuki_gysin}.
	\[ \begin{tikzcd} \dotsb \ar{r} & \cancel{H^1_{Z,\rig}(U, E)} \ar{r} & H^1_\rig(U, E) \ar{r} & H^1_\rig(V, E|_V) \ar{r} & \dotsb \end{tikzcd} \]
	
	Now to see that $\tau$ is well-defined and injective, observe that $\alpha \in H^1_{\rig,c}(V, E|_V)$ vanishes in $H^1_p(E|_V)$ if and only its image in $H^1_p(E)$ vanishes: this is an elementary diagram chase that uses the fact that $H^1_\rig(U, E) \to H^1_\rig(V, E|_V)$ is injective. Surjectivity of $\tau$ follows immediately from surjectivity of $H^1_{\rig,c}(V, E|_V) \to H^1_{\rig,c}(U, E)$. 
\end{proof}

\subsection{Intermediate extensions of arithmetic D-modules}

\newcommand{\perf}{\textnormal{perf}}
\newcommand{\bbD}{\mathbb{D}}
\newcommand{\Soln}{\textnormal{Soln}}

In this subsection, we relate parabolic cohomology to the intermediate extension operation in the theory of arithmetic D-modules. 

Recall from \cref{main notation} that $\frX$ denotes the formal projective line over $K^\circ$. We then have Berthelot's sheaf $\scrD^\dagger_{\frX,\Q}$ of differential operators of infinite order and finite level \cite[2.4]{berthelot1}. Also, for any closed subset $T \subset \P^1_k$, we also have the sheaf $\scrD^\dagger_{\frX,\Q}(^\dagger T)$ of differential operators of infinite order and finite level with overconvergent singularities along $T$ \cite[4.2.5]{berthelot1}. We can and will freely regard overconvergent isocrystals on $\P^1 \setminus T$ as $\O_{\frX}(^\dagger T)_\Q$-coherent $\scrD^\dagger_{\frX}(^\dagger T)_\Q$-modules \cite[th\'eor\`eme 2.2.12]{caro}. 

There is a natural homomorphism $\scrD^\dagger_{\frX,\Q} \to \scrD^\dagger_{\frX}(^\dagger T)_\Q$ of sheaves of rings on $\frX$ and restriction of scalars along this homomorphism is exact and naturally induces a functor on the derived categories of perfect complexes
\[ \begin{tikzcd} D_\perf(\scrD^\dagger_\frX(^\dagger T)_\Q) \ar{r}{j_+} & D_\perf(\scrD^\dagger_{\frX,\Q}), \end{tikzcd} \]
called the \emph{ordinary direct image} along the inclusion $j : \P^1 \setminus T \hookrightarrow \P^1$. This functor is right adjoint to the \emph{ordinary inverse image} functor 
\[ \begin{tikzcd} D_\perf(\scrD^\dagger_{\frX,\Q}) \ar{r}{j^+} &  D_\perf(\scrD^\dagger_\frX(^\dagger T)_\Q) \end{tikzcd} \]
where $j^+E = \scrD^\dagger_\frX(^\dagger T)_\Q \otimes_{\scrD^\dagger_{\frX,\Q}}^{\mathbb{L}}  E$. 

We also have the Verdier duality functor \cite[d\'efinition 3.2]{virrion}
\[ \begin{tikzcd} D_\perf(\scrD^\dagger_\frX(^\dagger T)_\Q) \ar{r}{\bbD} & D_\perf(\scrD^\dagger_\frX(^\dagger T)_\Q). \end{tikzcd} \]
This functor is an involution: there is a natural isomorphism $\bbD^2 = 1$ \cite[chapitre II, th\'eor\`eme 3.6]{virrion}. Further, it commutes with the ordinary inverse image functor: there is a natural isomorphism $\bbD \circ j^+ = j^+ \circ \bbD$ \cite[chapitre II, proposition 4.4]{virrion}. 

We then define the \emph{extraordinary direct image} functor 
\[ j_! := \bbD j_+ \bbD. \]
There is a natural morphism of functors $j_! \to j_+$ \cite[paragraph 1.3.4]{abe_weights}.

Furthermore, we have
\[ \begin{aligned} \rightderived\Gamma_{\rig}(\P^1,-) := \rightderived\Gamma(\frX,  \Omega^\bullet_{\frX} \otimes_{\O_{\frX}} - ) &= \rightderived\Gamma(\frX,  \rightderived\mathscr{H}\mathrm{om}_{\scrD^\dagger_{\frX,\Q}}(\O_{\frX,\Q}, -)) = \rightderived\!\Hom_{\scrD^\dagger_{\frX,\Q}}(\O_{\frX,\Q}, -), \end{aligned} \]
as functors on $D_\perf(\scrD^\dagger_{\frX,\Q})$. Here, we have used the fact that the arithmetic Spencer complex resolves $\O_{\frX,\Q}$ \cite[proposition 4.3.3]{berthelot2}.

Suppose $E \in \Isoc^\dagger(\P^1 \setminus T)$ admits a Frobenius structure. Then $j_+E$ is a holonomic $\scrD^\dagger_{\frX,\Q}$-module \cite[proposition 3.1]{huyghe-trihan}.\footnotemark\ Also, applying the duality functor $\D$ to $E$ yields the usual tannakian dual $E^\vee$, which also admits a Frobenius structure, and duality preserves holonomicity \cite[proposition 2.15]{caro_frobless}, so it follows that $j_!E$ is also a holonomic $\scrD^\dagger_{\frX,\Q}$-module. As in \cite[definition 1.4.1]{abe_weights}, we define the \emph{intermediate extension} $j_{!+}E$ of $E$ by
\[ j_{!+}E := \im(j_!E \to j_+ E). \]
This is also a holonomic $\scrD^\dagger_{\frX,\Q}$-module, since the category of holonomic $\scrD^\dagger_{\frX,\Q}$-modules is abelian \cite[proposition 2.14]{caro_frobless}.

\footnotetext{In \cite{huyghe-trihan}, the $\scrD^\dagger_{\frX}(^\dagger T)$-module associated to an overconvergent isocrystal $E$ with Frobenius structure is denoted $\tilde{\scrD}^\dagger(E)$, and $j_+E$ is denoted $\scrD^\dagger(E)$.}

\begin{theorem} \label{result:parabolic and middle}
	If $E \in \Isoc^\dagger(U)$ admits a Frobenius structure, then
	\[ H^1_{\rig,p}(U, E) = H^1_{\rig}(\P^1, j_{!+}E ), \]
	where $j$ denotes the inclusion $U \hookrightarrow \P^1$.
\end{theorem}

\begin{proof}
	For every $z \in Z = \P^1 \setminus U$, we let
	\[ \Soln_z(E) := H^0_{\dR}((E_z)^\vee) = H^0_{\dR}(\Hom(E_z, \Robba)). \]
	As we noted in \cref{remark:frobenius implies strict}, each $E_z$ is strict since $E$ admits a Frobenius structure. Thus $(E_z)^\vee$ is strict as well by \cref{result:stability of strictness}. Moreover, we have a natural identification \[ \Soln_z(E)^\vee = H^1_{\dR}(E_z)  \]
	using the local duality pairing of \cite[theorem 6.3]{crew_finiteness}. Combining this with \cite[proposition 5.1]{li}\footnotemark\ tells us that we have a natural exact sequence
	\[ \begin{tikzcd} 0 \ar{r} & j_{!+}E \ar{r} & j_+E \ar{r} & \displaystyle \prod_{z \in Z} i_{z !} H^1_{\dR}(E_z) \ar{r} & 0 \end{tikzcd} \]
	of $\scrD^\dagger_{\frX,\Q}$-modules. We now apply $\rightderived\Gamma_{\rig}(\P^1, -)$ to get a distinguished triangle of vector spaces over $K$, and then we consider the resulting long exact sequence. We know that $H^1_{\rig}(\P^1, j_+E) = H^1_\rig(U, E)$ \cite[corollaire 4.1.7]{berthelot_survey}. Together with \cref{de rham of skyscraper} below, we obtain an exact sequence
	\[ \begin{tikzcd} 0 \ar{r} & H^1_{\rig}(\P^1, j_{!+}E) \ar{r} & H^1_\rig(U, E) \ar{r} & \displaystyle \prod_{z \in Z} H^1_{\dR}(E_z). \end{tikzcd} \]
	Comparing against Crew's six-term exact sequence \eqref{eqn:crew six term exact sequence} and the definition of parabolic cohomology, we obtain the result.
\end{proof}

\footnotetext{
	There is a minor error in \cite{li}. Lemma 4.2 in \emph{loc. cit.} should state that \[ H^{1-s}(i^!\mathscr{E}) = \Ext^s_{\scrD^\dagger}(\mathscr{E},\O^\an)^\vee \]
	for $s = 0, 1$ (in \emph{loc. cit.}, the dual seems to be missing), and then the same correction applies to lemma 4.3. The result of this is that proposition 5.1 should assert that 
	\[ j_!\mathscr{E}/j_{!+}\mathscr{E} = \delta_\alpha \otimes_{K_\alpha} \Soln_\alpha^\vee \]
	(again, the dual in \emph{loc. cit.} is missing). Indeed, the fourth display in the proof should say
	\[ \Hom_{\scrD_{\overline{\frX}}(\infty)}(j_{!+}\mathscr{E},\O^\an) = \Ext^1_{\scrD_{\overline{\frX}}(\infty)}(j_+\mathscr{E}/j_{!+}\mathscr{E},\O^\an) = H^0(i_\alpha^!i_{\alpha!}V)^\vee = V^\vee \]
	using the aforementioned correction of lemma 4.3, so then the final display of the proof should start with $V^\vee$ (instead of $V$). These corrections do not affect the main result of \emph{loc. cit.}, since for that only dimensions matter. 
}

\begin{lemma} \label{de rham of skyscraper}
	For any closed point $z \in \P^1$ and any vector space $V$ over $K$, we have
	\[ H^i_{\rig}(\P^1, i_{z!}V) = \begin{cases} V & \text{if } i = 1, \text{ and} \\ 0 & \text{otherwise.} \end{cases} \]
\end{lemma}

\begin{proof}
	Let $\delta_z$ be the $\scrD^\dagger_{\frX,\Q}$-module which is a skyscraper sheaf at $z$ with
	\begin{equation} \label{delta}
	\Gamma(Y, \delta_z) = \left\{ \left. \sum_{i = 0}^\infty a_i \partial ^{[i]} \, \right| \begin{matrix} a_i \in K \text{ for all } i \in \N \text{ and there exist } c > 0 \text{ and } \\ 0 < \eta < 1 \text{ such that } |a_i| < c\eta^i \text{ for all } i \in \N  \end{matrix} \right\}
	\end{equation}
	for any affine open neighborhood $Y$ of $z$. We then have $i_{z!}V = \delta_z \otimes V$ for any vector space $V$ over $K$. Thus it suffices to prove the assertion of the lemma when $V = K$. 
	
	In other words, we would like to compute $\rightderived\Gamma(\frX,\Omega^\bullet_{\frX} \otimes_{\O_{\frX}} \delta_z)$. Since $\delta_z$ is a skyscraper sheaf at $z$, it is sufficient to compute the cohomology of the complex
	\[ \begin{tikzcd} \Gamma(Y, \delta_s) \ar{r}{\nabla} & \Gamma \left(Y, \Omega^1_{\frX} \otimes_{\O_{\frX}} \delta_z \right) \end{tikzcd} \]
	where $Y$ is an affine open neighborhood of $z$. Using the description of $\delta_z$ given above in \cref{delta}, we see that
	and $\nabla$ is given by $P \mapsto dt \otimes \partial P$ for $P \in \Gamma(Y, \delta_s)$. It is clear from this description that $\ker(\nabla) = 0$ and $\coker(\nabla) = K$, spanned by the image of $dt \otimes 1$.
\end{proof}

	\section{Deformations of isocrystals} \label{chapter:deformations of isocrystals}

\subsection{Deformations of isocrystals}

We continue to use \cref{main notation}.

\begin{theorem} \label{result:main part 1}
	Suppose $E \in \MC^f(\tube{U})$. Then we have the following.
	\[ \begin{aligned} 
	& \Inf(\Def_E) = H^0_{\dR}(\End(E)) & & T(\Def_E) = H^1_{\dR}(\End(E)) \\ 
	& \Inf(\Def_E^{\sharp,+}) = 0 & & T(\Def_E^{\sharp,+}) = H^1_{c}(\End(E)) \\
	& \Inf(\Def_E^\sharp) = H^0_{\dR}(\End(E)) & & T(\Def_E^{\sharp}) = H^1_{p}(\End(E))
	\end{aligned} \]
	$\Def_E$ is smooth, and $H^2_c(\End(E))$ is compatibly an obstruction space for both $\Def_E^{\sharp,+}$ and $\Def_E^\sharp$. Moreover, if $\End(E)$ is strict, then all three of the deformation categories $\Def_E$, $\Def_E^{\sharp,+}$ and $\Def_E^\sharp$ have hulls, and the duality pairing on $H^1_p(\End(E))$ is symplectic, so $\dim T(\Def_E^\sharp)$ is even. 
\end{theorem}

\begin{proof}
	The observations about $\Def_E$ follow immediately from \cref{result:governing complex connection}. The observations about $\Def_E^{\sharp,+}$ and $\Def_E^\sharp$ are consequences of \cref{result:zeroth compactly supported cohomology vanishes,result:infinitesimal deformation theory of sharp categories,}. If $\End(E)$ is strict then all of the tangent spaces above are finite dimensional by Crew's finiteness theorem \ref{result:crew finiteness}, so all of the above functors have hulls by the fundamental theorem of deformation theory \ref{deformation theory}. Finally, we saw in \cref{result:duality pairing is alternating} that the duality pairing on $H^1_p(\End(E))$ is alternating, and Crew's finiteness theorem \ref{result:crew finiteness} guarantees that it is perfect; in other words, it is symplectic. 
\end{proof}

\begin{theorem} \label{result:main part 2}
	Suppose $E \in \MC^f(\tube{U})$ is absolutely irreducible and $\End(E)$ is strict. Then $\Def_E^\sharp$ and $\Def_E^{\sharp,+}$ are both smooth, and $\overline{\Def}_E$ and $\overline{\Def}_E^{\sharp}$ are both prorepresentable. 
\end{theorem}

\begin{proof}
	Observe that
	\[ \dim_K \End_D(E) = \dim_K H^0_{\dR}(\End(E)) \leq \dim_\O \End(E), \]
	so $\End_D(E)$ is finite dimensional. Since $E$ is absolutely irreducible, \[ H^0_{\dR}(\End(E)) = \End_D(E) = K \]
	by \cref{endomorphisms of absolutely irreducibles}. By Crew's finiteness theorem \ref{result:crew finiteness}, we know that $H^2_c(\End(E))$ is dual to $H^0_{\dR}(\End(E))$, so 
	\[ \dim H^2_c(\End(E)) = 1. \]
	Applying \cref{obstructions annihilated}, we see that all of the obstruction classes vanish, so $\Def_E^{\sharp}$ is smooth. Thus $\Def_E^{\sharp,+}$ is also smooth by \cref{smoothness of trivialized and trivializable}. 
	
	Since $\End_D(E) = K$, the map $\Aut(F', \theta') \to \Aut(F, \theta)$ is surjective for every $(R', F', \theta') \to (R, F, \theta)$ in $\Def_E$ lying over a surjective $R' \to R$ in $\Art_K$, by \cref{automorphisms lift for endomorphically simples}. Now $\Def_E^{\sharp}$ is a full subcategory of $\Def_E$, so the same is true for every $(R', F', \theta') \to (R, F, \theta)$ in $\Def_E^{\sharp}$. Since $\End(E)$ is strict, we know that \[ H^1_{\dR}(\End(E)) = T(\Def_E) \text{ and } H^1_p(\End(E)) = T(\Def_E^{\sharp}) \] are finite dimensional by Crew's finiteness theorem \ref{result:crew finiteness}. Thus the functors $\overline{\Def}_E$ and $\overline{\Def}_E^{\sharp}$ are prorepresentable by the fundamental theorem of deformation theory \ref{deformation theory}. 
\end{proof}

\subsection{Algebraizing deformations of isocrystals}

We now observe that infinitesimal deformations of an isocrystal can usually be ``algebraized.'' Let $\mathbb{X}$ denote the scheme-theoretic projective line over $K^\circ$ and let $\mathbb{U}$ be an affine open subset whose special fiber is $U$. Let 
\[ \O^\alg := \Gamma(\mathbb{U}_K, \O_{\mathbb{X}_K}) \]
be the ring of algebraic functions on the generic fiber $\mathbb{U}_K$. We can then regard $\O^\alg$ as a finite type $K$-subalgebra of $\O$ which is stable under the derivation $\partial$. For a differential $\O^\alg$-module $E^\alg$, when we write $\Def_{E^\alg}^\sharp$ and $\Def_{E^\alg}^{\sharp,+}$, we mean with respect to the homomorphism 
\[ \begin{tikzcd} \O^\alg \ar{r} & \displaystyle \O^\sharp = \prod_{z \in Z} \Robba_z. \end{tikzcd} \]

\begin{theorem} \label{result:algebraizing infinitesimal deformations}
	Suppose $E \in \MC^f(\tube{U})$ admits a Frobenius structure. Then there exists a differential $\O^\alg$-module $E^\alg$ such that $E = \O \otimes_{\O^\alg} E^\alg$, and the functor
	\[ \begin{tikzcd} \Def_{E^\alg} \ar{r} & \Def_{E} \end{tikzcd} \]
	is an equivalence of deformation categories, as are
	\[ \begin{tikzcd} \Def_{E^\alg}^{\sharp} \ar{r} & \Def_{E}^\sharp \end{tikzcd} \text{ and  } \begin{tikzcd} \Def_{E^\alg}^{\sharp,+} \ar{r} & \Def_{E}^{\sharp,+}. \end{tikzcd} \]
\end{theorem}

\begin{proof}
	Since $E$ admits a Frobenius structure, so does $\End(E)$. Thus all exponents of $E$ and of $\End(E)$ are in $\Z_{(p)}$ \cite[th\'eor\`eme 5.5--3]{christol_mebkhout2}, so the Christol-Mebkhout algebraization theorem \cite[th\'eor\`eme 5.0--10]{christol_mebkhout4} guarantees the existence of $E^\alg$ such that $E = \O \otimes_{\O^\alg} E^\alg$. Moreover, since all of the exponents are in $\Z_{(p)}$, the natural map 
	\[ \begin{tikzcd} \dR(\O^\alg, \End(E^\alg)) \ar{r} & \dR(\O, \End(E)) \end{tikzcd} \]
	is a quasi-isomorphism of differential graded $K$-algebras \cite[chapter 4, proposition 5.2.4]{andre}. We know from \cref{result:governing complex connection} that the domain and codomain govern $\Def_{E^\alg}$ and $\Def_E$, respectively. Moreover, quasi-isomorphisms of differential graded algebras induce isomorphisms on deformation categories by \cref{result:quasi-isomorphism invariance}. This gives us the first equivalence in the statement of the theorem. 
	
	Now observe that we have a 2-commutative diagram of deformation categories as follows.
	\begin{equation} \label{eqn:square} \begin{tikzcd} \Def_{E^\alg} \ar{r} \ar{d} & \Def_E \ar{d} \\ \Def_{E^\sharp} \ar[equals]{r} & \Def_{E^\sharp} \end{tikzcd}
	\end{equation}
	Letting $\Gamma$ denote the residual gerbe of $\Def_{E^\sharp}$, we obtain 2-commutative diagram as follows, in which the square \eqref{eqn:square} is the face on the far right. 
	\setlength{\perspective}{5pt}
	\[ \begin{tikzcd}[row sep={40,between origins}, column sep={40,between origins}]
	&[-\perspective] \Def_{E^\alg}^{\sharp,+} \ar{rr} \ar{dd} \ar[dotted]{dl} &[-\perspective] &[-\perspective] \Def_{E^\alg}^\sharp \ar{rr} \ar{dd} \ar[dotted]{dl} &[\perspective] &[-\perspective] \Def_{E^\alg} \ar{dd} \ar{dl} \\[-\perspective]
	\Def_E^{\sharp,+} \ar[crossing over]{rr} \ar{dd} & & \Def_{E}^\sharp \ar[crossing over]{rr} & & \Def_E \\[\perspective]
	& h_K \ar{rr} \ar[equals]{dl} & & \Gamma  \ar{rr} \ar[equals]{dl} & &  \Def_{E^\sharp} \ar[equals]{dl} \\[-\perspective]
	h_K \ar{rr} & & \Gamma \ar{rr} \ar[from=uu,crossing over] && \Def_{E^\sharp} \ar[from=uu,crossing over]
	\end{tikzcd}\]
	The dotted map $\Def_{E^\alg}^\sharp \to \Def_E^\sharp$ is the natural one
	\[ \begin{tikzcd} \Def_{E^\alg}^\sharp = \Gamma \times_{\Def_{E^\sharp}} \Def_{E^\alg} \ar{r} & \Gamma \times_{\Def_{E^\sharp}} \Def_E = \Def_E^\sharp \end{tikzcd} \]
	induced by the equivalence $\Def_{E^\alg} \to \Def_E$, so it is an equivalence as well. Similarly the dotted map $\Def_{E^\alg}^{\sharp,+} \to \Def_E^{\sharp,+}$ must be an equivalence as well, using $h_K$ in place of $\Gamma$. 
\end{proof}

\subsection{Cohomologically rigid isocrystals}

\begin{definition} \label{cohomologically rigid isocrystal}
	We say $E \in \MC^f(\tube{U})$ is \emph{cohomologically rigid} if $H^1_p(\End(E)) = 0$. 
\end{definition}

\begin{example} \label{example:regular singularities}
	Suppose $E \in \MC^\dagger(\tube{U}) = \Isoc^\dagger(U)$ is absolutely irreducible of rank 2, admits a Frobenius structure (or, more generally, has exponents whose differences are non-Liouville), and has regular singularities along $Z$ (i.e., $\irr_z(E) = 0$ for all $z \in Z$). Let us compute $\dim H^1_p(\End(E))$, thereby finding a criterion for $E$ to be cohomologically rigid. 
	
	Fix $z \in Z$. Note that $E_z$ being regular means precisely that it is pure of slope 0. Since $E$ admits a Frobenius structure, the exponents $\textnormal{Exp}(E_z) = \{\bar{\alpha}, \bar{\beta}\}$ of $E_z$ are in $\Z_{(p)}/\Z$ \cite[th\'eor\`eme 5.5--3]{christol_mebkhout3}. If we choose representatives $\alpha, \beta \in \Z_{(p)}$ for $\bar{\alpha}, \bar{\beta} \in \Z_p/\Z$, then Christol-Mebkhout's $p$-adic Fuchs theorem \cite[th\'eor\`eme 2.2--1]{christol_mebkhout4} tells us that there is a basis of $E_z$ with respect to which the derivation $\partial$ acts by an upper triangular matrix of the form
	\[ \begin{bmatrix} \alpha & * \\ 0 & \beta \end{bmatrix}, \]
	and that if $\bar{\alpha} \neq \bar{\beta}$, then we can even take $* = 0$. We will say that $E$ \emph{has a scalar singularity at $z$} if there is a basis such that $\partial$ acts by a scalar matrix on $E$. 
	
	If $E$ has a scalar singularity at $z$, it splits into a direct sum of two isomorphic differential modules of rank 1 (both corresponding to the differential equation $t\partial - \alpha$), and then it is clear that $\End(E_z)$ must be a constant differential module over the Robba ring $\Robba_z$. In other words, we have $\artin_z(\End(E)) = 0$. 
	
	Otherwise, there are two cases. 
	\begin{itemize}
		\item If $E_z$ has two distinct exponents, then $E_z$ splits into a direct sum of two non-isomorphic differential modules over $\Robba_z$ of rank 1. It is then clear that \[ \dim H^0_{\dR}(\End(E_z)) = 2, \]
		so $\artin_z(\End(E)) = 2$.
		
		\item If not, then $E_z$ has just one exponent $\bar{\alpha}$ of multiplicity 2, but does not have a scalar singularity at $z$. We know that there is a basis $(e_1, e_2)$ such that $\partial$ acts by a matrix of the form
		\[ \begin{bmatrix} \alpha & * \\ 0 & \alpha \end{bmatrix}. \]
		Since $E_z$ does not have a scalar singularity at $z$, the function $*$ must not have an antiderivative in $\Robba_z$. If it did have an antiderivative $f \in E_z$, then $(e_1, e_2 - fe_1)$ would be a basis on which $\partial$ would act by a scalar matrix. 
		
		We can then compute that $\dim H^0_{\dR}(\End(E)) = 2$, as follows. Note that $\End(E_z)$ splits as a direct sum of the identity component and the trace-zero component $\End^0(E_z)$. Now $\End^0(E_z)$ is spanned by the endomorphisms $\sigma_e, \sigma_f, \sigma_h$ of $F^0$ given by
		\[ [\sigma_e] = \begin{bmatrix} 0 & 1 \\ 0 & 0 \end{bmatrix}, [\sigma_f] = \begin{bmatrix} 0 & 0 \\ 1 & 0 \end{bmatrix}, [\sigma_h] = \begin{bmatrix} 1 & 0 \\ 0 & -1 \end{bmatrix}. \]
		Now $H^0_{\dR}(\End^0(E_z))$ is spanned by $\sigma_e$ over $K$. To see this, suppose we have a horizontal $\sigma = u \sigma_e + v \sigma_f + w \sigma_h$ for some $u, v, w \in \Robba_z$. We compute the following.
		\begin{align*}
		\sigma \partial (e_1) &= a w e_1 + a v e_2 \\
		\partial \sigma (e_1) &= (a w + * v + \partial(w) ) e_1 + (\partial(v) + a v) e_2 \\
		\sigma \partial (e_2) &= (* w + a u)e_1 + (* v - a w) e_2 \\
		\partial \sigma (e_2) &= ( \partial(u) + a u - * w )e_1 + (-\partial(w) -a w) e_2 
		\end{align*}
		Equating the first two, we see first that we must have $\partial(v) = 0$, so $v$ is a scalar. Then we must have $\partial(w) = -v*$. Then we see that in fact we must have $v = 0$, since otherwise $-w/v$ would be an antiderivative of $*$. This then forces $w$ to be scalar. Then, equating the second two equations above, we see that $\partial(u) = 2w*$. Again this forces $w = 0$, since otherwise $*$ would have an antiderivative, and further it must be that $u$ is scalar. This proves that $H^0_\dR(\End^0(E_z))$ is spanned by $\sigma_e$ over $K$. Thus $H^0_{\dR}(\End(E_z))$ is spanned by $1$ and $\sigma_e$, proving that
		\[ \dim H^0_{\dR}(\End(E_z)) = 2, \]
		whence $\artin_z(\End(E)) = 2$.  
	\end{itemize}
	Since $\End(E)$ is self-dual via the trace pairing of \cref{example:trace pairing}, we know that $H^0_{\dR}(\End(E))$ and $H^2_c(\End(E))$ are dual by Crew's finiteness theorem \ref{result:crew finiteness}. Moreover, since $E$ is absolutely irreducible, we know that $\dim H^0_{\dR}(\End(E)) = 1$. Thus \cref{result:dimension parabolic} tells us that
	\[ \dim H^1_p(\End(E)) = -6 + \sum_{z \in Z} \artin_z(\End(E)) = 2(m-3) \]
	where $m$ is the number of points $z \in Z$ such that $E$ has non-scalar singularities at $z$.  
	
	In other words, if all of the singularities of $E$ are non-scalar, then $E$ is cohomologically rigid if and only if $\# Z = 3$. This is analogous to what we saw in \cref{result:irreducible rank 2}. 
\end{example}

	\appendix
	\section{Isolated points of algebraic stacks}

\newcommand{\scrX}{\mathscr{X}}
\newcommand{\ftpts}{\textnormal{ft-pts}}
\newcommand{\red}{\textnormal{red}}

Throughout, let $\scrX$ be an algebraic stack \tag{026O}. We make the following definition. 

\begin{definition}
	A point $x \in |\scrX|$ is \emph{isolated} if $\{x\}$ is open and closed in $|\scrX|$.\footnote{If $|\scrX|$ is locally connected, then this is equivalent to requiring that $\{x\}$ is a connected component of $|\scrX|$. This is implied, for instance, by the condition that $\scrX$ be locally noetherian \cite[\taglink{04MF, 0DQI}]{stacks}.}
\end{definition}

\begin{lemma}
	Any isolated point of $\scrX$ must be a point of finite type.
\end{lemma}

\begin{proof}
	If $x$ is an isolated point, then $\{x\} \cap \scrX_\ftpts$ must be nonempty \tag{06G2}.  
\end{proof}

\begin{lemma} \label{result:unique ftpt}
	Suppose $\scrX$ has a unique finite type point $x$. Then $|\scrX| = \{x\}$. In particular, $x$ is an isolated point of $\scrX$.
\end{lemma}

\begin{proof}
	Suppose $f : U \to \scrX$ is any smooth map with $U$ a nonempty scheme. The image of $U$ is a nonempty open subset of $|\scrX|$ \tag{04XL},  so it must contain $x$ \tag{06G2}. Moreover, the complement is a closed subset of $|\scrX|$ containing no finite type points, so the complement must be empty. In other words, $f$ must be surjective. This means that we can replace $U$ with a nonempty affine open subscheme and $f$ will still be surjective. 
	
	Let $\Gamma_x$ be the residual gerbe at the unique finite type point $x$ \tag{06G3}. Then the inclusion $\Gamma_x \hookrightarrow \scrX$ is a locally of finite type monomorphism, so its pullback $R_x := \Gamma_x \times_{\scrX} U \to U$ is a locally finite type monomorphism of algebraic spaces. This must be representable \tag{0418}. In other words, $R_x$ is actually a scheme, and we have a cartesian diagram as follows. 
	\[ \begin{tikzcd} R_x \ar[hookrightarrow]{r} \ar{r} \ar[twoheadrightarrow]{d} & U \ar[twoheadrightarrow]{d}{f} \\ \Gamma_x \ar[hookrightarrow]{r} & \scrX \end{tikzcd} \]
	Since $|\Gamma_x| = \{x\}$, clearly it is sufficient to show that $R_x \to U$ is surjective. 
	
	Chevalley's theorem \tag{054K} guarantees that the image of $R_x  \to U$ is a locally constructible subset of $U$, but $U$ is affine so in fact the image must be constructible \tag{054C}. Then the complement $Z$ of the image is also constructible \tag{005H}. In particular, $Z$ is a finite union of locally closed subsets. Thus $Z \cap U_\ftpts$ must be dense in $Z$. If $Z$ were nonempty, there would have to exist some $u \in Z \cap U_{\ftpts}$. But then $f(u)$ would have to be a finite type point of $\scrX$ \tag{06G0}, and it could not equal $x$ since $u$ is constructed to not be in the image of $R_x \to U$. This contradicts the assumption that $\scrX_\ftpts = \{x\}$. Thus $Z$ must be empty, proving that $R_x \to U$ is surjective. 
\end{proof}

\begin{lemma} \label{result:isolated implies nonpositive dimension}
	Suppose $\scrX$ is quasi-separated and locally of finite type over a field $k$, and $x \in |\scrX|$ is isolated. Then $\dim_x(\scrX) \leq 0$. 
\end{lemma}

\begin{proof}
	Notice first that if $\mathscr{U}$ is the open substack corresponding to the open subset $\{x\} \subseteq |\scrX|$ \tag{06FJ}, then
	\[ \dim_x(\scrX) = \dim_x(\mathscr{U}). \]
	In other words, by replacing $\scrX$ with $\mathscr{U}$ if necessary, we can assume that $|\scrX| = \{x\}$. 
	
	Now if $f : U \to \scrX$ is any smooth map with $U$ a nonempty locally noetherian scheme, then clearly $f$ must be surjective. Thus we may assume that $U$ is affine and of finite type over $k$. Then $U$ has finitely many irreducible components $Z_1, \dotsc, Z_n$, and we can replace $U$ with the nonempty open subset $U \setminus (Z_2 \cup \dotsb \cup Z_n)$ in order to assume that $U$ is also irreducible. 
	
	Since $\scrX$ is quasi-separated, $R := U \times_{\scrX} U$ is a finite type algebraic space over $k$. Moreover, we have a smooth groupoid in algebraic spaces 
	\[ (U, R, s, t, c, e, i)  \]
	and an equivalence $\scrX = [U/R]$ \tag{04T5}. For every $u \in U$, let $T_u$ be the connected component of $R_u := s^{-1}(u)$ containing $e(u)$. Since $R_u$ is smooth over $\kappa(u)$, the connected component $T_u$ is also an irreducible component of $R_u$. 
	
	The image of $T_u$ under $t$ is therefore an irreducible subset $O(u)$ of $U$ which contains $u$. It is constructible by Chevalley's theorem \tag{0ECX}. As $u$ varies, we obtain a partition of $U$ into irreducible constructible subsets of the form $O(u)$. 
	
	If $\eta \in U$ is the generic point, then $O(\eta)$ is a constructible subset of $U$ containing $\eta$, so in fact $O(\eta)$ contains a dense open subset of $U$ \tag{005K}. Then there exists $v  \in \O(\eta) \cap U_{\ftpts}$ \tag{02J4}. Moreover, since $v \in O(\eta) \cap O(v)$, we must have $O(\eta) = O(v)$. 
	
	In other words, $T_v \to U$ is a dominant morphism of algebraic spaces of finite type over $k$. Thus $\dim(U) \leq \dim(T_v)$ \cite[IV\textsubscript{2}, corollaire 2.3.5(i)]{ega}. Since $f(v) = x$, we see from the definition of dimension \tag{0AFN} that
	\[ \dim_x(\scrX) = \dim_v(U) - \dim_{e(v)}(R_u) = \dim(U) - \dim(T_v) \leq 0. \qedhere \] 
\end{proof}
	
	
	\bibliographystyle{alphaurl}
	\bibliography{references}
	
\end{document}